\documentclass{article}
\usepackage{amsthm,amsmath,amssymb,mathrsfs,amsfonts,xcolor,wasysym,bbm}
\usepackage{booktabs}
\usepackage{xcolor}
\usepackage{enumitem}
\usepackage{graphicx}
\usepackage{caption}
\usepackage{bm}
\usepackage{amsfonts}

\usepackage{algorithm}
\usepackage{algpseudocode}
\usepackage{algorithmicx}
\usepackage{booktabs}
\usepackage{array}
\usepackage{multirow}
\usepackage[colorlinks=true, linkcolor=blue, citecolor=blue, urlcolor=blue]{hyperref}

\usepackage{geometry}
\RequirePackage[round,authoryear]{natbib}
\usepackage[most]{tcolorbox}
\usepackage{authblk}
\numberwithin{equation}{section}
\newtheorem{theorem}{Theorem}[section]
\newtheorem{definition}{Definition}[section]
\newtheorem{lemma}{Lemma}[section]

\newtheorem{remark}{Remark}[section]

\newtheorem{assumption}{Assumption}[section]

\begin{document}

\title{\textbf{The logarithmic law of sample correlation matrices}}

\author[1]{Yanpeng Li\thanks{Email: \texttt{20230256@hit.edu.cn}}}
\author[2]{Zhi Liu\thanks{Email: \texttt{liuzhi@um.edu.mo}}}
\author[3]{Jiahui Xie\thanks{Email: \texttt{jh.xie@nus.edu.sg}}}
\author[3]{Wang Zhou\thanks{Email: \texttt{wangzhou@nus.edu.sg}}}
\affil[1]{\small School of Mathematics, Harbin Institute of Technology, China}
\affil[2]{\small Faculty of Science and Technology, University of Macau, China}
\affil[3]{\small Department of Statistics and Data Science, National University of Singapore, Singapore}

\date{}

\maketitle
\vspace{-3em}

\begin{abstract}
    Let $\mathbf{R}$ be the sample correlation matrix constructed from $\mathbf{X}\in \mathbb{R}^{p\times n}$, whose entries are independent and identically distributed random variables with mean zero and tail probability condition $\lim_{x\rightarrow \infty}x^3\mathbb{P}(|\xi|>x)=0$. We derive the universal logarithmic law for $\log \det \mathbf{R}$,
    \begin{equation*}
    \frac{\log \det \mathbf{R}-(p-n+1/2)\log (1-\frac{p-1}{n})+p-\frac{p}{n}}{\sqrt{-2\log (1-\frac{p-1}{n})-2\frac{p}{n}}}\stackrel{d}{\rightarrow} {N}(0,1),
    \end{equation*}
    if $p\le n$ as $p,n\rightarrow \infty$. Moreover, under the near-singularity case $0\le n-p\le n^{1-w}$ for any $w\in (0,1)$, it is shown that the tail probability condition can be weakened to $\lim_{x\rightarrow \infty}x^3(\log x)^{-1/4+\mathfrak{c}}\mathbb{P}(|\xi|>x)<\infty$ for any constant $0<\mathfrak{c}<1/4$.
\end{abstract}

\section{Introduction}\label{sec_introduction}
As determinants are among the most fundamental matrix functions, it is a basic problem in the theory of random matrices to study the distribution of random determinants, and this study has a long and rich history. The investigations of random determinants trace back to \cite{szekeres1937extremal} in 1937, who studied the sum of squares and the sum of the fourth powers of all determinants for random Bernoulli matrices (whose entries are $-1$ and $1$ with the same probability). Later in the 1950s, there were a series of papers (see \cite{dembo1989random,nyquist1954distribution} and references therein) devoted to the computation of moments of fixed orders of random matrices. The central limit theorem (CLT) for the logarithmic determinant of random Gaussian matrices,
Wigner matrices and matrices with real independent and identically distributed (i.i.d.) entries with sub-exponential tails were proved by \cite{goodman1963distribution}, \cite{tao2011random}, and \cite{nguyen2014random}, respectively. Specifically,  \cite{goodman1963distribution} considered random Gaussian matrices whose entries are i.i.d. standard Gaussian variables. He noticed that in this case, the determinant is a product of independent Chi-square variables. Therefore, its logarithm
is the sum of independent variables and, thus, one expects a central limit theorem to hold. He further showed that
\begin{equation}\label{eq_defloglaw}
    \frac{\log |\det G_n|-(1/2)\log (n-1)!}{\sqrt{(1/2)\log n}}\stackrel{d}{\rightarrow}N(0,1),
\end{equation}
as $n\rightarrow\infty$, where $G_n=(G_{ij})_{1\le i\le n,1\le j\le n}$ with $G_{ij}$s' being i.i.d. standard normal random variables. Later,  \cite{girko1978central,girko1998refinement} stated and partially proved that the CLT result \eqref{eq_defloglaw} (also named ``the logarithmic law'' of random determinant) holds for random matrices in a very generic case under the existence of the $(4+\delta)$ moment for some small $\delta>0$ by using an elegant method of perpendiculars and the classical CLT for martingale. However, though the proof route of \cite{girko1978central} is clear and quite original, the proof is not complete, and several parts lack mathematical rigour, which was completed by \cite{bao2015logarithmic} and extended to the sharp finite fourth moment condition.

Consider a random sample $\mathbf{x}_1,\ldots,\mathbf{x}_n$ from a $p$-dimensional distribution collected into a $p\times n$ random data matrix $\mathbf{X}$. For statistical applications the logarithmic determinants of the sample covariance matrix $\mathbf{S} = n^{-1}\mathbf{X}\mathbf{X}^{\top}$ and the sample correlation matrix $\mathbf{R}=\{\operatorname{diag}\mathbf{S}\}^{-1/2}\mathbf{S}\{\operatorname{diag}\mathbf{S}\}^{-1/2}$ are of real importance, such as the volume of random parallelotopes in random geometry, the likelihood ratio statistics in multivariate statistics, the difference of the log determinants of two sample covariance matrices in quadratic discriminant analysis for classification and so on, see \cite{bai2010spectral,heiny2023logdet} and references therein. The logarithmic law of large sample covariance matrices can be deduced from the work of \cite{bai2004clt} for the linear spectral statistics (LSS) $\operatorname{tr}(f(\mathbf{S}))$ with a test function $f(x) = \log(x)$ when $p/n\rightarrow \gamma\in (0,1)$ as $n\rightarrow \infty$, which was further extended to near-singularity cases $p/n\rightarrow 1$ and $p\le n$ under finite fourth moment by \cite{bao2015logarithmic,wang2018}.
The first generic result for $\log \det \mathbf{R}$ was proved in \cite{gao2014high} under the fourth moment condition, which gives
\begin{equation}\label{eq_logdetR}
    \frac{\log \det \mathbf{R}-(p-n+1/2)\log (1-(p-1)/n)+p-\frac{p}{n}}{\sqrt{-2\log (1-(p-1)/n)-2p/n}}\stackrel{d}{\rightarrow} N(0,1),
\end{equation}
for $p/n\rightarrow \gamma\in (0,1)$ as $n\rightarrow\infty$. Surprisingly, the fourth moment is
not present in \eqref{eq_logdetR}, which was further extended to a weaker moment condition by \cite{heiny2023logdet}, who proved that \eqref{eq_logdetR} holds given that $X_{ij}$ are symmetric i.i.d., having regularly varying tails with index $\alpha\in (3,4)$. Thus, one may ask \textit{whether the logarithmic law of the determinant of $\mathbf{R}$ holds under weaker moment conditions regardless of the singularity}.

In this paper, we provide an affirmative answer to this question and establish the universal logarithmic law of $\det\mathbf{R}$ under condition \eqref{eq_oldcondition}. Specifically, under the non-singularity case $p/n\in (0,1)$, \eqref{eq_oldcondition} coincides with the necessary and sufficient condition for CLT of LSS of $\mathbf{R}$ in \cite{Li2024clt}. While in the near-singularity case $0\le n-p\le n^{1-w}$ for any $w\in (0,1)$, the divergent asymptotic variance due to singularity relaxes the constraints of the moments, we provide a weakened condition \eqref{eq_newcondition}, which we believe to be near optimal.

Consider a $p$-dimensional population $\mathbf{x}=(X_1,X_2,\ldots,X_p)^{\top}\in \mathbb{R}^{p}$, where the coordinates $X_i$ for $1\le i\le p$ are i.i.d. non-degenerate random variables with the same law of $\xi$ as follows.
\begin{assumption}\label{assump1}
	Assume that $\xi$ has a regularly varying tail with index $\alpha>0$, that is
	\begin{equation*}
	\mathbb{P}(|\xi|>x)=\mathcal{L}(x)x^{-\alpha},~\text{as}~x\rightarrow \infty,
	\end{equation*}
	where $\mathcal{L}(x)$ is a ``slowly varying function'' such that $\lim_{x\rightarrow \infty}\frac{\mathcal{L}(tx)}{\mathcal{L}(x)}=1$
	for all $t\in \mathbb{R}$. Moreover, we assume that $\mathbb{E}\xi=0$ and $\mathbb{E}\xi^2=1$ if $\alpha>2$.
\end{assumption}
We now state the universal central limit theorem for the log-determinant of the sample correlation matrix $\mathbf{R}$, which constitutes the main result of this paper.
\begin{theorem}\label{thm_logdet}
	Under
	\begin{equation}\label{eq_oldcondition}
	\lim_{x\rightarrow \infty}x^3\mathbb{P}(|\xi|>x)=0,
	\end{equation}
	the log determinant of the sample correlation matrix $\mathbf{R}$ satisfies
	\begin{equation}\label{eq_universalclt}
	\frac{\log \det \mathbf{R}-(p-n+1/2)\log (1-\frac{p-1}{n})+p-\frac{p}{n}}{\sqrt{-2\log (1-\frac{p-1}{n})-2\frac{p}{n}}}\stackrel{d}{\rightarrow} {N}(0,1)
	\end{equation}
	when $p/n\rightarrow \gamma\in (0,1]$ as $n,p\rightarrow \infty$ and $p\le n$.
\end{theorem}
\begin{remark}
	Note that \eqref{eq_oldcondition} holds automatically under Assumption \ref{assump1} for $\alpha>3$. In the sequel, we often use this assumption equipped with this precise tail probability condition. For the case of $p/n\rightarrow \gamma\in (0,1)$, Theorem \ref{thm_logdet} shows that the symmetry assumption and $\mathbb{E}|\xi|^3<\infty$ of Theorem 3.1 of \cite{heiny2023logdet} are redundant.
\end{remark}
Moreover, in the near-singularity case $0\le n-p\le n^{1-w}$ for any constant $w\in (0,1)$, we have
\begin{theorem}\label{thm_pn}
	Suppose
	\begin{equation}\label{eq_newcondition}
	\limsup_{x\rightarrow \infty}\frac{x^3\mathbb{P}(|\xi|>x)}{(\log x)^{1/4-\mathfrak{c}}}<\infty
	\end{equation}
	for any constant $0<\mathfrak{c}<1/4$.
	Then, if $0\le n-p\le n^{1-w}$ for any constant $w\in (0,1)$, we have
	\begin{equation}\label{eq_universalcltpn}
	\frac{\log \det \mathbf{R}-(p-n+1/2)\log (1-\frac{p-1}{n})+p}{\sqrt{-2\log (1-\frac{p-1}{n})}}\stackrel{d}{\rightarrow} {N}(0,1).
	\end{equation}
	In addition, for $p=n$, we have
	\begin{equation*}
	\frac{\log\det \mathbf{R}+\frac{1}{2}\log n+n}{\sqrt{2\log n}}\stackrel{d}{\rightarrow}N(0,1).
	\end{equation*}
\end{theorem}
\begin{remark}\label{remark_differentcondition}
	The condition \eqref{eq_newcondition} is used when justifying the martingale central limit theorem, see \eqref{eq_sharp1} and \eqref{eq_sharp2} for details. We believe that \eqref{eq_newcondition} is at least near optimal in two aspects. On the one hand, if \eqref{eq_newcondition} fails, then one can not use the martingale method to show the asymptotic normality since both the second-order term and error term in expansion \eqref{eq_expansion} are non-negligible. On the other hand, to implement the multilevel block truncations (Section \ref{sec_multitruncation}) and hierarchical Gaussian repalcement (Section \ref{sec_gaussianreplace}), the nearly sharp bound $\log^{1/4-\mathfrak{c}}n$ is essential to guarantee that one can control the conditional variance $\operatorname{Var}(\Delta_{i+1}^2\mid \mathcal{F}_i)$ well (Section \ref{sec_bulkgaussianreplace}), which is also nearly sharp due to the edge replacement in Section \ref{sec_edgegaussianreplace}.
\end{remark}

Our proof strategies rely on the method of perpendiculars and the classical CLT for martingale developed by \cite{girko1998refinement} and \cite{bao2015logarithmic,wang2018}. To adapt the sample correlation matrix setting under infinite fourth moment and near singularity, we also deploy the mixed moments of self-normalization entries (Lemma \ref{lemma_odd_moment}), the (nearly) sharp upper bounds and concentrations for the diagonals of large projection matrices (Lemma \ref{lemma_upperdiagonal} and Lemma \ref{lemma_diagonals}, respectively), and the convergence rate of the expected spectral distribution (Lemma \ref{lemma_esdconvergence}). To overcome the asymmetry, we consider a centralized version of the expansion \eqref{eq_decomZ}, which does not affect limiting distribution for $\alpha\ge 3$ due to an improved estimation for the mixed moment of $\mathbb{E}(Y_{11}Y_{12})$ by Lemma \ref{lemma_odd_moment}, which is of its own interest contrast to the classical estimation $\mathbb{E}(Y_{11}Y_{12})=\mathrm{o}(n^{-2})$ for general $\alpha>2$. Moreover, to dispose of the heavy-tailedness and singularity under the cases of divergent slowly varying functions \eqref{eq_newcondition}, we propose multilevel block truncations (Section \ref{sec_multitruncation}) and hierarchical Gaussian replacement (Section \ref{sec_gaussianreplace}), which are brand new and shed insights into the context of heavy-tailed random matrices.
In addition, we deploy the resampling techniques to assist the analysis of the projection matrices, which provides powerful tools to handle the heavy-tailed condition as witnessed by \cite{aggarwal2018goe,bao2023phase,Li2024clt}. The rest of the paper is organized as follows. In Section \ref{sec_main2}, we outline the proof sketch. Section \ref{sec_proof} consists of the detailed proofs of our theoretical results. In the Appendix, we present the auxiliary lemmas and their proofs.

\textbf{Notation}. Throughout this article, we set $C,c>0$ to be constants whose value may be different from line to line in the request. We write $[n]:=\{1,\ldots,n\}$ for $n\in \mathbb{Z}_+$. For any value $a,b\in\mathbb{R}$, $a\land b=\min(a,b)$ and $a\lor b=\max(a,b)$. For a matrix $A=(A_{ij})\in\mathbb{R}^{p\times n}$, $\operatorname{tr}A$ denotes the trace of $A$, $\|A\|$ denotes the spectral norm of $A$ equal to the largest singular value of $A$.
For two sequences of numbers $\{a_{n}\}_{n=1}^{\infty}$, $\{b_{n}\}_{n=1}^{\infty}$, $a_{n}\asymp b_{n}$ if there exist
constants $C_{1},C_{2}>0$ such that $C_{1}|b_{n}|\le|a_{n}|\le C_{2}|b_{n}|$, and we denote $a_n=\Omega(b_n)$ for convenience.
We denote $a_n\lesssim b_n$ if there exist constant $C_1>0$ such that $|a_n|\le C_1|b_n|$. Also, we write $a_n\lesssim b_n$ as $b_n\gtrsim a_n$. We say $\mathrm{O}(a_{n})$ and $\mathrm{o}(a_{n})$ in the sense that $|\mathrm{O}(a_{n})/a_{n}|\le C$ with some constant $C>0$ for all large $n$ and $\lim_{n\to\infty}\mathrm{o}(a_{n})/a_{n}=0$. Specifically, we use $\mathrm{O}_{\prec}(a_n)$ and $\mathrm{o}_{\prec}(a_n)$ in the sense that, with high probability, $|\mathrm{O}_{\prec}(a_n)/a_n|\le C$ for all large $n$ and $\lim_{n\rightarrow\infty}|\mathrm{o}_{\prec}(a_n)/a_n|=0$, repestively.
We also use $\mathrm{O}_{\mathbb{P}}$ and $\mathrm{o}_{\mathbb{P}}$ to denote the big and small $\mathrm{O}$ notation in probability. When we write $a_n\ll b_n$ and $a_n\gg b_n$, we mean $|a_n|/b_n\rightarrow0$ and $|a_n|/b_n\rightarrow\infty$ respectively when $n\rightarrow\infty$. For a set $A$, $A^{c}$ denotes its complement (with respect to some whole set which is clear in the context). For any event $\Xi$, $\mathbbm{1}(\Xi)$ denotes the indicator of the event $\Xi$, equal to $1$ if $\Xi$ occurs and $0$ if $\Xi$ does not occur. For any number $x>0$, let $\lceil x \rceil$ represent the least integer not less than $x$, and $\lfloor x\rfloor$ denote the largest integer not larger than $x$.

\section{Outline of the proof}\label{sec_main2}
Let $\mathbf{x}_1,\ldots,\mathbf{x}_n$ be independent observations from the population. Define the data matrix
\begin{equation*}
    \mathbf{X}=\mathbf{X}_n=(\mathbf{x}_1,\ldots,\mathbf{x}_n)=(X_{ij})_{1\le i\le p,1\le j\le n}\in \mathbb{R}^{p\times n}.
\end{equation*}
Then the sample covariance matrix and the sample correlation matrix are given by $ \mathbf{S}=\mathbf{S}_n=n^{-1}\mathbf{X}\mathbf{X}^{\top}$ and
\begin{equation*}
    \mathbf{R}=\mathbf{R}_n=\{\operatorname{diag}(\mathbf{S}_n)\}^{-1/2}\mathbf{S}_n\{\operatorname{diag}(\mathbf{S}_n)\}^{-1/2}=:\mathbf{Y}\mathbf{Y}^{\top},
\end{equation*}
respectively, where the entries of the self-normalized matrix $\mathbf{Y}=\mathbf{Y}_n=(Y_{ij})_{1\le i\le p,1\le j\le n}$ for the sample correlation matrix take the form as
\begin{equation}\label{eq_selfentry}
    Y_{ij}=\frac{X_{ij}}{\sqrt{X_{i1}^2+\cdots+X_{in}^2}}.
\end{equation}
Throughout the paper, we suppress the explicit dependence on $n$ in our notation. We assume that the dimension $p = p_n$ grows with the sample size $n$ such that $p/n\rightarrow \gamma\in(0,1]$. This setting includes the nearly singular case $\gamma=1$, for which the smallest eigenvalue $\lambda_{p}(\mathbf{S})$ converges to zero as $p\rightarrow\infty$.

The proofs of Theorems \ref{thm_logdet} and \ref{thm_pn} rely on the resampling strategy \citep{aggarwal2018goe,bao2023phase}, Gaussian replacement \citep{bao2015logarithmic,nguyen2014random,wang2018},
novel bounds for projection matrices \citep{heiny2023logdet}, and convergence rate of expected spectral distributions \citep{bao2015logarithmic}. Moreover, to address the singularity, we propose multilevel block truncations to control variance estimation (see Section \ref{sec_multitruncation} for details).
At a high level, compared with the previous works, to overcome the barriers from the heavy-tailed condition and singularity, our proof strategy involves adapting the approachs from \cite{heiny2023logdet} for non-singularity case of sample correlation matrices, \cite{bao2015logarithmic,nguyen2014random,wang2018} for near-singularity case of sample covariance matrices, and \cite{aggarwal2018goe,bao2023phase} for resampling technique. However, this adaptation is far from being straightforward. We summarize major ideas as follows.
\begin{itemize}
	\item[(i)] \textit{Multilevel block truncations}. Similarly to many previous works on RMT under finite fourth moments, one can work on a truncated data matrix without affecting the limiting properties, whose entries can have appropriate high-order moment bounds. However, under the heavy-tailed case $\alpha<4$, the classical truncation level at $n^{1/2}$ is not enough since there are about $n^{2-\alpha/2}$ entries larger than $n^{1/2}$ among $n^2$ entries of $\mathbf{X}$ (see Appendix \ref{app_pre} for more details).
	Moreover, due to singularity, a uniform crude truncation bound is far from canceling the rate from the heavy-tailedness when $p$ is close to $n$. To tackle the barriers from singularity, heavy-tailedness, and the divergent slowly varying function case \eqref{eq_newcondition}, we need multilevel block truncations which consist of two parts, and handle each block carefully to derive good enough estimations block by block. At a global scale for the first $p-\mathrm{o}(p)$ rows, we can use a unified truncation level since there is no need to consider the effect of singularity. At a local scale for the last $\mathrm{o}(p)$ rows which are divided into $(K+1)$ blocks, we conduct two-level truncations on each block, which provide improved bounds for the fourth moment of the self-normalized entries after this truncation and thus can cancel the divergent rate caused by heavy-tailedness, singularity and the divergent slowly varying funcation.
\item[(ii)] \textit{Hierarchical Gaussian replacement}. Similar to many previous works on the logarithmic law, Girko's method of perpendiculars and the martingale method rely on a Taylor expansion, which can not be guaranteed near singularity (say, $n-i=\mathrm{O}(1)$). To handle the large $i$ parts, \cite{girko1998refinement} proposed to replace the last $\log^{c}n$ rows by Gaussian ones and proved the logarithmic law for the matrix after replacement, and then recovered the result to the original one by a comparison procedure. Such a strategy was also used in  \cite{bao2015logarithmic,nguyen2014random,wang2018}, where they replaced the last $\log^{c}n$ rows with standard Gaussian ones line by line and then controlled the difference carefully to recover the logarithmic law for Wigner matrices and sample covariance matrices. However, for sample correlation matrices under heavy-tailed conditions, such a Taylor expansion is no longer valid if $n-p=n^{c}$ for some constant $0<c<1$, since one can not control the difference after replacing lines one by one due to error accumulation. Thus, we blaze a new trail and aim to control the overall difference for Gaussian replacement, which is achieved by distinguishing two different cases according to the value of $n-p$. On the one hand, for $n-p\ge \log^c n$, the comparison procedure boils down to controlling the variance for the products of a series of dependent random variables, where we use mathematical induction to derive the desired bounds for the variance. In contrast to the Gaussian part, the variance estimation is much more involved due to the heavy-tailed condition, which needs meticulous handling for different levels of $n-i$, depending on
(nearly) sharp bounds for the diagonals of the projection matrices (Lemma \ref{lemma_upperdiagonal}) and novel multilevel block truncations above (Section \ref{sec_multitruncation}). This is pivotal to control the variance of the $U$-part for $\alpha\ge 3$. On the other hand, for $n-p\le \log^c n$, we use mathematical induction following the ideas of \cite{wang2018} and \cite{nguyen2014random} with modifications to fit our sample correlation matrix setting.
\item[(iii)]\textit{Novel bounds for large projection matrices}. According to \cite{heiny2023logdet} for the non-singular case, the martingale method depends on the variance estimation of ${Z}_{i+1}$ (see \eqref{eq_decomZ0} below), where the main contribution comes from the off-diagonal parts $V_{i+1}$. The variance estimation for $U$-parts (i.e., \eqref{eq_meanU2}) relies on $\mathbb{E}Y_{ij}^4$ and the concentration for the diagonals of projection matrices $\mathbf{P}_i$ (see Lemma \ref{lemma_diagonals} for details). Under the heavy-tailed case $\alpha\in (3,4)$, one needs fast rate $\mathrm{O}(n^{-1/2})$ to cancel the factor in $\mathbb{E}Y_{ij}^4\sim \mathcal{L}(n^{1/2})n^{-\alpha/2}$. However, under near-singularity, the concentration result for the diagonals of projection matrices $\mathbf{P}_i$ fails since the least eigenvalue of the sample covariance matrix tends to zero.
Unlike \cite{heiny2023logdet}, we estimate $\mathbb{E}p_{i,kk}^2$ directly by the resampling technique from \cite{aggarwal2018goe,bao2023phase} and some innovative high probability estimations for the expected spectral distribution $\mathbb{E}p^{-1}\operatorname{tr}(n^{-1}\mathbf{X}\mathbf{X}^{\top}+\epsilon_{n}\mathbf{I}_n)^{-1}$ below, where $\epsilon_{n}\rightarrow 0$.
Fortunately, for large $i$ parts, the Gaussian replacement above improves the estimation rate for $\mathbb{E}\check{Y}_{ij}^4$, which together with some novel (nearly) sharp bounds for the maximum of the diagonals of $\mathbf{P}_i$ gives the desired result (see Lemma \ref{lemma_upperdiagonal} for details). The (nearly) sharp upper bounds for the maximum of the diagonals of $\mathbf{P}_i$ require a (nearly) sharp estimation for expected spectral distribution $\mathbb{E}p^{-1}\operatorname{tr}(n^{-1}\mathbf{X}\mathbf{X}^{\top}+\epsilon_{n}\mathbf{I}_n)^{-1}$ again. While for the $V$-parts, it is crucial to estimate $\mathbb{E}V_{i+1}^4$ to verify the condition for martingale CLT. There are many extra terms for $\mathbb{E}V_{i+1}^4$ without the symmetry condition, which need novel bounds for the projection matrices (Lemma \ref{lemma_Qbounds}) and the mixed moments of self-normalized entries (Lemma \ref{lemma_odd_moment}).
\item[(iv)]\textit{Improved convergence rate of expected spectral distribution}. Similarly to \cite{bao2015logarithmic,wang2018}, by the Sherman-Morrison formula, one can write
$p_{i,kk}\le (1+n^{-1}\mathbf{v}_{k,i}^{\top}(n^{-1}\mathbf{B}_{(i,k)}\mathbf{B}_{(i,k)}^{\top}+\epsilon_n\mathbf{I}_i)^{-1}\mathbf{v}_{k,i})^{-1}$. Intuitively, we expect that $p_{i,kk}\approx \mathbb{E}p_{i,kk}=(n-i)/n$. But, to the best of our knowledge, there is no idea to derive it for large $i$ since there is no available variance estimation for $p_{i,kk}$ due to near singularity. As demonstrated by \cite{bao2015logarithmic,wang2018}, since $p_{i,kk}$ are identically distributed but dependent due to the sum identity $\sum_{k=1}^{n}p_{i,kk}=n-i$, one may expect that the maximum can not hit large values with probability tending to one. With this starting point, \cite{bao2015logarithmic} stated that $\mathbb{E}\max_{1\le k\le n}p_{i,kk}=\mathrm{O}(\log^{-C_1}p)$ and \cite{wang2018} showed that $\mathbb{P}(\max_{1\le k\le n}p_{i,kk}\ge \log^{-C_1}p)=\mathrm{O}(n^{-1/3})$ for large $i\ge p-p/\log^{C_2} p$, where $C_1$ and $C_2$ are some positive constants. However, those bounds are far away from the ideal point since we need a fast rate to cancel the effect of the heavy-tailed condition.
To establish (nearly) sharp upper bounds for the maximum of the diagonals of $\mathbf{P}_i$ for large $i$, we need appropriate estimation for the convergence rate of expected spectral distribution $\mathbb{E}p^{-1}\operatorname{tr}(n^{-1}\mathbf{X}\mathbf{X}^{\top}+\epsilon_{n}\mathbf{I}_n)^{-1}$ (Lemma \ref{lemma_esdconvergence}). Specifically, we deploy the resampling technique, which is powerful to tackle the heavy-tailed condition as illustrated by \cite{aggarwal2018goe,bao2023phase,Li2024clt}, to obtain a crude estimation first. Then we invoke the bootstrap strategy in local law \cite{erdHos2017dynamical} to get a near-optimal convergence rate for large $i$. Then, with the intuition that $\max_{1\le k\le n}p_{i,kk}\approx [\mathbb{E}p^{-1}\operatorname{tr}(n^{-1}\mathbf{B}_{(i,k)}\mathbf{B}_{(i,k)}^{\top}+\epsilon_{n}\mathbf{I}_i)^{-1}]^{-1}$ with probability tending to one, we finally derive nearly sharp bounds for $\max_{1\le k\le n}p_{i,kk}$.
\end{itemize}

As it is quite long and technical, the proof will be divided into several parts.

Firstly, using the method of perpendiculars (see \cite{bao2015logarithmic,girko1978central,heiny2023logdet} for more details), we can write the determinant of the sample correlation matrix $\mathbf{R}=\mathbf{Y}\mathbf{Y}^{\top}$ as
\begin{equation}\label{eq_defdetR}
    \mathrm{det}(\mathbf{R})=\prod_{i=0}^{p-1}\mathbf{y}_{i+1}^{\top}\mathbf{P}_i\mathbf{y}_{i+1}=:\prod_{i=0}^{p-1}\Delta_{i+1}^2,
\end{equation}
where $\mathbf{y}_{i+1}^{\top}=(Y_{i+1,1},\ldots,Y_{i+1,n})$ denotes the $(i+1)$-st row of $\mathbf{Y}$. As illustrated by \cite{bao2015logarithmic,nguyen2014random}, $\Delta_{i+1}^2=\mathbf{y}_{i+1}^{\top}\mathbf{P}_i\mathbf{y}_{i+1}$ denotes the distance from $\mathbf{y}_{i+1}$ to the subspace generated by the first $i$ rows of $\mathbf{X}$.
The projection matrix $\mathbf{P}_i=\mathbf{I}_n-\mathbf{B}_i^{\top}(\mathbf{B}_i\mathbf{B}_i^{\top})^{-1}\mathbf{B}_i$ is the same as that of sample covariance matrix cases due to the scaled invariance of correlation matrix as illustrated by \cite{heiny2023logdet}, where the matrix $\mathbf{B}_i$ is defined by the first $i$ rows of the data matrix $\mathbf{X}$, that is, $\mathbf{B}_i=(\mathbf{b}_1,\cdots,\mathbf{b}_i)^{\top}$ with $\mathbf{b}_i=(X_{i1},\cdots,X_{in})^{\top}$. It is obvious that $\mathbf{y}_{i+1}$ is independent of $\mathbf{P}_i$, which suggests that
\begin{equation*}
    \mathbb{E}\Delta_{i+1}^2=\mathbb{E}\mathbf{y}_{i+1}^{\top}\mathbf{P}_i\mathbf{y}_{i+1}\approx \frac{1}{n}\mathbb{E}\operatorname{tr}\mathbf{P}_i= \frac{n-i}{n}.
\end{equation*}
Following the decomposition in \cite{bao2015logarithmic,heiny2023logdet}, we represent the log determinant of the sample correlation matrix $\mathbf{R}=\mathbf{Y}\mathbf{Y}^{\top}$ as
\begin{equation}\label{eq_decomlogdetR}
    \log \mathrm{det}(\mathbf{R})= \sum_{i=0}^{p-1}\log \Delta_{i+1}^2=\sum_{i=0}^{p-1}\log\frac{n-i}{n}+\sum_{i=0}^{p-1}\log \frac{n\Delta_{i+1}^2}{n-i}=:c_n+\sum_{i=0}^{p-1}\log (1+Z_{i+1}),
\end{equation}
where $c_n=-p\log n+\log (n(n-1)\cdots (n-p+1))$ and $Z_{i+1}=n\mathbf{y}_{i+1}^{\top}\mathbf{Q}_i\mathbf{y}_{i+1}-1$ with $\mathbf{Q}_i$ being an $n\times n$ projection matrix normalized by its trace and only depending on $\mathbf{x}_1,\ldots,\mathbf{x}_i$, which can be decomposed as
\begin{equation*}
    Z_{i+1}=n\mathbf{y}_{i+1}^{\top}\operatorname{diag}(\mathbf{Q}_i)\mathbf{y}_{i+1}+n\mathbf{y}_{i+1}^{\top}\operatorname{off-diag}(\mathbf{Q}_i)\mathbf{y}_{i+1}-1,
\end{equation*}
where
\begin{equation}\label{eq_defQ}
    \mathbf{Q}_i=\mathbf{P}_i/(n-i)=(q_{i,k\ell})_{1\le k,\ell\le n}~\text{with}~\mathbf{P}_i=\mathbf{I}_n-\mathbf{B}_i^{\top}(\mathbf{B}_i\mathbf{B}_i^{\top})^{-1}\mathbf{B}_i,
\end{equation}
which is independent of $\mathbf{y}_{i+1}$. It is obvious that
\begin{equation*}
    Z_{i+1}=\frac{n\Delta_{i+1}^2-(n-i)}{n-i}.
\end{equation*}
By choosing $p_n=1$ in Theorem 4.1 of \cite{gotze2010circularlaw}, without loss of generality, we can assume that all square submatrices of $\mathbf{X}$ are invertible. So, all the random variables are well-defined.

Secondly, similarly to the non-singularity case $\lim_{n\rightarrow \infty}p/n=\gamma<1$ in \cite{heiny2023logdet}, we shall perform a Taylor series expansion (which is guaranteed by Lemma \ref{lemma_momentestimateUV}) for the random parts in $\log \operatorname{det}(\mathbf{R})$, say,
\begin{equation*}
	\sum_{i=0}^{p-1}\log(1+{Z}_{i+1})\approx \sum_{i=0}^{p-1}{Z}_{i+1}-\sum_{i=0}^{p-1}\frac{{Z}_{i+1}^2}{2}+\mathrm{error}.
\end{equation*}
and then use the martingale difference method to show the first part $\sum_{i=0}^{p-1}Z_{i+1}$ converges to Gaussian asymptotically, while the second part $\sum_{i=0}^{p-1}Z_{i+1}^2/2$ converges to some deterministic limit in probability, and the error term is negligible. However, it is not the case when $\lim_{n\rightarrow\infty}p/n\rightarrow 1$, and the martingale trick fails without the symmetry condition in \cite{heiny2023logdet}. On the one hand, notice that the random variable adapts the following decomposition
\begin{equation}\label{eq_decomZ0}
    Z_{i+1}=\sum_{j=1}^{n}q_{i,jj}(nY_{i+1,j}^2-1)+\sum_{k\ne l}q_{i,kl}nY_{i+1,k}Y_{i+1,l},
\end{equation}
where $\mathbb{E}Y_{i+1,k}Y_{i+1,l}=0$ if and only if $\xi$ is symmetric by \cite{Jonsson2010quadratic}. To tackle the asymmetry, we focus on the following centralized form as
\begin{equation}\label{eq_decomZ}
    \widetilde{Z}_{i+1}=\sum_{j=1}^{n}q_{i,jj}(nY_{i+1,j}^2-1)+\sum_{k\ne l}q_{i,kl}n(Y_{i+1,k}Y_{i+1,l}-\beta_{1,1})=:U_{i+1}+V_{i+1},
\end{equation}
where $\beta_{1,1}:=\mathbb{E}(Y_{11}Y_{12})=\mathrm{o}(n^{-3})$ for $\alpha> 3$ (see Lemma \ref{lemma_odd_moment} for details) does not affect the asymptotic properties of $\log \det\mathbf{R}$ (see Section \ref{sec_proof12} for details). This also suggests that the asymmetry does not affect the asymptotic distribution of $\log \det \mathbf{R}$ while it might have a non-negligible effect on the convergence rate as pointed out by \cite{gine1997student,Li2024clt}.

Moreover, since one can not guarantee the Taylor series expansion near the singularity case $\lim_{n\rightarrow\infty} p/n=1$, we adapt a replacement argument inspired by \cite{bao2015logarithmic,wang2018} and a resampling strategy by \cite{aggarwal2018goe,bao2023phase,Li2024clt}. Specifically, we introduce three parameters
\begin{equation}\label{eq_def_parameters}
    s_1=[p/100],~s_2=[(-\log (1-(p-1)/n))^{1/2}],~\text{and}~~s_3=[p/\log^a p]
\end{equation}
for some constant $a>2$. For $0\le i\le p-s_1-1$, similarly to \cite{heiny2023logdet}, we can perform a Taylor series expansion (which is guaranteed by Lemma \ref{lemma_momentestimateUV}) of the first $(p-s_1)$ random parts of $\log \operatorname{det}(\mathbf{R})$, say,
\begin{equation}\label{eq_expansion}
	\sum_{i=0}^{p-s_1-1}\log(1+\widetilde{Z}_{i+1})\approx \sum_{i=0}^{p-s_1-1}\widetilde{Z}_{i+1}-\sum_{i=0}^{p-s_1-1}\frac{\widetilde{Z}_{i+1}^2}{2}+\mathrm{error}.
\end{equation}
However, for large $i$ near $p$ (say, for instance, $i\ge p-s_3$), this expansion can not be guaranteed due to singularity in our heavy-tailed setting. The main barrier is that one can not control the variance of $U_{i+1}$ well for $i>p-s_3$ since $\mathbb{E}Y_{i+1,j}^4\sim \mathcal{L}(n^{1/2})n^{-\alpha/2}$ is deterministic by Lemma \ref{lem_moment_rates} for $\alpha\in [3,4)$. To overcome this barrier, we adopt a replacement strategy in the spirit of the work \citep{bao2015logarithmic,wang2018}. However, due to the limitations from heavy-tail and near singularity, it is more suitable to control the difference between two larger determinants rather than change lines one by one as introduced in \cite{bao2015logarithmic,wang2018}. Specifically, we consider a ``replaced random matrix'', that is, let $\check{\mathbf{X}}$ be a random matrix differing from $\mathbf{X}$ only in $s_1$ rows with the first $p-s_1$ rows unchanged, where the $s_1$ rows consist of i.i.d. standard Gaussian random variables. Then we aim to prove
    \begin{equation}\label{eq_diffRR}
        \mathbb{P}\left(\sigma_{n}|\log\operatorname{det}(\check{\mathbf{R}})-\log\operatorname{det}({\mathbf{R}})|=\mathrm{o}(1)\right)=1-\mathrm{o}(1),
    \end{equation}
 where $\sigma_{n}^2=:(-2\log (1-\frac{p-1}{n})-2\frac{p}{n})^{-1}$ and $\check{\mathbf{R}}$ denotes the sample correlation matrix of $\check{\mathbf{X}}$. Thus, it suffices to show
    \begin{equation}\label{eq_diffGR}
        \sigma_{n}(\log\operatorname{det}(\check{\mathbf{R}})-\mu_n)\stackrel{d}{\rightarrow}N(0,1),
    \end{equation}
where $\mu_n=:(p-n+1/2)\log (1-\frac{p-1}{n})-p+\frac{p}{n}$ is defined in Theorem \ref{thm_logdet}. Due to the additive structure in \eqref{eq_decomlogdetR}, we can perform a Taylor expansion for the first $(p-s_2)$ random parts of $\log\operatorname{det}(\check{\mathbf{R}})$ since we can control the variance of the diagonal part well for $p-s_1<i\le p-s_2-1$ by the replacement trick. Finally, for the last $s_2$ rows, we show that this part can be negligible due to the variance being of order $-\log (1-(p-1)/n)$.

For convenience, we begin with some notation. Let $\mathbf{z}_{p-s_1+1},\ldots,\mathbf{z}_{p}\in \mathbb{R}^n$ consist of i.i.d. standard normal random variables. Denote the self-normalized vectors $\check{\mathbf{y}}_{p-s_1+1},\ldots,\check{\mathbf{y}}_{p}\in \mathbb{R}^n$ with $\check{\mathbf{y}}_{i+1}=\mathbf{z}_{i+1}/\|\mathbf{z}_{i+1}\|$ for $p-s_1\le i\le p-1$. For the replaced random matrix $\check{\mathbf{X}}$, we also define $\check{\mathbf{B}}_i,\check{\mathbf{P}}_i,\check{\mathbf{Q}}_i$ similar to \eqref{eq_defQ}. Moreover, for $p-s_1\le i\le p-1$, let $\check{\Delta}_{i+1}^2=\check{\mathbf{y}}_{i+1}^{\top}\check{\mathbf{P}}_i\check{\mathbf{y}}_{i+1}$ and
\begin{equation}\label{eq_defcheckZ}
    \begin{split}
        \check{Z}_{i+1}
    =\sum_{j=1}^{n}\check{q}_{i,jj}(n\check{Y}_{i+1,j}^2-1)+\sum_{k\ne l}\check{q}_{i,kl}n\check{Y}_{i+1,k}\check{Y}_{i+1,l}=:\check{U}_{i+1}+\check{V}_{i+1}.
    \end{split}
\end{equation}
It is straightforward to get $\mathbb{E}\check{Z}_{i+1}=0$. For simplicity, for $0\le i\le p-1$, denote
\begin{equation}\label{eq_defxixi}
    \Delta_{i+1}^2=\mathbf{y}_{i+1}^{\top}\mathbf{P}_{i}\mathbf{y}_{i+1}~\text{and}~\check{\Delta}_{i+1}^2=\check{\mathbf{y}}_{i+1}^{\top}\check{\mathbf{P}}_{i}\check{\mathbf{y}}_{i+1}.
\end{equation}
In view of \eqref{eq_defdetR}, it is obvious that $\Delta_{i+1}^2=\check{\Delta}_{i+1}^2$ for $0\le i\le p-s_1-1$ due to the definitions of $\mathbf{X}$ and $\check{\mathbf{X}}$. So, we can write
\begin{equation*}
    \det (\mathbf{R})=(\Delta_{p-1}^2\cdots \Delta_{p-s_2}^2)\cdot (\Delta_{p-s_2-1}^2\cdots\Delta_{p-s_1}^2)\cdot\det(\mathbf{Y}_{p-s_1}\mathbf{Y}_{p-s_1}^{\top})
\end{equation*}
and
\begin{equation*}
    \det (\check{\mathbf{R}})=(\check{\Delta}_{p-1}^2\cdots \check{\Delta}_{p-s_2}^2)\cdot (\check{\Delta}_{p-s_2-1}^2\cdots\check{\Delta}_{p-s_1}^2)\cdot\det(\mathbf{Y}_{p-s_1}\mathbf{Y}_{p-s_1}^{\top}),
\end{equation*}
since the first $p-s_1$ rows of $\mathbf{X}$ and $\check{\mathbf{X}}$ are the same. So, we can represent $\log \det \mathbf{R}$ and $\log \det\check{\mathbf{R}}$ as
\begin{equation}\label{eq_splitlogdetR}
    \log\det \mathbf{R}=\sum_{i=0}^{p-1}\log \Delta_{i+1}^2=\sum_{i=0}^{p-s_1-1}\log \Delta_{i+1}^2+\sum_{i=p-s_1}^{p-s_2-1}\log \Delta_{i+1}^2+\sum_{i=p-s_2}^{p-1}\log \Delta_{i+1}^2
\end{equation}
and
\begin{equation}\label{eq_splitlogdetRc}
    \log\det \check{\mathbf{R}}=\sum_{i=0}^{p-1}\log \check{\Delta}_{i+1}^2=\sum_{i=0}^{p-s_1-1}\log \Delta_{i+1}^2+\sum_{i=p-s_1}^{p-s_2-1}\log \check{\Delta}_{i+1}^2+\sum_{i=p-s_2}^{p-1}\log \check{\Delta}_{i+1}^2.
\end{equation}
It is straightforward that
\begin{equation*}
    \log\det \mathbf{R}-\log\det \check{\mathbf{R}}=\sum_{i=p-s_1}^{p-1}\log \Delta_{i+1}^2-\sum_{i=p-s_1}^{p-1}\log \check{\Delta}_{i+1}^2
\end{equation*}
and
\begin{equation*}
    \log\det \check{\mathbf{R}}=\sum_{i=0}^{p-s_1-1}\log \Delta_{i+1}^2+\sum_{i=p-s_1}^{p-s_2-1}\log \check{\Delta}_{i+1}^2+\sum_{i=p-s_2}^{p-1}\log \frac{n-i}{n}+\sum_{i=p-s_2}^{p-1}\log \frac{n\check{\Delta}_{i+1}^2}{n-i}.
\end{equation*}
Formally, the proofs of Theorems \ref{thm_logdet} and \ref{thm_pn} are divided into the following four steps:
\begin{itemize}
    \item[(i)] Warm up: Multilevel block truncations
    \begin{equation*}
    	\mathbb{P}(\log\det\mathbf{R}= \log \det\hat{\mathbf{R}})=1-\mathrm{o}(1),
    \end{equation*}
    where $\hat{\mathbf{R}}$ is contructed from $\mathbf{X}$ after appropriate truncations.
    \item[(ii)] Opening: Hierarchical Gaussian replacement
    \begin{equation}\label{eq_controldiff}
        \mathbb{P}\left(\sigma_{n}\left|\sum_{i=p-s_1}^{p-1}\log \Delta_{i+1}^2-\sum_{i=p-s_1}^{p-1}\log \check{\Delta}_{i+1}^2\right|=\mathrm{o}(1)\right)=1-\mathrm{o}(1).
    \end{equation}
    \item[(iii)] Mid game: Martingale central limit theorem
    \begin{equation}\label{eq_martingaleclt}
        \sum_{i=0}^{p-s_1-1}\log \Delta_{i+1}^2+\sum_{i=p-s_1}^{p-s_2-1}\log \check{\Delta}_{i+1}^2+\sum_{i=p-s_2}^{p-1}\log \frac{n-i}{n}-\mu_n\stackrel{d}{\rightarrow}N(0,\sigma_{n}^{-2}),~\text{as}~n\rightarrow \infty.
    \end{equation}
    \item[(iv)] End game: Last $s_2$ rows are negligible
    \begin{equation}\label{eq_lasts2rows}
        \sigma_{n}\sum_{i=p-s_2}^{p-1}\log \frac{n\check{\Delta}_{i+1}^2}{n-i}\stackrel{\mathbb{P}}{\rightarrow} 0.
    \end{equation}
\end{itemize}

\section{Proof of the main results}\label{sec_proof}

\subsection{Preliminaries}\label{sec_proof1}
This subsection collects preliminaries which are used frequently throughout the sequel, including some key estimations for the mixed moments of self-normalized entries (Lemma \ref{lem_moment_rates} and Lemma \ref{lemma_odd_moment}) and novel identities of the normalized projection matrix $\mathbf{Q}_i$ defined by \eqref{eq_defQ} (see Lemma \ref{lemma_Qbounds}).
We begin with basic definitions.
\begin{definition}[High probability event]\label{def_highpro}
	We say an event $\Xi$ holds with high probability if for any constant $D> 0$, $\mathbb{P}(\Xi) \ge  1- n^{-D}$ for large enough $n$.
\end{definition}
\begin{lemma}[High probability lower bound of $\|\mathbf{b}_i\|^2$]\label{lemma_highpronorm}
	Under Assumption \ref{assump1} for $\alpha>2$, there exist positive constants $c_1,c_2,c_3>0$ such that
	\begin{equation}\label{eq_highpronorm1}
		\mathbb{P}(\|\mathbf{b}_i\|^2\le c_1n)\le \exp(-c_2n)
	\end{equation}
	and
	\begin{equation}\label{eq_highpronorm2}
	\mathbb{P}(\min_{1\le i\le p}\|\mathbf{b}_i\|^2\ge c_1n)\ge 1-\exp(-c_3n).
	\end{equation}
\end{lemma}
\begin{proof}
	Without loss of generality, let $C$ be a positive constant such that $\mathbb{P}(|\xi|>C)=1/2$. For constant $c_1>0$ such that $c_1/C^2<1/10$, one has
	\begin{equation*}
		\begin{split}
		\mathbb{P}(\|\mathbf{b}_i\|^2\le c_1n)\le \mathbb{P}(\sum_{j=1}^{n}\mathbbm{1}_{|X_{ij}|>C}\le \frac{nc_1}{C^2})\le \mathbb{P}(\sum_{j=1}^{n}\mathbbm{1}_{|X_{ij}|>C}\le \frac{n}{10})
		\le \sum_{k=0}^{n/10}\binom{n}{k}2^{-n}\le \exp(-c_2n),
		\end{split}
	\end{equation*}
	for some constant $c_2>0$, as $n\rightarrow \infty$, where in the last line we used $\binom{n}{k}\le (en/k)^{k}$. For the second statement,
	\begin{equation*}
	\begin{split}
	\mathbb{P}(\min_{1\le i\le p}\|\mathbf{b}_i\|^2\ge c_1n)
	\ge 1-\sum_{i=1}^{p}\mathbb{P}(\|\mathbf{b}_i\|^2\le c_1n)\ge 1-\exp(-c_3n)
	\end{split}
	\end{equation*}
	for some constant $c_3>0$, as $n\rightarrow \infty$. This completes the proof of this proposition.
\end{proof}

\subsubsection{The mixed moments of self-normalized entries}\label{sec_proof11}
Throughout this paper, for positive integers $k_1,\ldots,k_r$, we will use the notation
\begin{equation*}
    \beta_{k_1,\ldots,k_r}:=\mathbb{E}(Y_{11}^{k_1}\cdots Y_{1r}^{k_r}),
\end{equation*}
where we recall the definition of $Y_{ij}$ from \eqref{eq_selfentry}. Due to the i.i.d. assumption of $X_{ij}$, for any permutation $\pi$ on $\{1,\ldots,r\}$, we have $\beta_{k_1,\ldots,k_r}=\beta_{k_{\pi(1)},\ldots,k_{\pi(r)}}$, and thus we typically write the indices in decreasing order. For instance, we write $\beta_{3,1}$ instead of $\beta_{1,3}$. The first preliminary borrowed from Lemma 4.1 of \cite{heiny2023logdet} gives the precise asymptotic behavior of the even mixed moments of $Y_{ij}$'s.
\begin{lemma}\label{lem_moment_rates}
	Suppose Assumption \ref{assump1} holds with $\alpha\in (2,4)$. Consider $k_1,\dots,k_q\ge1$ for $q\in\{1,\dots, n\}$. Then
	\begin{equation*}
		\lim_{n\rightarrow\infty}\frac{n^{N_1(1-\alpha/2)+q\alpha/2}}{\mathcal{L}^{q-N_1}(n^{1/2})}\mathbb{E}(Y_{11}^{2k_1}\cdots Y_{1q}^{2k_q})=\frac{(\alpha/2)^{q-N_1}\Gamma(N_1(1-\alpha/2)+q\alpha/2)\prod_{i:k_i\ge2}^q\Gamma(k_i-\alpha/2)}{\Gamma(k_1+\dots+k_q)},
	\end{equation*}
	where $N_1=\#\{1\le i\le q:k_i=1\}$. In particular, we have
	\begin{equation*}
		\lim_{n\rightarrow\infty}\frac{n^{\alpha/2}}{\mathcal{L}(n^{1/2})}\mathbb{E}(Y_{11}^{2k})=\frac{\alpha\Gamma(\alpha/2)\Gamma(k-\alpha/2)}{2\Gamma(k)},\quad k\ge2.
	\end{equation*}
\end{lemma}
Next, we turn to the estimation of odd mixed moments of $Y_{ij}$'s, which relies on a pivotal estimation of $\beta_{1,1}=\mathbb{E}(Y_{11}Y_{12})$ and the self-normalization sum identity $Y_{11}^2+\cdots+Y_{1n}^2=1$, see Lemma 2.8 of \cite{Li2024clt}, and thus we omit the proof.
\begin{lemma}\label{lemma_odd_moment}
Suppose Assumption \ref{assump1} holds with $\alpha\in [3,4)$. Then, we have, as $n\rightarrow\infty$
\begin{equation*}
    \beta_{1,1}=\mathrm{o}(n^{-3}), \beta_{3,1}=\mathrm{o}(n^{-3}), \beta_{\underbrace{1,\ldots,1}_{2q}}=\mathrm{o}(n^{-2q-1})
\end{equation*}
for integer $q\ge 1$.
\end{lemma}
\begin{lemma}[Potter's bound \citep{albrecher2007asymptotic}]\label{lemma_potterbound}
    For the regularly varying function $\mathcal{L}(x)$, for any $\epsilon>0$ and $C>1$, we have $C^{-1}x^{-\epsilon}\le\mathcal{L}(x)\le Cx^{\epsilon}$ for sufficiently large $x$.
\end{lemma}

\subsubsection{The normalized projection matrices}\label{sec_proof12}
For a given integer $i\in [0,p-1]$, from the decomposition of $\widetilde{Z}_{i+1}$ in \eqref{eq_decomZ}, we can see that both the entries of the normalized projection matrices $\mathbf{Q}_i$ and the self-normalized entries $Y_{i+1,j}$'s govern the asymptotic behaviors of $\widetilde{Z}_{i+1}$'s. In this subsection, let's illustrate some deterministic bounds for  $\mathbf{Q}_i$, whose proof is deferred to Appendix \ref{sec-offdiag}. For simplicity, denote
\begin{equation}\label{eq_defSki}
S_{k}^{(i)}=:\sum_{j=1}^{n}q_{i,jj}^{k},~\text{for}~k\ge 1.
\end{equation}
It can be checked that $S_{1}^{(i)}=1$ and $S_{2}^{(i)}\le S_{1}^{(i)}/(n-i)=1/(n-i)$.
\begin{lemma}[Deterministic bounds for $\mathbf{Q}_i$]\label{lemma_Qbounds}
	Given $0\le i\le p-1$, for the matrices $\mathbf{Q}_i=\mathbf{P}_i/(n-i)$ defined by \eqref{eq_defQ}, we have the following bounds
	\begin{equation}\label{eq_boundq1}
	\sum_{l}|q_{i,kl}|\le \frac{n^{1/2}}{(n-i)}~,~\sum_{k\ne l}|q_{i,kl}|\le \frac{n}{(n-i)^{1/2}},~\sum_{l\ne s}|q_{i,kl}q_{i,ks}|\le \frac{n}{(n-i)^2},
	\end{equation}
	\begin{equation}\label{eq_boundq2}
	\sum_{k\ne l\ne s}|q_{i,kl}q_{i,ks}|\le \frac{n}{n-i}~,~\sum_{k\ne l\ne s}|q_{i,kl}q_{i,ks}q_{i,ls}|\le \frac{n}{(n-i)^2}.
	\end{equation}
	\begin{equation}\label{eq_boundq3}
	\sum_{k\ne l\ne s\ne m}|q_{i,kl}q_{i,ks}q_{i,lm}|\le \frac{n}{(n-i)^{3/2}},~\sum_{k\ne l\ne s\ne m}|q_{i,kl}q_{i,ks}q_{i,km}|\le \frac{n^{3/2}}{(n-i)^2}.
	\end{equation}
\end{lemma}
Thereafter, with \eqref{eq_boundq1}, we have
\begin{equation*}
    \begin{split}
        \sum_{i=0}^{p-1}|Z_{i+1}-\widetilde{Z}_{i+1}|=\sum_{i=0}^{p-1}|\sum_{k\ne l}q_{i,kl}n\beta_{1,1}|\le \sum_{i=0}^{p-1}\frac{n^2\beta_{1,1}}{(n-i)^{1/2}}\le Cn^{5/2}\beta_{1,1}=\mathrm{o}(1)
    \end{split}
\end{equation*}
by Lemma \ref{lemma_odd_moment}. Thus, by the elementary inequality $\log (1+x)-\log (1+y)=\log (1+\frac{x-y}{1+y})\le\mathrm{O}(|x-y|)$, it's sufficient to study the asymptotic behaviors of $\sum_{i=0}^{p-1}\log (1+\widetilde{Z}_{i+1})$.
\subsubsection{The variance of quadratics}
The variance estimation for $\widetilde{Z}_{i+1}$ plays an essential role in the martingale method. Before we present the proof of the main results, we collect the variance estimation for ${U}_{i+1}$ and $V_{i+1}$, which are used frequently in the sequel.
 Recalling the definitions of $U_{i+1}$ and $V_{i+1}$, we have
\begin{equation}\label{eq_meanU2}
    \begin{split}
        \mathbb{E}(U_{i+1}^2)=&\mathbb{E}\big(\sum_{j}q_{i,jj}(nY_{i+1,j}^2-1)\big)^2
        =\mathbb{E}\sum_{j}q_{i,jj}^2(n^2\beta_{4}-1)+(n^2\beta_{2,2}-1)\mathbb{E}(\sum_{j\ne \ell}q_{i,jj}q_{i,\ell\ell})\\
        =&(n^2\beta_4-1)\mathbb{E}\big(S_{2}^{(i)}-(n-1)^{-1}[1-S_{2}^{(i)}]\big)
        =n\frac{n^2\beta_4-1}{n-1}\mathbb{E}\big(S_{2}^{(i)}-n^{-1}\big)
    \end{split}
\end{equation}
by $n^2\beta_{2,2}-1=n^2(n-1)^{-1}(\beta_2-\beta_{4})-1=(1-n^2\beta_{4})/(n-1)$ and $\sum_{j}q_{i,jj}=1$.
Similarly, by the identities $\mathbf{Q}_{i}/(n-i)=\mathbf{Q}_i^2$ and $\operatorname{tr}(\mathbf{Q}_i)=1$, we have
\begin{equation*}
    \sum_{k\ne l}q_{i,kl}^2=\operatorname{tr}(\mathbf{Q}_i^2)-\sum_{k}q_{i,kk}^2=\frac{1}{n-i}-S_{2}^{(i)},
\end{equation*}
which further gives
\begin{equation}\label{eq_meanV2}
    \begin{split}
        \mathbb{E}(V_{i+1}^2)=&2\mathbb{E}\sum_{k\ne l}q_{i,kl}^2n^2\beta_{2,2}+4\mathbb{E}\sum_{k\ne l\ne s}q_{i,kl}q_{i,ks}n^2\beta_{2,1,1}\\
        &+\mathbb{E}\sum_{k\ne l\ne s\ne m}q_{i,kl}q_{i,sm}n^2\beta_{1,1,1,1}-n^2\beta_{1,1}^2\mathbb{E}\big(\sum_{k\ne l}q_{i,kl}\big)^2.
    \end{split}
\end{equation}
It is obvious that the last three terms vanish when $X_{ij}$ is symmetric. Without the symmetry condition, we claim that the first term is the main term, while the other terms are negligible by Lemma \ref{lemma_odd_moment} and Lemma \ref{lemma_Qbounds}. Taking expectation on both hands of $Y_{11}Y_{12}=Y_{11}Y_{12}(Y_{11}^2+\cdots+Y_{1n}^2)$ gives $\beta_{1,1}=2\beta_{3,1}+(n-2)\beta_{2,1,1}$, which further implies $\beta_{2,1,1}\le Cn^{-1}\beta_{1,1}=\mathrm{o}(n^{-4})$ by Lemma \ref{lemma_odd_moment}. This together with \eqref{eq_boundq2} gives
\begin{equation*}
    \sum_{k\ne l\ne s}q_{i,kl}q_{i,ks}n^2\beta_{2,1,1}\lesssim n^3\beta_{2,1,1}/(n-i)=\mathrm{o}(n^{-1}(n-i)^{-1}).
\end{equation*}
By \eqref{eq_boundq1} and Lemma \ref{lemma_odd_moment}, we have $\sum_{k\ne l\ne s\ne m}q_{i,kl}q_{i,sm}\le \big(\sum_{k\ne l}q_{i,kl}\big)^2\le n^2/(n-i),$
which further gives
\begin{equation*}
   \sum_{k\ne l\ne s\ne m}q_{i,kl}q_{i,sm}n^2\beta_{1,1,1,1}-n^2\beta_{1,1}^2\big(\sum_{k\ne l}q_{i,kl}\big)^2\le \mathrm{o}(n^{-1}(n-i)^{-1}).
\end{equation*}
Therefore, we have
\begin{equation}\label{eq_vsquare}
    \mathbb{E}(V_{i+1}^2)=2n^2\beta_{2,2}\big(\frac{1}{n-i}-\mathbb{E}S_{2}^{(i)}\big)+\mathrm{o}(\frac{1}{n(n-i)}).
\end{equation}

\subsection{Warm up: Multilevel block truncations}\label{sec_multitruncation}
We always consider the case $\lim_{n\rightarrow \infty}p/n\rightarrow 1$ with \eqref{eq_newcondition} unless otherwise specified in the sequel, since the non-singular case with \eqref{eq_oldcondition} follows from the martingale method directly.
For simplicity, denote
\begin{equation}\label{eq_defdn}
d_n=:s_2^{1-\mathfrak{c}},
\end{equation}
where $s_2=[(-\log (1-(p-1)/n))^{1/2}]\rightarrow \infty$ is given in \eqref{eq_def_parameters} and $0<\mathfrak{c}<1/4$ in \eqref{eq_newcondition}. It is obvious that $d_n\rightarrow\infty$ for $\lim_{n\rightarrow \infty}p/n=1$.
Firstly, we introduce multilevel truncations to the entries of $\mathbf{X}$ under the condition  \eqref{eq_newcondition}, which does not affect the limiting distribution of $\log\det\mathbf{R}$. On the one hand, if $n-p\ge p^{1/6}\log^{1/4} p$, we consider the following global truncation
\begin{equation}\label{eq_globaltruncation}
\hat{X}_{ij}=
{X}_{ij}\mathbbm{1}(|{X}_{ij}|<n^{2/3}\log n),
\end{equation}
which further implies
\begin{equation*}
\mathbb{P}(\hat{\mathbf{X}}\ne {\mathbf{X}})\le\sum_{i,j}\mathbb{P}(|X_{ij}|>n^{2/3}\log n)\lesssim \mathcal{L}(n^{2/3}\log n)n^{2-2\alpha/3}\log^{-\alpha}n\le C\log^{-2}n \rightarrow 0
\end{equation*}
for $\alpha\ge 3$ with \eqref{eq_newcondition}.
For the case of $0\le n-p<p^{1/6}\log^{1/4} p$, we need multilevel block truncations depending on the location of the rows. It is easy to see $-\log (1-(p-1)/n)\asymp \log n$ if $0\le n-p<n^{1/4}$, which yields that
\begin{equation*}
s_2\asymp \log^{1/2} n~,~d_n\asymp \log^{(1-\mathfrak{c})/2}n.
\end{equation*}
Specifically, let
\begin{equation}\label{eq_trunction}
\hat{X}_{ij}=\begin{cases}
{X}_{ij}\mathbbm{1}(|{X}_{ij}|<n^{2/3}\log n),~&\text{if}~1\le i\le p-p^{1/6}\log^{1/4} p,\\
{X}_{ij}\mathbbm{1}(|{X}_{ij}|<(np^{1/6})^{1/3}\log^{1/6} n),~&\text{if}~ p-p^{1/6}\log^{1/4} p<i\le p-p^{1/18}\log^{5/12}p,\\
{X}_{ij}\mathbbm{1}(|{X}_{ij}|<(n(p-i))^{1/3}\log^{1/12} n),~&\text{if}~ p-p^{1/18}\log^{5/12} p<i\le p-p^{1/18}\log^{1/4}p,\\
{X}_{ij}\mathbbm{1}(|{X}_{ij}|<(np^{1/18})^{1/3}\log^{1/6} n),~&\text{if}~ p-p^{1/18}\log^{1/4}p<i\le p-p^{1/54}\log^{5/12}p,\\
{X}_{ij}\mathbbm{1}(|{X}_{ij}|<(n(p-i))^{1/3}\log^{1/12} n),~&\text{if}~ p-p^{1/54}\log^{5/12} p<i\le p-p^{1/54}\log^{1/4}p,\\
\qquad\vdots~&~\qquad\vdots\\
{X}_{ij}\mathbbm{1}(|{X}_{ij}|<(nd_n)^{1/3}\log^{1/12} n),~&\text{if}~p-d_n+1\le i\le p.
\end{cases}
\end{equation}
For simplicity, denote $\mathbf{X}_0=(\mathbf{b}_1^{\top},\cdots,\mathbf{b}_{[p-p^{1/6}\log^{1/4}p]}^{\top})^{\top}$ as the base block of $\mathbf{X}$. Let
\begin{equation*}
\mathbf{X}_k=(\mathbf{b}_{[p-p^{\frac{1}{2\cdot 3^k}}\log^{1/4}p]+1}^{\top},\cdots,\mathbf{b}_{[p-p^{\frac{1}{2\cdot 3^{k+1}}}\log^{1/4}p]}^{\top})^{\top}
\end{equation*}
be the $k$-th block for integer $1\le k\le K$, where $K$ is some integer which will be fixed later.
Moreover, denote
\begin{equation*}
\mathbf{X}_{K}=:(\mathbf{b}_{[p-p^{\frac{1}{2\cdot 3^K}}\log^{1/4}p]+1}^{\top},\cdots,\mathbf{b}_{p-d_n}^{\top})^{\top}~\text{and}~\mathbf{X}_{K+1}=:(\mathbf{b}_{p-d_n+1}^{\top},\cdots,\mathbf{b}_{p}^{\top})^{\top}
\end{equation*}
as the last two blocks of $\mathbf{X}$. The multilevel block truncations are implemented as follows: in each block $\mathbf{X}_k$ with $1\le k\le K$, we first truncate the ``most'' top rows at a high level depending on the initial location of the block, and then apply a localized truncation line by line for the ``little'' bottom rows in the block. Specifically, for $1\le k\le K$,
\begin{equation}\label{eq_blocktruncation}
\hat{X}_{ij}=\begin{cases}
X_{ij}\mathbbm{1}(|X_{ij}|<(np^{\frac{1}{2\cdot 3^k}})^{1/3}\log^{1/6}n),&\text{if}~p-p^{\frac{1}{2\cdot 3^k}}\log^{1/4}p<i\le p-p^{\frac{1}{2\cdot 3^{k+1}}}\log^{5/12}p,\\
X_{ij}\mathbbm{1}(|X_{ij}|<(n(p-i))^{1/3}\log^{1/12}n),&\text{if}~p-p^{\frac{1}{2\cdot 3^{k+1}}}\log^{5/12}p<i\le p-p^{\frac{1}{2\cdot 3^{k+1}}}\log^{1/4}p.
\end{cases}
\end{equation}
Let $\hat{\mathbf{X}}_0,\ldots,\hat{\mathbf{X}}_{K+1}$ be the block matrices of $\hat{\mathbf{X}}$ defined in a similar way. Note $d_n\asymp \log^{(1-\mathfrak{c})/2} n>\log^{1/4}n$ by \eqref{eq_newcondition}. This shows that there exists $1\le K\le \log^{\mathfrak{c}/2}n$ such that
\begin{equation}\label{eq_defK}
p^{\frac{1}{2\cdot 3^{K+1}}}\log^{1/4}p< d_n\le p^{\frac{1}{2\cdot 3^{K}}}\log^{1/4}p,
\end{equation}
which implies that $K\le \log^{\mathfrak{c}/2}n$ as $n\rightarrow\infty$. Thus, by \eqref{eq_newcondition}, we have,
\begin{equation*}
\mathbb{P}(\hat{\mathbf{X}}_0\ne \mathbf{X}_0)\le \sum_{i=1}^{p-p^{1/6}\log^{1/12}p}\sum_{j=1}^{n}\mathbb{P}(|X_{ij}|>n^{2/3}\log n )\lesssim n^2\frac{\mathcal{L}(n^{2/3}\log n)}{(n^{2/3}\log n )^{\alpha}}\le C\log^{-2}n
\end{equation*}
and for all $1\le k\le K$,
\begin{equation*}
\begin{split}
\mathbb{P}(\hat{\mathbf{X}}_k\ne \mathbf{X}_k)\le& \sum_{i=[p-p^{\frac{1}{2\cdot 3^k}}\log^{1/4}p]+1}^{[p-p^{\frac{1}{2\cdot 3^{k+1}}}\log^{5/12}p]}\sum_{j=1}^{n}\mathbb{P}(|X_{ij}|> (np^{\frac{1}{2\cdot 3^k}})^{1/3}\log^{1/6}n)\\
&+\sum_{i=[p-p^{\frac{1}{2\cdot 3^{k+1}}}\log^{5/12}p]+1}^{[p-p^{\frac{1}{2\cdot 3^{k+1}}}\log^{1/4}p]}\sum_{j=1}^{n}\mathbb{P}(|X_{ij}|> (n(p-i))^{1/3}\log^{1/12}n)\\
\lesssim & np^{\frac{1}{2\cdot 3^k}}\log^{1/4}p\frac{\mathcal{L}((np^{\frac{1}{2\cdot 3^k}})^{1/3}\log^{1/6}n)}{[(np^{\frac{1}{2\cdot 3^k}})^{1/3}\log^{1/6}n]^{\alpha}}+n\sum_{i=[p-p^{\frac{1}{2\cdot 3^{k+1}}}\log^{5/12}p]+1}^{[p-p^{\frac{1}{2\cdot 3^{k+1}}}\log^{1/4}p]}\frac{\mathcal{L}((n(p-i))^{1/3}\log^{1/12}n)}{[(n(p-i))^{1/3}\log^{1/12}n]^{\alpha}}\\
\lesssim& \frac{\log^{1/2-\mathfrak{c}}n}{\log^{1/2}n}+\frac{\log^{1/4-\mathfrak{c}}n}{\log^{1/4}n}\sum_{i=[p-p^{\frac{1}{2\cdot 3^{k+1}}}\log^{5/12}p]+1}^{[p-p^{\frac{1}{2\cdot 3^{k+1}}}\log^{1/4}p]}\frac{1}{p-i}\lesssim (\log\log p)\log^{-\mathfrak{c}}n\le \log^{-2\mathfrak{c}/3}n
\end{split}
\end{equation*}
by \eqref{eq_newcondition} as $n\rightarrow\infty$.
Similarly, for the last block, we have
\begin{equation*}
\begin{split}
\mathbb{P}(\hat{\mathbf{X}}_{K+1}\ne \mathbf{X}_{K+1})\le \sum_{i=p-d_n+1}^{p}\sum_{j=1}^{n}\mathbb{P}(|X_{ij}|> (nd_n)^{1/3}\log^{1/12}n)
\lesssim  nd_n\frac{\mathcal{L}((nd_n)^{1/3}\log^{1/12}n)}{[(nd_n)^{1/3}\log^{1/12}n]^{\alpha}}\lesssim \log^{-\mathfrak{c}}n.
\end{split}
\end{equation*}
It follows that
\begin{equation*}
\begin{split}
\mathbb{P}(\hat{\mathbf{X}}\ne {\mathbf{X}})\le \sum_{k=0}^{K+1}\mathbb{P}(\hat{\mathbf{X}}_k\ne \mathbf{X}_k)
\lesssim \log^{-2}n+K\log^{-2\mathfrak{c}/3}n+\log^{-\mathfrak{c}}n
\le \log^{-\mathfrak{c}/6}n= \mathrm{o}(1)
\end{split}
\end{equation*}
as $n\rightarrow\infty$. Thus, we have
\begin{equation*}
\mathbb{P}(\log\det\mathbf{R}= \log \det\hat{\mathbf{R}})=1-\mathrm{o}(1),
\end{equation*}
which implies that we can always work on the data matrix $\mathbf{X}$ under the multilevel block truncations \eqref{eq_trunction}. Without abusing notation, we also use $\mathbf{X}$ and $\mathbf{R}$ to denote the associated matrices under truncation \eqref{eq_trunction}. For simplicity, we call the truncation level $n^{2/3}\log n$ the global truncation.

\subsection{Opening: Hierarchical Gaussian replacement}\label{sec_gaussianreplace}
The hierarchical Gaussian replacement includes two parts depending on the value of $n-p$. Recall $d_n=s_2^{1-\mathfrak{c}}$ and $s_2=[(-\log (1-(p-1)/n))^{1/2}]$ in \eqref{eq_defdn}.
\subsubsection{Bulk replacement: The case $n-p\ge d_n$}\label{sec_bulkgaussianreplace}
Recall the definitions of $\Delta_{i+1}^2$ and $\check{\Delta}_{i+1}^2$ in \eqref{eq_defxixi}. It follows that $\mathbb{E}\check{\Delta}_i^2=(n-i)/n$ for $p-s_1\le i\le p-1$ and
\begin{equation*}
    \mathbb{E}\Delta_{i+1}^2= \frac{n-i}{n}+\beta_{1,1}(n-i)\sum_{k\ne l}q_{i,kl}=\frac{n-i}{n}(1+\mathrm{o}(n^{-3/2})),
\end{equation*}
where we used $\beta_{1,1}=\mathrm{o}(n^{-3})$ by Lemma \ref{lemma_odd_moment} and $\sum_{k\ne l}|q_{i,kl}|\le n(n-i)^{-1/2}$ by \eqref{eq_boundq1}, which together with the conditional argument implies
\begin{equation}\label{eq_meandiff}
    |\log \mathbb{E}(\Delta_{p-1}^2 \cdots \Delta_{p-s_1+1}^2)-\log \mathbb{E}(\check{\Delta}_{p-1}^2 \cdots \check{\Delta}_{p-s_1+1}^2)|\lesssim |s_1\log (1+\mathrm{o}(n^{-3/2}))|=\mathrm{o}(n^{-1/2}),
\end{equation}
where we used $\log (1+x)\le |x|$ for $x\rightarrow 0$.
Note the fact that $-2\log (1-(p-1)/n)-2p/n\sim -2\log (1-(p-1)/n)$ for $\lim_{n\rightarrow \infty}p/n\rightarrow 1$.
Let
\begin{equation*}
    t_{np}=\sqrt{-2\log (1-(p-1)/n)}\cdot s_2^{-1}\log s_2\asymp \log s_2\rightarrow \infty
\end{equation*}
for $\lim_{n\rightarrow\infty}p/n\rightarrow 1$.
In what follows, we aim to show that
\begin{equation}\label{eq_diffalpha}
    \mathbb{P}(|\log (\Delta_{p-1}^2 \cdots \Delta_{p-s_1+1}^2)-\log \mathbb{E}(\Delta_{p-1}^2 \cdots \Delta_{p-s_1+1}^2)|>t_{np})=\mathrm{o}(1)
\end{equation}
and
\begin{equation}\label{eq_diffG}
    \mathbb{P}(|\log (\check{\Delta}_{p-1}^2 \cdots \check{\Delta}_{p-s_1+1}^2)-\log \mathbb{E}(\check{\Delta}_{p-1}^2 \cdots \check{\Delta}_{p-s_1+1}^2)|>t_{np})=\mathrm{o}(1),
\end{equation}
which together with \eqref{eq_meandiff} complete the proof of \eqref{eq_controldiff}.
Consider \eqref{eq_diffG} first. It is obvious that
\begin{equation*}
    \begin{split}
        &\mathbb{P}\big(\left|\log \frac{\check{\Delta}_{p-1}^2 \cdots \check{\Delta}_{p-s_1+1}^2}{\mathbb{E}(\check{\Delta}_{p-1}^2 \cdots \check{\Delta}_{p-s_1+1}^2)}\right|>t_{np}\big)\\
        \le & \mathbb{P}\big(\frac{\check{\Delta}_{p-1}^2 \cdots \check{\Delta}_{p-s_1+1}^2}{\mathbb{E}(\check{\Delta}_{p-1}^2 \cdots \check{\Delta}_{p-s_1+1}^2)}-1>e^{t_{np}}-1\big)+\mathbb{P}\big(\frac{\check{\Delta}_{p-1}^2 \cdots \check{\Delta}_{p-s_1+1}^2}{\mathbb{E}(\check{\Delta}_{p-1}^2 \cdots \check{\Delta}_{p-s_1+1}^2)}-1<e^{-t_{np}}-1\big)\\
        \le & \frac{\mathbb{E}(\check{\Delta}_{p-1}^2 \cdots \check{\Delta}_{p-s_1+1}^2-\mathbb{E}(\check{\Delta}_{p-1}^2 \cdots \check{\Delta}_{p-s_1+1}^2))^2}{(e^{t_{np}}-1)^2[\mathbb{E}(\check{\Delta}_{p-1}^2 \cdots \check{\Delta}_{p-s_1+1}^2)]^2}+\frac{\mathbb{E}(\check{\Delta}_{p-1}^2 \cdots \check{\Delta}_{p-s_1+1}^2-\mathbb{E}(\check{\Delta}_{p-1}^2 \cdots \check{\Delta}_{p-s_1+1}^2))^2}{(e^{-t_{np}}-1)^2[\mathbb{E}(\check{\Delta}_{p-1}^2 \cdots \check{\Delta}_{p-s_1+1}^2)]^2}\\
        \le& C\frac{\operatorname{Var}(\check{\Delta}_{p-1}^2 \cdots \check{\Delta}_{p-s_1+1}^2)}{ [\mathbb{E}(\check{\Delta}_{p-1}^2 \cdots \check{\Delta}_{p-s_1+1}^2)]^2},
    \end{split}
\end{equation*}
where we used Markov's inequality and the fact $t_{np}\rightarrow \infty$ as $n\rightarrow\infty$.
Then, it suffices to show
\begin{equation}\label{eq_diffRR1}
   \operatorname{Var}(\check{\Delta}_{p-1}^2 \cdots \check{\Delta}_{p-s_1+1}^2) = \mathrm{o}(1)\cdot[\mathbb{E}(\check{\Delta}_{p-1}^2 \cdots \check{\Delta}_{p-s_1+1}^2)]^2
\end{equation}
as $n\rightarrow\infty$ and $n-p\ge d_n$.
Moreover, recalling \eqref{eq_defcheckZ} and the relation $\check{\Delta}_{i+1}^2=(\check{Z}_{i+1}+1)(n-i)/n$, we have
\begin{equation}\label{eq_conditionvar}
    \operatorname{Var}(\check{\Delta}_{i+1}^2\mid \mathcal{F}_{i})\le C(\frac{n-i}{n})^2\operatorname{Var}(\check{Z}_{i+1}\mid \mathcal{F}_i)\le C(\frac{n-i}{n})^2\mathbb{E}(\check{U}_{i+1}^2+\check{V}_{i+1}^2\mid \mathcal{F}_i)\le C\frac{n-i}{n^2},
\end{equation}
where we used the facts
\begin{equation*}
    \mathbb{E}(\check{U}_{i+1}^2+\check{V}_{i+1}^2\mid \mathcal{F}_i)\le Cn^2(\mathbb{E}\check{Y}_{i+1,j}^4)\sum_{k=1}^{n}\check{q}^2_{i,kk}+Cn^2(\mathbb{E}\check{Y}_{i+1,1}^2\check{Y}_{i+1,2}^2)\sum_{k\ne l}\check{q}_{i,kl}^2\le C(n-i)^{-1}
\end{equation*}
due to $\mathbb{E}\check{Y}_{i+1,j}^4=Cn^{-2}$ and the trivial bound $\sum_{k,l}\check{q}_{i,kl}^2=\operatorname{tr}(\check{\mathbf{Q}}_i)/(n-i)=\frac{1}{n-i}$.
Applying the conditional argument gives
\begin{equation}\label{eq_Exixi}
    \mathbb{E}\check{\Delta}_{p-1}^2\cdots \check{\Delta}_{p-s_1+1}^2=\prod_{i=p-s_1}^{p-1}\frac{n-i}{n}.
\end{equation}
In what follows, we use the induction on $p-s_1+1\le p-j\le p-1$ to show $\operatorname{Var}(\check{\Delta}_{p-1}^2\cdots \check{\Delta}_{p-s_1+1}^2)\le \mathrm{o}(1)(\mathbb{E}(\check{\Delta}_{p-1}^2\cdots \check{\Delta}_{p-s_1+1}^2))^2$. First, we have
\begin{equation*}
    \operatorname{Var}(\check{\Delta}_{p-s_1+1}^2)\le C\frac{n-p+s_1}{n^2}=\mathrm{o}(1)\frac{(n-p+s_1)^2}{n^2}=\mathrm{o}(1)(\mathbb{E}\check{\Delta}_{p-s_1+1}^2)^2.
\end{equation*}
Suppose, for some $p-s_1+1\le p-j\le p-2$, we have
\begin{equation}\label{eq_inductionvar}
    \operatorname{Var}(\check{\Delta}_{p-j}^2\cdots \check{\Delta}_{p-s_1+1}^2)\le \mathrm{o}(1)(\mathbb{E}(\check{\Delta}_{p-j}^2\cdots \check{\Delta}_{p-s_1+1}^2))^2.
\end{equation}
Thus, the law of total variance and conditional expectation argument imply
\begin{equation*}
    \begin{split}
        &\operatorname{Var}(\check{\Delta}_{p-j+1}^2\cdots \check{\Delta}_{p-s_1+1}^2)
        =\mathbb{E}(\operatorname{Var}(\check{\Delta}_{p-j+1}^2\cdots \check{\Delta}_{p-s_1+1}^2\mid \mathcal{F}_{p-j}))+\operatorname{Var}(\mathbb{E}(\check{\Delta}_{p-j+1}^2\cdots \check{\Delta}_{p-s_1+1}^2\mid \mathcal{F}_{p-j}))\\
        =&\mathbb{E}\big([\check{\Delta}_{p-j}^2\cdots \check{\Delta}_{p-s_1+1}^2]^2\cdot\operatorname{Var}(\check{\Delta}_{p-j+1}^2\mid \mathcal{F}_{p-j})\big)+(\frac{n-p+j}{n})^2\operatorname{Var}(\check{\Delta}_{p-j}^2\cdots \check{\Delta}_{p-s_1+1}^2)\\
        \le &C[(\mathbb{E}(\check{\Delta}_{p-j}^2\cdots \check{\Delta}_{p-s_1+1}^2))^2+\operatorname{Var}(\check{\Delta}_{p-j}^2\cdots \check{\Delta}_{p-s_1+1}^2)]\frac{n-p+j}{n^2}+\mathrm{o}(1)(\frac{n-p+j}{n})^2(\mathbb{E}(\check{\Delta}_{p-j}^2\cdots \check{\Delta}_{p-s_1+1}^2))^2\\
        \le & \mathrm{o}(1)(\mathbb{E}(\check{\Delta}_{p-j+1}^2\cdots \check{\Delta}_{p-s_1+1}^2))^2,
    \end{split}
\end{equation*}
for $n-p\ge d_n\rightarrow \infty$, where in the first inequality we used \eqref{eq_conditionvar} and \eqref{eq_inductionvar}, in the last line we used \eqref{eq_Exixi}, completing the proof of \eqref{eq_diffG}.

It remains to show \eqref{eq_diffalpha}, which is more involved and relies on the multilevel block truncations \eqref{eq_trunction} and Lemma \ref{lemma_upperdiagonal}.
Following a similar argument for \eqref{eq_diffG}, it suffices to show
\begin{equation}\label{eq_diffXX}
    \mathbb{P}\big(|\log \frac{\Delta_{p-1}^2 \cdots \Delta_{p-s_1+1}^2}{\mathbb{E}(\Delta_{p-1}^2 \cdots \Delta_{p-s_1+1}^2)}|>t_{np}\big)=\mathrm{o}(1).
\end{equation}
To this end, we have to distinguish three cases for $p-s_1\le i\le p-1$. \\
\noindent\textbf{Case I. $p-s_1\le i\le p-p^{2/3}$.}\\
It holds that
\begin{equation*}
\operatorname{Var}(\Delta_{i+1}^2\mid \mathcal{F}_i)\le C(\frac{n-i}{n})^2\mathbb{E}(U_{i+1}^2+V_{i+1}^2\mid \mathcal{F}_i)\lesssim C(\frac{n-i}{n})^2\frac{\mathcal{L}(n^{1/2})n^{2-\alpha/2}}{p^{2/3}}=\mathrm{o}((\frac{n-i}{n})^2).
\end{equation*}
for $\alpha\ge 3$. Thus, applying the induction as before, we can show
\begin{equation*}
\operatorname{Var}(\Delta_{i+1}^2\cdots \Delta_{p-s_1+1}^2)=\mathrm{o}(1)(\mathbb{E}(\Delta_{i+1}^2\cdots \Delta_{p-s_1+1}^2))^2
\end{equation*}
as $n\rightarrow \infty$, which implies \eqref{eq_diffXX} similar to \eqref{eq_diffG}.

\noindent\textbf{Case II. $p-p^{2/3}\le i\le p-p^{1/6}\log^{1/4}p$.}\\
In this case, we need to control the magnitude of $\max_{1\le k\le n}p_{i,kk}$ carefully. By \eqref{eq_maxpikklarge}, we have, for any $p-p^{2/3}\le i\le p-1$, with high probability,
\begin{equation*}
    \mathbb{P}(\max_{1\le k\le n}p_{i,kk}\ge n^{-1/8}\log s_2)=\mathrm{O}(n^{-3/4}\log^{10} n),
\end{equation*}
which implies
\begin{equation}\label{eq_eventmaxp1}
    \begin{split}
        &\mathbb{P}\big(\bigcup_{p-p^{2/3}\le i\le p-1}\max_{1\le k\le n}p_{i,kk}\ge n^{-1/8}\log s_2\big)\\
        \le &\sum_{i=[p-p^{2/3}]}^{p-1} \mathbb{P}(\max_{1\le k\le n}p_{i,kk}\ge n^{-1/8}\log s_2)
        \lesssim p^{2/3}n^{-3/4}\log ^{10}n\le \mathrm{o}(n^{-1/20})
    \end{split}
\end{equation}
with high probability, as $n\rightarrow \infty$. Analgously, by \eqref{eq_maxpikksmall}, we have, with high probability,
\begin{equation}\label{eq_eventmaxp3}
\begin{split}
\mathbb{P}\big(\bigcup_{p-p^{7/16}\le i\le p-1}\max_{1\le k\le n}p_{i,kk}\ge n^{-1/4}\log s_2\big)\lesssim \mathrm{o}(n^{-1/20})
\end{split}
\end{equation}
and
\begin{equation}\label{eq_eventmaxp2}
    \begin{split}
        \mathbb{P}\big(\bigcup_{p-p^{7/24}\le i\le p-1}\max_{1\le k\le n}p_{i,kk}\ge n^{-1/3}\log s_2\big)\lesssim \mathrm{o}(n^{-1/30}).
    \end{split}
\end{equation}
Define three events
\begin{equation}\label{eq_event1}
    \begin{split}
    \mathcal{A}_1&=\bigcup_{p-p^{2/3}\le i\le p-p^{7/16}}\max_{1\le k\le n}p_{i,kk}\ge n^{-1/8}\log s_2,\\
    \mathcal{A}_2&=\bigcup_{p-p^{7/16}\le i\le p-p^{7/24}}\max_{1\le k\le n}p_{i,kk}\ge n^{-1/4}\log s_2,\\
    \mathcal{A}_3&=\bigcup_{p-p^{7/24}\le i\le p-1}\max_{1\le k\le n}p_{i,kk}\ge n^{-1/3}\log s_2.
    \end{split}
\end{equation}
It follows that $\mathbb{P}(\mathcal{A}_1\cup\mathcal{A}_2\cup\mathcal{A}_3)=\mathrm{o}(n^{-1/30})$ with high probability as $n\rightarrow \infty$. Therefore, similarly to \eqref{eq_diffG}, we have
\begin{equation*}
    \begin{split}
        \mathbb{P}\big(\left|\log \frac{{\Delta}_{p-1}^2 \cdots {\Delta}_{p-s_1+1}^2}{\mathbb{E}({\Delta}_{p-1}^2 \cdots {\Delta}_{p-s_1+1}^2)}\right|>t_{np}\big)\le \mathbb{P}\big(\left|\log \frac{{\Delta}_{p-1}^2 \cdots {\Delta}_{p-s_1+1}^2}{\mathbb{E}({\Delta}_{p-1}^2 \cdots {\Delta}_{p-s_1+1}^2)}\right|>t_{np},\mathcal{A}_1^c\cap\mathcal{A}_2^c\cap\mathcal{A}_3^c\big)+\mathrm{o}(n^{-1/30}),
    \end{split}
\end{equation*}
where we used Markov's inequality and the fact $t_{np}\rightarrow \infty$ as $n\rightarrow\infty$. It suffices to consider the variance estimation conditional on $\mathcal{A}_1^c\cap\mathcal{A}_2^c\cap\mathcal{A}_3^c$. So, in the sequel, we assume
\begin{equation}\label{eq_conditionmaxpikk}
    \begin{split}
    \max_{p-p^{2/3}\le i\le p-p^{7/16}}\max_{1\le k\le n}p_{i,kk}< n^{-1/8}\log s_2,
    \max_{p-p^{7/16}\le i\le p-p^{7/24}}\max_{1\le k\le n}p_{i,kk}< n^{-1/4}\log s_2,
    \end{split}
\end{equation}
and
\begin{equation}\label{eq_conditionmaxpikk2}
    \max_{p-p^{7/24}\le i\le p-1}\max_{1\le k\le n}p_{i,kk}<n^{-1/3}\log s_2.
\end{equation}
For $p-p^{2/3}\le i\le p-p^{7/16}$, we have
\begin{equation*}
\operatorname{Var}(\Delta_{i+1}^2\mid \mathcal{F}_i)\le C(\frac{n-i}{n})^2\mathbb{E}(U_{i+1}^2+V_{i+1}^2\mid \mathcal{F}_i).
\end{equation*}
Since $\sum_{k=1}^{n}q_{i,kk}=1$ and Lemma \ref{lem_moment_rates}, we have
\begin{equation*}
\begin{split}
\mathbb{E}(U_{i+1}^2\mid \mathcal{F}_i)\le& Cn^2(\mathbb{E}Y_{i+1,j}^4)\sum_{k=1}^{n}q_{i,kk}^2\le Cn^2(\mathbb{E}Y_{i+1,j}^4)\max_{1\le k\le n}q_{i,kk}
\lesssim \mathcal{L}(n^{1/2})n^{2-\alpha/2}\frac{\max_{1\le k\le n}p_{i,kk}}{n-i}
\end{split}
\end{equation*}
by $\sum_{k=1}^{n}q_{i,kk}^2\le \max_{1\le k\le n}q_{i,kk}\sum_{k=1}^{n}q_{i,kk}=\max_{1\le k\le n}q_{i,kk}$. From \eqref{eq_vsquare}, we have $\mathbb{E}(V_{i+1}^2\mid\mathcal{F}_i)\le C(n-i)^{-1}$. Given \eqref{eq_conditionmaxpikk}, it is easy to verify that
\begin{equation*}
\mathbb{E}(U_{i+1}^2+V_{i+1}^2\mid \mathcal{F}_i)\lesssim \mathcal{L}(n^{1/2})n^{2-\alpha/2}\frac{n^{-1/8}\log s_2}{n-i}\lesssim \frac{\mathcal{L}(n^{1/2})n^{1/2}n^{-1/8}\log s_2}{p^{7/16}}=\mathrm{o}(1)
\end{equation*}
for $\alpha\ge 3$ and $p-p^{2/3}\le i\le p-p^{7/16} $ as $n\rightarrow \infty$.

Analgously, for $p-p^{7/16}\le i\le p-p^{7/24}$,
given \eqref{eq_conditionmaxpikk}, we have
\begin{equation*}
    \mathbb{E}(U_{i+1}^2+V_{i+1}^2\mid \mathcal{F}_i)\lesssim \mathcal{L}(n^{1/2})n^{2-\alpha/2}\frac{n^{-1/4}\log s_2}{n-i}\lesssim \mathcal{L}(n^{1/2})n^{-1/24}\log s_2=\mathrm{o}(1)
\end{equation*}
for $\alpha\ge 3$. For $p-p^{7/24}\le i\le p-p^{1/6}\log^{1/4} p$, given \eqref{eq_conditionmaxpikk2}, we get
\begin{equation*}
    \mathbb{E}(U_{i+1}^2+V_{i+1}^2\mid \mathcal{F}_i)\lesssim \mathcal{L}(n^{1/2})n^{2-\alpha/2}\frac{n^{-1/3}\log s_2}{n-i}\le \mathcal{L}(n^{1/2})n^{(3-\alpha)/2}\frac{\log s_2}{\log^{1/4} p}\lesssim \frac{\log\log n}{\log^{1/4-\mathfrak{c}}n}=\mathrm{o}(1)
\end{equation*}
due to \eqref{eq_newcondition} with $0<\mathfrak{c}<1/4$.

\noindent\textbf{Case III. $p-p^{1/6}\log^{1/4}p\le i\le p-1$.}\\
While for the last $[p^{1/6}\log^{1/4} p]$ rows, we need the multilevel block truncations in \eqref{eq_trunction}. Firstly, by \eqref{eq_highpronorm2} in Lemma \ref{lemma_highpronorm}, we assume that $\|\mathbf{b}_{i+1}\|^2\ge cn$ for all $p-p^{1/6}\log^{1/4}p\le i\le p-1$. For $1\le k\le K$, set
\begin{equation*}
t_{nk}=:(np^{\frac{1}{2\cdot 3^k}})^{1/3}\log^{1/6}n~\text{and}~T_{ni}=:(n(p-i))^{1/3}\log^{1/12}n.
\end{equation*}
For the $k$-th block $\mathbf{X}_k$ with $1\le k\le K$, recall \eqref{eq_blocktruncation},
\begin{equation*}
	|{X}_{ij}|<\begin{cases}
t_{nk},&\text{for}~p-p^{\frac{1}{2\cdot 3^k}}\log^{1/4}p<i\le p-p^{\frac{1}{2\cdot 3^{k+1}}}\log^{5/12}p,\\
	T_{ni},&\text{for}~p-p^{\frac{1}{2\cdot 3^{k+1}}}\log^{5/12}p<i\le p-p^{\frac{1}{2\cdot 3^{k+1}}}\log^{1/4}p,
	\end{cases}
\end{equation*}
after truncation. Thus, for $p-p^{\frac{1}{2\cdot 3^k}}\log^{1/4}p<i\le p-p^{\frac{1}{2\cdot 3^{k+1}}}\log^{5/12}p$, the fourth moment of the self-normalized entries can be bounded as
\begin{equation*}
    \begin{split}
        \mathbb{E}Y_{i+1,j}^4\mathbbm{1}(|X_{i+1,j}|<t_{nk})\lesssim \frac{\mathbb{E}X_{i+1,j}^4\mathbbm{1}(|X_{i+1,j}|<t_{nk})}{\|\mathbf{b}_{i+1}\|^4}
        \lesssim \frac{\mathcal{L}(t_{nk})t_{nk}^{4-\alpha}}{n^2},
    \end{split}
\end{equation*}
where we used Lemma \ref{lemma_truncatemoment}. Analgously, for $p-p^{\frac{1}{2\cdot 3^{k+1}}}\log^{5/12}p<i\le p-p^{\frac{1}{2\cdot 3^{k+1}}}\log^{1/4}p$, we have
\begin{equation*}
\begin{split}
\mathbb{E}Y_{i+1,j}^4\mathbbm{1}(|X_{i+1,j}|<T_{nk})\lesssim  \frac{\mathcal{L}(T_{ni})T_{ni}^{4-\alpha}}{n^2}.
\end{split}
\end{equation*}
Therefore, following a similar argument above, we get, given \eqref{eq_conditionmaxpikk2}, for each $i$ such that $p-p^{\frac{1}{2\cdot 3^k}}\log^{1/4} p< i\le p-p^{\frac{1}{2\cdot 3^{k+1}}}\log^{5/12} p$ in block $\mathbf{X}_k$,
\begin{equation*}
    \begin{split}
    \mathbb{E}(U_{i+1}^2+V_{i+1}^2\mid \mathcal{F}_i)\lesssim& Cn^2\mathbb{E}Y_{i+1,j}^4\mathbbm{1}(|X_{i+1,j}|<t_{nk})\frac{n^{-1/3}\log s_2}{n-i}\\
    \lesssim& \mathcal{L}(t_{nk})t_{nk}^{4-\alpha}\frac{n^{-1/3}\log s_2}{n-i}=\mathcal{L}(t_{nk})[(np^{\frac{1}{2\cdot 3^k}})^{1/3}\log^{1/6}n]^{4-\alpha}\cdot\frac{n^{-1/3}\log s_2}{n-p+p-i}\\
    \lesssim& \mathcal{L}(t_{nk})n^{(3-\alpha)/3}(\log^{1/6}n)\log s_2\cdot\frac{p^{\frac{1}{2\cdot 3^{k+1}}}}{d_n+p^{\frac{1}{2\cdot 3^{k+1}}}\log^{5/12} p}\lesssim \frac{\log \log n}{\log^{1/4-\mathfrak{c}}n}\lesssim \mathrm{o}(1)
    \end{split}
\end{equation*}
with high probability, as $n\rightarrow \infty$, where in the last line we used $0<\mathfrak{c}<1/4$ in \eqref{eq_newcondition} and the assumption $n-p\ge d_n$. For each $i$ such that $p-p^{\frac{1}{2\cdot 3^{k+1}}}\log^{5/12}p<i\le p-p^{\frac{1}{2\cdot 3^{k+1}}}\log^{1/4}p$ in block $\mathbf{X}_k$, it follows that
\begin{equation*}
\begin{split}
\mathbb{E}(U_{i+1}^2+V_{i+1}^2\mid \mathcal{F}_i)
\lesssim& \mathcal{L}(T_{nk})T_{nk}^{4-\alpha}\frac{n^{-1/3}\log s_2}{n-i}=\mathcal{L}(T_{nk})[(n(p-i))^{1/3}\log^{1/12}n]^{4-\alpha}\cdot\frac{n^{-1/3}\log s_2}{n-p+p-i}\\
\lesssim& \mathcal{L}(T_{nk})(\log^{1/12}n)\log s_2\cdot\frac{(p-i)^{1/3}}{d_n+p-i}\lesssim \frac{(\log^{1/3-\mathfrak{c}}n)\log \log n}{d_n^{2/3}}\lesssim \frac{\log\log n}{\log^{2\mathfrak{c}/3}}=\mathrm{o}(1)
\end{split}
\end{equation*}
where in the last line we used $p-i\ge d_n\asymp \log^{(1-\mathfrak{c})/2} n$.

For the case of $p-d_n\le i\le p-1$, invoking the multilevel truncations \eqref{eq_trunction}, we have
\begin{equation*}
	\mathbb{E}Y_{i+1,j}^4\mathbbm{1}(|X_{i+1,j}|<(nd_n)^{1/3}\log^{1/12} n)\lesssim \frac{\mathcal{L}((nd_n)^{1/3}\log^{1/12} n )[(nd_n)^{1/3}\log^{1/12} n]^{4-\alpha}}{n^2},
\end{equation*}
which together with \eqref{eq_conditionmaxpikk2} further implies
\begin{equation*}
\begin{split}
\mathbb{E}(U_{i+1}^2+V_{i+1}^2\mid \mathcal{F}_i)\lesssim& \mathcal{L}((nd_n)^{1/3}\log^{1/12} n)[(nd_n)^{1/3}\log^{1/12} n]^{4-\alpha}\frac{n^{-1/3}\log s_2}{n-p+p-i}\\
\lesssim& \log^{1/3-\mathfrak{c}}n\frac{\log s_2}{d_n^{2/3}}\lesssim \frac{\log \log n}{\log^{-2\mathfrak{c}/3}n}=\mathrm{o}(1)
\end{split}
\end{equation*}
with high probability, by \eqref{eq_newcondition}.
And thus, given \eqref{eq_conditionmaxpikk2}, with high probability, we have
\begin{equation}\label{eq_conditionvar2}
\operatorname{Var}({\Delta}_{i+1}^2\mid \mathcal{F}_{i})\le C(\frac{n-i}{n})^2\mathbb{E}(U_{i+1}^2+V_{i+1}^2\mid \mathcal{F}_i)\lesssim \mathrm{o}(1)(\frac{n-i}{n})^2
\end{equation}
for all $p-p^{1/6}\log^{1/4}p\le i\le p-1$.
Therefore, similarly to \eqref{eq_inductionvar}, applying the induction on $p-p^{2/3}\le i\le p-1$, by the law of total variance and conditional expectation argument, we have
\begin{equation*}
    \begin{split}
        \operatorname{Var}({\Delta}_{i+1}^2\cdots {\Delta}_{p-s_1+1}^2)
        \lesssim
        \mathrm{o}(1)(\mathbb{E}({\Delta}_{i+1}^2\cdots {\Delta}_{p-s_1+1}^2))^2,
    \end{split}
\end{equation*}
by induction for $n-p\ge d_n\rightarrow \infty$, where we used \eqref{eq_conditionvar2} and the induction step, completing the proof of \eqref{eq_diffXX}.

\subsubsection{Edge replacement: The case $0\le n-p<d_n$}\label{sec_edgegaussianreplace}
For the case of $0\le n-p\le d_n$, noticing that $-2\log (1-(p-1)/n)\asymp \log n$. Our overall strategy is to replace lines one by one. Noting that since the magnitude of a determinant
is invariant under swapping of two rows, without loss of generality, we first compare two random matrices $\mathbf{B}$ and $\bar{\mathbf{B}}$ satisfying \eqref{eq_newcondition} such that they only differ in the last row. Assume that $b_{jk}=\bar{b}_{jk}, 1\le j\le p-1, 1\le k\le n$, and $\mathbf{b}_p^{\top}$ and $\bar{\mathbf{b}}_p^{\top}$ are independent, which denote the $p$-th row of $\mathbf{B}$ and $\bar{\mathbf{B}}$, respectively. Moreover, let $\mathbf{y}_p^{\top}=\mathbf{b}_p^{\top}/\|\mathbf{b}_p^{\top}\|$ and $\bar{\mathbf{y}}_p^{\top}=\bar{\mathbf{b}}_p^{\top}/\|\bar{\mathbf{b}}_p^{\top}\|$ be the self-normalized versions, which are the $p$-th row of $\mathbf{Y}$ and $\bar{\mathbf{Y}}$, respectively. We consider the special case $n=p$ first.
It is elementary that
\begin{equation*}
    \det(\mathbf{Y})=\sum_{k=1}^{n}Y_{n,k}B_{nk},~ \det(\bar{\mathbf{Y}})=\sum_{k=1}^{n}\bar{Y}_{n,k}B_{nk},
\end{equation*}
where $B_{nk}$ is the cofactor of $Y_{n,k}$ and $\bar{Y}_{n,k}$. Similarly to \eqref{eq_diffalpha}, it suffices to show
\begin{equation}\label{eq_diff00}
    \frac{\log (\sum_{k=1}^{n}Y_{n,k}B_{nk})^2}{\sqrt{\log n}}-\frac{\log (\sum_{k=1}^{n}\bar{Y}_{n,k}B_{nk})^2}{\sqrt{\log n}}=\mathrm{O}_{\mathbb{P}}(s_2^{-1}\log s_2).
\end{equation}
Set $\tilde{t}_{np}=:\sqrt{\log n}\cdot s_2^{-1}\log s_2\asymp (\log\log n)\rightarrow \infty$ and $\Delta=\sqrt{\sum_{k=1}^n B_{nk}^2}$.
Thus, \eqref{eq_diff00} is equivalent to
\begin{equation}\label{eq_difflarge}
    \begin{split}
    \big(\log (\sum_{k=1}^{n}Y_{n,k}B_{nk})^2-\log(n^{-1}\Delta^2)\big)-\big(\log (\sum_{k=1}^{n}\bar{Y}_{n,k}B_{nk})^2-\log(n^{-1}\Delta^2)\big)
    =\mathrm{O}_{\mathbb{P}}(\tilde{t}_{np}).
    \end{split}
\end{equation}
Similarly to \eqref{eq_diffG}, we have
\begin{equation}\label{eq_berryesseen1}
    \begin{split}
        &\mathbb{P}\left(\left|\log\big(\sum_{k=1}^{n}Y_{n,k}B_{nk}\big)^2-\log (n^{-1}\Delta^2) \right|>\tilde{t}_{np}\right)\\
        =&\mathbb{P}\left(\frac{\sqrt{n}|\sum_{k=1}^{n}Y_{n,k}B_{nk}|}{\Delta}>e^{\tilde{t}_{np}/2}\right)+\mathbb{P}\left(\frac{\sqrt{n}|\sum_{k=1}^{n}Y_{n,k}B_{nk}|}{\Delta}<e^{-\tilde{t}_{np}/2}\right).
    \end{split}
\end{equation}
To deploy the Berry-Esseen bound, we consider the following event
\begin{equation}\label{eq_normevent}
    \mathcal{A}=\{|n^{-1}\|\mathbf{b}_p^{\top}\|^2-1|\ge \epsilon_n\},
\end{equation}
where $\epsilon_n\rightarrow 0$ will be fixed later. Invoking the global truncation level at $n^{2/3}\log n$ and Lemma \ref{lemma_quadratic3}, we have
\begin{equation*}
    \begin{split}
        \mathbb{P}(\mathcal{A})\le \epsilon_n^{-2}\mathbb{E}(n^{-1}\|\mathbf{b}_p^{\top}\|^2-1)^2\lesssim  \epsilon_n^{-2}\mathcal{L}(n^{2/3}\log n)n^{(4-\alpha)2/3-1}\log^2 n
    \end{split}
\end{equation*}
for $\alpha\ge 3$. Thus, we can choose $\epsilon_n=s_2^{-3/2}\sim (\log n)^{-3/4}\rightarrow 0$, which satisfies $\mathbb{P}(\mathcal{A})\lesssim \mathrm{o}(n^{-1/4})$ for $\alpha\ge 3$ by Potter's bound. Thus, consider the quantity
\begin{equation*}
    \sum_{k=1}^{n}Y_{n,k}\frac{\sqrt{n}B_{nk}}{\Delta}=\sum_{k=1}^{n}\frac{\sqrt{n}}{\|\mathbf{b}_p^{\top}\|}b_{n,k}\frac{B_{nk}}{\Delta},
\end{equation*}
which further implies
\begin{equation*}
    \begin{split}
        &\mathbb{P}\left(\frac{\sqrt{n}|\sum_{k=1}^{n}Y_{n,k}B_{nk}|}{\Delta}>e^{\tilde{t}_{np}/2}\right)\le \mathbb{P}\left(\frac{\sqrt{n}|\sum_{k=1}^{n}Y_{n,k}B_{nk}|}{\Delta}>e^{\tilde{t}_{np}/2},\mathcal{A}^c\right)+\mathbb{P}(\mathcal{A})\\
        \le & \mathbb{P}\left(\left|\sum_{k=1}^{n}b_{n,k}\frac{B_{nk}}{\Delta}\right|>e^{\tilde{t}_{np}/2}(1-\epsilon_n)^{1/2}\right)+\mathrm{o}(n^{-1/4}).
    \end{split}
\end{equation*}
Analogously, we can obtain
\begin{equation*}
    \mathbb{P}\left(\frac{\sqrt{n}|\sum_{k=1}^{n}Y_{n,k}B_{nk}|}{\Delta}<e^{-\tilde{t}_{np}/2}\right)\lesssim \mathbb{P}\left(\left|\sum_{k=1}^{n}b_{n,k}\frac{B_{nk}}{\Delta}\right|<e^{-\tilde{t}_{np}/2}(1+\epsilon_n)^{1/2}\right)+\mathrm{o}(n^{-1/4}).
\end{equation*}
For the case of $\alpha>3$, by Lemma \ref{lemma_berryesseen} and following a similar argument of Section 5 in \cite{bao2015logarithmic}, we can get
\begin{equation*}
    \sup_{x}\left|\mathbb{P}\left(\sum_{k=1}^{n}b_{n,k}\frac{B_{nk}}{\Delta}>x\right)-\Phi(x)\right|=\mathrm{o}(1).
\end{equation*}
Specifically, to adapt the Berry-Esseen bound for $\alpha= 3$, we first note
\begin{equation*}
	\mathbb{E}|b_{n,k}|^3\lesssim \mathbb{E}|\xi|^3\mathbbm{1}(|\xi|<n^{2/3}\log n)\lesssim \int_C^{n^{2/3}\log n}\frac{\mathcal{L}(x)}{x}\mathrm{d}x\lesssim \mathcal{L}(n^{2/3}\log n)\log n
\end{equation*}
as $n\rightarrow \infty$, due to the global truncation $n^{2/3}\log n$. Then, following a similar argument as Section 5 in \cite{bao2015logarithmic}, by Lemma \ref{lemma_berryesseen}, we obtain
\begin{equation*}
	\sup_{x}\left|\mathbb{P}\left(\sum_{k=1}^{n}b_{n,k}\frac{B_{nk}}{\Delta}>x\mid \mathcal{F}_{n-1}\right)-\Phi(x)\right|\le C[\mathcal{L}(n^{2/3}\log n)\log n]\mathbb{E}\sum_{k=1}^{n}\frac{B_{nk}^3}{\Delta^3},
\end{equation*}
where $\mathcal{F}_{n-1}=:\sigma(\mathbf{b}_{1},\ldots,\mathbf{b}_{n-1})$ denotes the sigma algebra generated by $\mathbf{b}_{1},\ldots,\mathbf{b}_{n-1}$. By \eqref{eq_Emaxpilllarge}, we get
\begin{equation*}
\mathbb{E}\sum_{k=1}^{n}\frac{B_{nk}^3}{\Delta^3}\le \mathbb{E}\max_{1\le k\le n}\frac{|B_{nk}|}{\Delta}\lesssim  \mathbb{E}\max_{1\le \ell\le n}p_{i,\ell\ell}^{1/2}\lesssim n^{-1/10}
\end{equation*}
with high probability, as $n\rightarrow \infty$.
Therefore, we have
\begin{equation*}
    \mathbb{P}\left(\left|\log\big(\sum_{k=1}^{n}Y_{n,k}B_{nk}\big)^2-\log (n^{-1}\Delta^2) \right|>\tilde{t}_{np}\right)=\mathrm{o}(1)
\end{equation*}
due to the fact that $\tilde{t}_{np}\rightarrow \infty$ and $\epsilon_n\rightarrow 0$, which further implies \eqref{eq_difflarge}.

For the general case where $n-p<d_n$, we aim to prove that, for any $1\le n-p\le d_n$,
\begin{equation}\label{eq_diffpn}
\frac{\log \operatorname{det}\big(\mathbf{Y}_{(p)} \mathbf{Y}_{(p)}^{\top}\big)}{\sqrt{\log n}}-\frac{\log \operatorname{det}\big(\bar{\mathbf{Y}}_{(p)} \bar{\mathbf{Y}}_{(p)}^{\top}\big)}{\sqrt{\log n}}=\mathrm{O}_{\mathbb{P}}\left(s_2^{- 1}\log s_2\right)+\mathrm{O}_{\mathbb{P}}(\log^{-1/2+c}n).
\end{equation}
For any $1\le n-p\le d_n$, let $\tilde{\mathbf{b}}_1^{\top},\ldots,\tilde{\mathbf{b}}_{n-p}^{\top}$ consist of standard Gaussian random variables and $\tilde{\mathbf{y}}_i^{\top}=\tilde{\mathbf{b}}_i^{\top}/\|\tilde{\mathbf{b}}_i^{\top}\|$ for $1\le i\le n-p$. We define two new $n\times n$ matrices based on $\mathbf{B}_{(p)}$ and $\bar{\mathbf{B}}_{(p)}$:
\begin{align*}
\mathbf{C}=\begin{pmatrix}
\mathbf{B}_{(p)} \\
\tilde{\mathbf{b}}_1^{\top} \\
\vdots\\
\tilde{\mathbf{b}}_{n-p}^{\top}
\end{pmatrix}, \quad {\mathbf{D}}=\begin{pmatrix}
\bar{\mathbf{B}}_{(p)} \\
\tilde{\mathbf{b}}_1^{\top} \\
\vdots\\
\tilde{\mathbf{b}}_{n-p}^{\top}
\end{pmatrix},
\end{align*}
where $\mathbf{C}$ and ${\mathbf{D}}$ only differ in the row next to $\tilde{\mathbf{b}}_1^{\top}$.
Denote the associated self-normalized matrices as
\begin{align*}
\mathbf{E}=\begin{pmatrix}
\mathbf{Y}_{(p)} \\
\tilde{\mathbf{y}}_1^{\top}\\
\vdots\\
\tilde{\mathbf{y}}_{n-p}^{\top}
\end{pmatrix}, \quad {\mathbf{F}}=\begin{pmatrix}
\bar{\mathbf{Y}}_{(p)}\\
\tilde{\mathbf{y}}_1^{\top}\\
\vdots\\
\tilde{\mathbf{y}}_{n-p}^{\top}
\end{pmatrix}.
\end{align*}
Since the magnitude of a determinant
is invariant under swapping of two rows, according to \eqref{eq_diff00}, we have
\begin{equation}\label{eq_indutionlarge}
\frac{\log \operatorname{det}\big(\mathbf{E}\mathbf{E}^{\top}\big)}{\sqrt{\log n}}-\frac{\log \operatorname{det}\big(\mathbf{F}\mathbf{F}^{\top}\big)}{\sqrt{\log n}}=\mathrm{O}_{\mathbb{P}}\left(s_2^{-1}\log s_2\right).
\end{equation}
We expand $\mathbf{E}$ and ${\mathbf{F}}$ for the last $n-p$ rows and obtain
\begin{align*}
\log \operatorname{det}\left(\mathbf{E}\mathbf{E}^{\top}\right)=\sum_{i=1}^{n-p}\log\tilde{\Delta}_i^2+\log \operatorname{det}\big(\mathbf{Y}_{(p)} \mathbf{Y}_{(p)}^{\top}\big),~\log \operatorname{det}\left(\mathbf{F}\mathbf{F}^{\top}\right)=\sum_{i=1}^{n-p}\log\hat{\Delta}_i^2+\log \operatorname{det}\big(\bar{\mathbf{Y}}_{(p)} \bar{\mathbf{Y}}_{(p)}^{\top}\big),
\end{align*}
where
\begin{align*}
\tilde{\Delta}_i^2=:\|\tilde{\mathbf{b}}_i^{\top}\|^{-2}\tilde{\mathbf{b}}_i^{\top}\big(\mathbf{I}_n-\mathbf{C}_{(p+i-1)}^{\top}\big(\mathbf{C}_{(p+i-1)} \mathbf{C}_{(p+i-1)}^{\top}\big)^{-1} \mathbf{C}_{(p+i-1)}\big) \tilde{\mathbf{b}}_i,
\end{align*}
and
\begin{equation*}
	\hat{\Delta}_i^2=\|\tilde{\mathbf{b}}_i^{\top}\|^{-2}\tilde{\mathbf{b}}_i^{\top}\big(\mathbf{I}_n-\mathbf{D}_{(p+i-1)}^{\top}\big(\mathbf{D}_{(p+i-1)} \mathbf{D}_{(p+i-1)}^{\top}\big)^{-1} \mathbf{D}_{(p+i-1)}\big) \tilde{\mathbf{b}}_i,.
\end{equation*}
Here, we denote $\mathbf{C}_{(p+i-1)}$ and $\mathbf{C}_{(p+i-1)}$ as the matrices constructed from the first $(p+i-1)$ rows of $\mathbf{C}$ and $\mathbf{D}$, respectively.
Therefore,
\begin{equation*}
\begin{split}
&\frac{\log \operatorname{det}\big(\mathbf{E}\mathbf{E}^{\top}\big)}{\sqrt{\log n}}-\frac{\log \operatorname{det}\big(\mathbf{F}\mathbf{F}^{\top}\big)}{\sqrt{\log n}}\\
=&\frac{\log \operatorname{det}\big(\mathbf{Y}_{(p)} \mathbf{Y}_{(p)}^{\top}\big)}{\sqrt{\log n}}-\frac{\log \operatorname{det}\big(\bar{\mathbf{Y}}_{(p)} \bar{\mathbf{Y}}_{(p)}^{\top}\big)}{\sqrt{\log n}}+\frac{\sum_{i=1}^{n-p}\log\tilde{\Delta}_i^2}{\sqrt{\log n}}-\frac{\sum_{i=1}^{n-p}\log\hat{\Delta}_i^2}{\sqrt{\log n}} .
\end{split}
\end{equation*}
Note $\mathbb{E}(\|\tilde{\mathbf{b}}_i^{\top}\|^{2}\tilde{\Delta}_i^2)=\mathbb{E}(\|\tilde{\mathbf{b}}_i^{\top}\|^{2}\hat{\Delta}_i^2)=n-p-i+1$ for $1\le i\le n-p$. Moreover, since $\tilde{\mathbf{b}}_i^{\top}$ is Gaussian, $\|\tilde{\mathbf{b}}_i^{\top}\|^{2}\tilde{\Delta}_i^2$ and $\|\tilde{\mathbf{b}}_i^{\top}\|^{2}\hat{\Delta}_i^2 \sim \chi_{n-p-i+1}^2$.
Hence, by \eqref{eq_indutionlarge}, in order to show \eqref{eq_diffpn}, it suffices to prove that, for any constant $0<c<1/100$,
\begin{equation}\label{eq_lastdiff}
\frac{\sum_{i=1}^{n-p}\log \frac{\tilde{\Delta}_i^2\|\tilde{\mathbf{b}}_i^{\top}\|^{2}}{n-p-i+1}}{\sqrt{\log n}}-\frac{\sum_{i=1}^{n-p}\log \frac{\hat{\Delta}_i^2\|\tilde{\mathbf{b}}_i^{\top}\|^{2}}{n-p-i+1}}{\sqrt{\log n}}=\mathrm{O}_{\mathbb{P}}(\log^{-1/2+c}n).
\end{equation}
Following Lemma 2.8 of \cite{nguyen2014random} and noticing $1\le n-p\le d_n\lesssim \log^{1/2} n$, one has, for any constant $0<c<1/100$,
\begin{equation*}
	\mathbb{P}\left(\left|\frac{\sum_{i=1}^{n-p}\log \frac{\tilde{\Delta}_i^2\|\tilde{\mathbf{b}}_i^{\top}\|^{2}}{n-p-i+1}}{\sqrt{2\log n}}\right|\ge \log^{-1/2+c}n\right)=\mathrm{o}(\exp(-\log^{c/2}n))
\end{equation*}
and
 \begin{equation*}
 \mathbb{P}\left(\left|\frac{\sum_{i=1}^{n-p}\log \frac{\hat{\Delta}_i^2\|\tilde{\mathbf{b}}_i^{\top}\|^{2}}{n-p-i+1}}{\sqrt{2\log n}}\right|\ge \log^{-1/2+c}n\right)=\mathrm{o}(\exp(-\log^{c/2}n)),
 \end{equation*}
which further implies \eqref{eq_lastdiff}.
Thus, we have shown that for any $1 \leq n-p \leq d_n$,
\begin{align*}
\frac{\log \operatorname{det}\big(\mathbf{Y}_{(p)} \mathbf{Y}_{(p)}^{\top}\big)}{\sqrt{2 \log n}}-\frac{\log \operatorname{det}\big(\bar{\mathbf{Y}}_{(p)} \bar{\mathbf{Y}}_{(p)}^{\top}\big)}{\sqrt{2 \log n}}=\mathrm{O}_{\mathbb{P}}(s_2^{- 1}\log s_2)+\mathrm{O}_{\mathbb{P}}(\log^{-1/2+c}n).
\end{align*}
Then, since $c$ is arbitrary small, which can be choosen as $c=\mathfrak{c}/25>0$, after $d_n=s_2^{1-\mathfrak{c}}\asymp \log^{(1-\mathfrak{c})/2}n$ steps of replacing, we can show that \eqref{eq_diffRR} holds. This completes the proof of the proposition.

\subsection{Mid game: Martingale central limit theorem}\label{sec_midgame}
From Section \ref{sec_gaussianreplace}, it can be seen that the Gaussian replacement for the last $s_1$ rows does not affect the limiting distribution of $\log\det \mathbf{R}$ in the case of $p/n\rightarrow 1$. Without abusing notation, we also use $\mathbf{X}$ to denote the replaced matrix, as well as other associated notation. Keep in mind that all the last $s_1$ rows of $\mathbf{X}$ are standard Gaussian. It suffices to consider
\begin{equation}\label{eq_taylor}
\sum_{i=0}^{p-s_2-1}\log (1+\widetilde{Z}_{i+1})=\sum_{i=0}^{p-s_2-1}\widetilde{Z}_{i+1}-\sum_{i=0}^{p-s_2-1}\frac{\widetilde{Z}_{i+1}^2}{2}+\sum_{i=0}^{p-s_2-1}{E}_{i+1},
\end{equation}
where
\begin{equation}\label{eq_error}
	E_{i+1}=:\log (1+\widetilde{Z}_{i+1})-\big(\widetilde{Z}_{i+1}-\frac{\widetilde{Z}_{i+1}^2}{2}\big).
\end{equation}
Thus, we can rewrite the left hand side of \eqref{eq_martingaleclt} as
\begin{equation*}
\begin{split}
\sum_{i=0}^{p-s_2-1}\widetilde{Z}_{i+1}-\sum_{i=0}^{p-s_2-1}\big(\frac{\widetilde{Z}_{i+1}^2}{2}-\mathbb{E}(\frac{\widetilde{Z}_{i+1}^2}{2}\mid \mathcal{F}_{i})\big)+\sum_{i=0}^{p-s_2-1}{E}_{i+1}-\sum_{i=0}^{p-s_2-1}\mathbb{E}(\frac{\widetilde{Z}_{i+1}^2}{2}\mid \mathcal{F}_{i})+c_n-\mu_n,
\end{split}
\end{equation*}
where $c_n$ and $\mu_n$ are given by \eqref{eq_decomlogdetR} and \eqref{eq_diffGR}, respectively. Here, $\mathcal{F}_{i}=\mathcal{F}_{i}^{(n)}$ denotes the sigma algebra generated by the first $i$ rows of $\mathbf{X}$. In the sequel, invoking $\sigma_{n}^2=(-2\log (1-\frac{p-1}{n})-2\frac{p}{n})^{-1}$, we shall show the following four statements:
\begin{align}
\sigma_{n}\sum_{i=0}^{p-s_2-1}\widetilde{Z}_{i+1}\stackrel{d}{\rightarrow} N(0,1),\label{eq_claim1}\\
\sigma_{n}\sum_{i=0}^{p-s_2-1}\big(\widetilde{Z}_{i+1}^2-\mathbb{E}(\widetilde{Z}_{i+1}^2\mid \mathcal{F}_{i})\big)\stackrel{\mathbb{P}}{\rightarrow} 0,\label{eq_claim2}\\
\sigma_{n}\sum_{i=0}^{p-s_2-1}{E}_{i+1}\stackrel{\mathbb{P}}{\rightarrow} 0,\label{eq_claim3}\\
\sigma_{n}\big(\frac{1}{2}\sum_{i=0}^{p-s_2-1}\mathbb{E}(\widetilde{Z}_{i+1}^2\mid \mathcal{F}_{i})-c_n+\mu_n\big)\stackrel{\mathbb{P}}{\rightarrow} 0,\label{eq_claim4}
\end{align}
which combined with Slutsky's lemma imply \eqref{eq_martingaleclt}. Here, we keep the universal variance formula $-2\log (1-(p-1)/n)-2p/n$ since it concides with the non-singularity case $\lim_{n\rightarrow \infty}p/n<1$.
We will distinguish two cases depending on the value of $n-p$ to prove the above statements.

The first result is the key ingredient to use the martingale central limit theorem.
\begin{lemma}\label{lemma_momentestimateUV}
	Under Assumption \ref{assump1} for $\alpha\ge 3$, we have
	\begin{equation}\label{eq_smallsumq}
	 0\le \sum_{i=0}^{p-s_1-1}\sum_{k=1}^{n}\mathbb{E}(q_{i,kk}^2-\frac{1}{n^2})\le Cn^{-1/2} +C\mathcal{L}(n^{1/2})n^{1-\alpha/2}.
	\end{equation}
	In addition, if $n-p\le n^{19/20}$, we have
	\begin{equation}\label{eq_largesumq1}
	\sum_{i=p-s_1}^{p-s_2-1}\sum_{k=1}^{n}\mathbb{E}q_{i,kk}^2=\mathrm{O}(\log\log n),
	\end{equation}
	and if $n-p\ge n^{19/20}$, we have
	\begin{equation}\label{eq_largesumq2}
		\sum_{i=p-s_1}^{p-s_2-1}\sum_{k=1}^{n}\mathbb{E}q_{i,kk}^2=\mathrm{O}(1).
	\end{equation}
\end{lemma}
\begin{proof}
	We first note that \eqref{eq_smallsumq} is a consequence of Lemma \ref{lemma_diagonals}. For \eqref{eq_largesumq1}, by \eqref{eq_Emaxpll} in Lemma \ref{lemma_upperdiagonal} and the trivial bound $\sum_{k=1}^{n}q_{i,kk}^2\le \operatorname{tr}\mathbf{Q}_i^2=\operatorname{tr}\mathbf{Q}_i/(n-i)=1/(n-i)$, we have
	\begin{equation*}
	\begin{split}
	\sum_{i=p-s_1}^{p-s_2-1}\sum_{k=1}^{n}\mathbb{E}q_{i,kk}^2\le & \sum_{i=p-s_1}^{p-s_3-1}\frac{1}{n-i}+\sum_{i=p-s_3}^{p-s_2-1}\frac{1}{n-i}\mathbb{E}\max_{1\le k\le n}p_{i,kk}\\
	\le & \sum_{i=p-s_1}^{p-s_3-1}\frac{1}{n-i}+\sum_{i=p-s_3}^{p-s_2-1}\frac{1}{n-i}\log^{-a}n\le C\log\log n +C\log^{-a+1}n,
	\end{split}
	\end{equation*}
	where we used $\sum_{k=1}^nq_{i,kk}^2\le \max_{1\le k\le n}q_{i,kk}\sum_{k=1}^{n}q_{i,kk}=\max_{1\le k\le n}q_{i,kk}$.
	It remains to show \eqref{eq_largesumq2} for $n-p\ge n^{19/20}$.
	Similarly to Lemma 3.6 of \cite{wang2018}, we have, if $n/2\le i\le p-s_2$,
	\begin{equation*}
	\mathbb{E}p_{i,kk}^2\le C\frac{1}{(1+n^{-1}\mathbb{E}\operatorname{tr}G_{i}(\epsilon_{n}))^2},
	\end{equation*}
	where $G_{i}(\epsilon_{n})=(n^{-1}\mathbf{B}_i\mathbf{B}_i^{\top}+\epsilon_{n} \mathbf{I}_i)^{-1}$ and $\epsilon_n=n^{-1/7}$. Specifically, following (3-13) to (3-22) in \cite{wang2018} and noticing the global truncation $n^{2/3}\log n$, we can show
	\begin{equation*}
	\begin{split}
	\mathbb{E}p_{i,kk}^2\le& 2\frac{1}{(1+n^{-1}\mathbb{E}\operatorname{tr}G_{i,k}(\epsilon_{n}))^2}+2\mathbb{E}\left[\frac{n^{-1}\mathbf{v}_{k,i}^{\top}G_{i,k}(\epsilon_{n})\mathbf{v}_{k,i}-n^{-1}\mathbb{E}\operatorname{tr}G_{i,k}(\epsilon_{n})}{1+n^{-1}\mathbb{E}\operatorname{tr}G_{i,k}(\epsilon_{n})}\right]^2\\
	\le &C\frac{1}{(1+n^{-1}\mathbb{E}\operatorname{tr}G_{i}(\epsilon_{n}))^2}+C\mathcal{L}(n^{2/3}\log n)(n^{2/3}\log n)^{4-\alpha}\frac{1}{n^2}\mathbb{E}\operatorname{tr}G^2_{i,k}(\epsilon_{n})+\frac{C}{n\epsilon_n^4}\\
	\le & C\frac{1}{(1+n^{-1}\mathbb{E}\operatorname{tr}G_{i}(\epsilon_{n}))^2}+\mathrm{o}(n^{-1/30})
	\end{split}
	\end{equation*}
	for $\epsilon_{n}=n^{-1/7}$ and $\alpha\ge 3$, where in the second line we used Lemma \ref{lemma_quadratic3} and
	\[
	\mathbb{E}v_{kj,i}^4\le \mathbb{E}\xi^4\mathbbm{1}(|\xi|<n^{2/3}\log n)\lesssim \mathcal{L}(n^{2/3}\log n)(n^{2/3}\log n)^{4-\alpha}
	\]
	by Lemma \ref{lemma_truncatemoment}, and in the last line we used $G_{i,k}(\epsilon_{n})\lesssim \epsilon_{n}^{-1}$. By Lemma \ref{lemma_esdconvergence}, we can see that for $n/2\le i\le p-s_2\le n-n^{19/20}$, with high probability,
	\begin{equation*}
	\mathbb{E}\big(\frac{1}{i}\operatorname{tr}G_i(\epsilon_{n})\big)=s_{i}(\epsilon_n)+\mathrm{O}(n^{-1/2}),
	\end{equation*}
	where
	\begin{equation*}
	s_{i}(\epsilon_n)=2\left(\epsilon_n+1-\frac{i}{n}+\sqrt{(\epsilon_n+1-\frac{i}{n})^{2}+\frac{4\epsilon_{n} i}{n}}\right)^{-1}\asymp (1-i/n)^{-1}.
	\end{equation*}
	Thus, similarly to (3-28) of \cite{wang2018}, we get
	\begin{equation*}
	\begin{split}
	\sum_{i=p-s_1}^{p-s_2-1}\mathbb{E}q_{i,kk}^2
	\le C\sum_{i=p-s_1}^{p-s_2-1}\frac{n}{(n-i)^2}\big(\frac{1}{1+\frac{i}{n}\frac{n}{n-i}}\big)^2\le C,
	\end{split}
	\end{equation*}
	finishing the proof of this proposition.
\end{proof}

\subsubsection{Proof of \eqref{eq_claim1}}\label{subsec_proofclaim1}
Notice that $(\widetilde{Z}_{i+1})$ is a martingale difference sequence with respect to the filtration $(\mathcal{F}_i)$ by our construction. Analogously to \cite{heiny2023logdet}, in what follows, we use the CLT for martingale differences to prove the claim \eqref{eq_claim1}. For the reader's convenience, we present a lemma of CLT for martingale difference as follows (e.g., \cite{heiny2023logdet}).
\begin{lemma}\label{lemma_cltmartingalediff}
Let $\{S_{ni},\mathcal{F}_{ni},1\le i\le k_n,n\ge 1\}$ be a zero-mean, square integrable martingale array with difference $Z_{ni}$. Suppose that $\mathbb{E}(\max_{i}Z_{ni}^2)$ is bounded in $n$ and that
\begin{equation*}
	\max_{i}|Z_{ni}|\stackrel{\mathbb{P}}{\rightarrow} 0~\text{and}~\sum_{i}Z_{ni}^2\stackrel{\mathbb{P}}{\rightarrow}1.
\end{equation*}
Then, we have $S_{nk_n}\stackrel{d}{\rightarrow}N(0,1)$.
\end{lemma}
Firstly, $(\widetilde{Z}_{i+1})$ is a martingale difference sequence with respect to the filtration $(\mathcal{F}_i)$. Thus, \eqref{eq_claim1} follows from applying Lemma \ref{lemma_cltmartingalediff} to $\sigma_{n}\widetilde{Z}_{i+1}$ with $\sigma_{n}=(-2\log (1-(p-1)/n)-2p/n)^{-1/2}$.

\noindent\textbf{Case I. $0\le n-p\le n^{19/20}$.}\\
In this case, one has $-\log (1-(p-1)/n)\asymp \log n$, which implies $\sigma_{n}^2\asymp \log^{-1}n$. Let $k_n=p-s_2$. To show $\max_{i=0,\ldots,p-s_2-1}|\sigma_{n}\widetilde{Z}_{i+1}|\stackrel{\mathbb{P}}{\rightarrow}0$, it suffices to show, for any $\epsilon>0$, that
\begin{equation*}
\sum_{i=0}^{p-s_2-1}\mathbb{P}(\sigma_{n}|\widetilde{Z}_{i+1}|>\epsilon)\rightarrow 0,~\text{as}~n\rightarrow \infty.
\end{equation*}
Recalling the decomposition \eqref{eq_decomZ}, one can obtain, for any $\epsilon>0$, that
\begin{equation}\label{eq_inprobability}
\begin{split}
\sum_{i=0}^{p-s_2-1}\mathbb{P}(\sigma_{n}|\widetilde{Z}_{i+1}|>\epsilon)\le & \sum_{i=0}^{p-s_2-1}\mathbb{P}(\sigma_{n}|U_{i+1}|>\epsilon/2)+\sum_{i=0}^{p-s_2-1}\mathbb{P}(\sigma_{n}|V_{i+1}|>\epsilon/2)\\
\lesssim & \frac{4}{\epsilon^2\log n}\sum_{i=0}^{p-s_2-1}\mathbb{E}U_{i+1}^2+\frac{16}{\epsilon^4\log^2 n}\sum_{i=0}^{p-s_2-1}\mathbb{E}V_{i+1}^4,
\end{split}
\end{equation}
which can be implied by
\begin{equation}\label{eq_inproU}
	\frac{1}{\log n}\sum_{i=0}^{p-s_2-1}\mathbb{E}U_{i+1}^2\rightarrow 0
\end{equation}
and
\begin{equation}\label{eq_inproV}
\frac{1}{\log^2 n}\sum_{i=0}^{p-s_2-1}\mathbb{E}V_{i+1}^4\rightarrow 0.
\end{equation}
Consider \eqref{eq_inproV} first. By Lemma \ref{lemma_offdiag}, one has
\begin{equation*}
	\frac{1}{\log^2 n}\sum_{i=0}^{p-s_1-1}\mathbb{E}V_{i+1}^4\le \frac{C}{\log^2 n}\big(\sum_{i=0}^{p-s_1-1}\frac{[\mathcal{L}(n^{1/2})]^2n^{4-\alpha}}{(n-i)^3}+\sum_{i=0}^{p-s_1-1}\frac{1}{(n-i)^2}\big)\le \frac{[\mathcal{L}(n^{1/2})]^2}{s_1\log^2 n}=\mathrm{o}(1)
\end{equation*}
and
\begin{equation*}
\frac{1}{\log^2 n}\sum_{i=p-s_1}^{p-s_2-1}\mathbb{E}V_{i+1}^4\le \frac{C}{\log^2 n}\sum_{i=p-s_1}^{p-s_2-1}\frac{1}{(n-i)^2}\le \frac{C}{s_2\log^2 n}=\mathrm{o}(1)
\end{equation*}
in view of \eqref{eq_newcondition} and $s_1$ in \eqref{eq_def_parameters}.
Now we proceed to the estimation for the variance of $\widetilde{Z}_{i+1}$. Recalling \eqref{eq_meanU2} and \eqref{eq_meanV2}, we have
\begin{equation}\label{eq_usquare}
	\sum_{i=0}^{p-s_2-1}\mathbb{E}(U_{i+1}^2)=\sum_{i=0}^{p-s_1-1}n\frac{n^2\beta_4-1}{n-1}\mathbb{E}\big(S_{2}^{(i)}-n^{-1}\big)+\sum_{i=p-s_1}^{p-s_2-1}n\frac{n^2\kappa_4-1}{n-1}\mathbb{E}\big(S_{2}^{(i)}-n^{-1}\big)
\end{equation}
and
\begin{equation}\label{eq_vsquaresum}
    \sum_{i=0}^{p-s_2-1}\mathbb{E}(V_{i+1}^2)=\sum_{i=0}^{p-s_1-1}2n^2\beta_{2,2}\big(\frac{1}{n-i}-\mathbb{E}S_{2}^{(i)}\big)+\sum_{i=p-s_1}^{p-s_2-1}2n^2\kappa_{2,2}\big(\frac{1}{n-i}-\mathbb{E}S_{2}^{(i)}\big)+\mathrm{o}\big(\frac{\log n}{n}\big),
\end{equation}
where we used $\kappa_{k_1,\ldots,k_r}$ to denote the mixed moments of the self-normalized entry for standard Gaussian random variables, similar to the notation of $\beta_{k_1,\ldots,k_r}$. It is straightforward that
\begin{equation*}
    \kappa_{2k}\sim n^{-k},~\kappa_{2,2}\sim n^{-2}.
\end{equation*}
Invoking the defintion of $\sigma_{n}^2$, combining \eqref{eq_usquare} and \eqref{eq_vsquare}, we get
\begin{equation*}
	\sigma_{n}^2\mathbb{E}\big(\max_{0\le i\le p-1}\widetilde{Z}_{i+1}^2\big)\le \sigma_{n}^2\sum_{i=0}^{p-s_2-1}\mathbb{E}(\widetilde{Z}_{i+1}^2)=\sigma_{n}^2\sum_{i=0}^{p-s_2-1}\big(\mathbb{E}U_{i+1}^2+\mathbb{E}V_{i+1}^2\big)+\mathrm{o}(1),
\end{equation*}
where we used that
\begin{equation}\label{eq_crossUV}
\begin{split}
\sum_{i=0}^{p-s_2-1}\mathbb{E}(U_{i+1}V_{i+1}\mid \mathcal{F}_i)
=&\sum_{i=0}^{p-s_2-1}\big(\sum_{j\ne l}^{n}q_{i,jj}q_{i,jl}n(n\beta_{3,1}-\beta_{1,1})+\sum_{j\ne k\ne l}^{n}q_{i,jj}q_{i,kl}n(n\beta_{2,1,1}-\beta_{1,1})\big)
\\
\lesssim &\sum_{i=0}^{p-s_2-1} n^{-1}\big(\sum_{j\ne l}q_{i,jj}^2\sum_{j\ne l}q_{i,jl}^2\big)^{1/2}+\sum_{i=0}^{p-s_2-1}\sum_{k\ne l}^{n}n^{-2}q_{i,kl}\\
\le& \sum_{i=0}^{p-s_2-1}n^{-1}\big(\frac{1}{(n-i)^2}\big)^{1/2}+\sum_{i=0}^{p-s_2-1}n^{-2}\frac{n}{(n-i)^{1/2}}\le n^{-1/2}\log n
\end{split}
\end{equation}
since $\sum_{j,l}q_{i,jl}^2= 1/(n-i)$, $\sum_{j\ne l}q_{i,jl}\le n/(n-i)^{1/2}$ and $\beta_{3,1}\le \beta_{1,1}=\mathrm{o}(n^{-3}), \beta_{2,1,1}\lesssim n^{-1}\beta_{1,1}=\mathrm{o}(n^{-4})$. Note that the last $s_1$ rows vanish due to the symmetry of the standard Gaussian random variable, which does not affect the estimation above since both terms are negligible as $n\rightarrow \infty$.

Now we consider \eqref{eq_usquare}. For the first part, by $\beta_4= C_{\alpha} n^{-\alpha/2}\mathcal{L}(n^{1/2})$ from Lemma \ref{lem_moment_rates}, we have
\begin{equation}\label{eq_usquaresmall}
	\sum_{i=0}^{p-s_1-1}\mathbb{E}U_{i+1}^2\lesssim \mathcal{L}(n^{1/2})n^{2-\alpha/2}\sum_{i=0}^{p-s_1-1}\mathbb{E}(S_2^{(i)}-n^{-1})\lesssim \mathcal{L}(n^{1/2})n^{(3-\alpha)/2}+[\mathcal{L}(n^{1/2})]^2n^{3-\alpha}
\end{equation}
by \eqref{eq_smallsumq}. While for the second part, note the relation
\begin{equation}\label{eq_meanS2i}
    \mathbb{E}S_2^{(i)}=\mathbb{E}\sum_{k=1}^{n}q_{i,kk}^2\le \mathbb{E}\max_{1\le k\le n}q_{i,kk}=\frac{1}{n-i}\mathbb{E}\max_{1\le k\le n}p_{i,kk},
\end{equation}
since $\sum_{k=1}^{n}q_{i,kk}=1$ and $q_{i,kk}\ge 0$. Thus, by \eqref{eq_largesumq1}, we get
\begin{equation}\label{eq_largeS2}
    \begin{split}
    \sum_{i=p-s_1}^{p-s_2-1}\mathbb{E}U_{i+1}^2= \sum_{i=p-s_1}^{p-s_2-1}n\frac{n^2\kappa_4-1}{n-1}\mathbb{E}\big(S_{2}^{(i)}-n^{-1}\big)\le C\sum_{i=p-s_1}^{p-s_2-1}\mathbb{E}\big(S_{2}^{(i)}-\frac{1}{n}\big)\le C\log\log n,
    \end{split}
\end{equation}
since $\kappa_4\sim 3n^{-2}$, which further implies
\begin{equation}\label{eq_usquare2}
	\frac{1}{\log n}\sum_{i=0}^{p-s_2-1}\mathbb{E}U_{i+1}^2\lesssim \frac{1}{\log n}\mathcal{L}(n^{1/2})n^{(3-\alpha)/2}+\frac{1}{\log n}[\mathcal{L}(n^{1/2})]^2n^{3-\alpha}+\frac{\log\log  n}{\log n}=\mathrm{o}(1)
\end{equation}
by \eqref{eq_newcondition} as $n\rightarrow \infty$.
Moreover, by \eqref{eq_vsquare}, \eqref{eq_largesumq1} and \eqref{eq_meanS2i}, one has
\begin{equation}\label{eq_vsquare2}
	\begin{split}
	\sum_{i=0}^{p-s_2-1}\mathbb{E}V_{i+1}^2=&\sum_{i=0}^{p-s_1-1}2n^2\beta_{2,2}\big(\frac{1}{n-i}-(\mathbb{E}S_{2}^{(i)}-\frac{1}{n})-\frac{1}{n}\big)\\
	&+\sum_{i=p-s_1}^{p-s_2-1}2n^2\kappa_{2,2}\big(\frac{1}{n-i}-(\mathbb{E}S_{2}^{(i)}-\frac{1}{n})-\frac{1}{n}\big)+\mathrm{o}(n^{-1/2})\\
    =&\sum_{i=0}^{p-s_2-1}\big(\frac{2}{n-i}-\frac{2}{n}\big)+C\log\log n
	= (-2\log (1-\frac{p-1}{n})-2\frac{p}{n})(1+\mathrm{o}(1)),
	\end{split}
\end{equation}
where we used $\beta_{2,2}=\frac{1-n\beta_4}{n(n-1)}=n^{-2}+\mathrm{O}(n^{-1-\alpha/2}\mathcal{L}(n^{1/2}))=n^{-2}+\mathrm{o}(n^{-5/2}\log n)$
by \eqref{eq_newcondition} and $ \kappa_{2,2}=\frac{1-n\kappa_4}{n(n-1)}=n^{-2}+\mathrm{O}(n^{-3})$.
Thereafter, we obtain
\begin{equation*}
	\sigma_{n}^2\mathbb{E}\big(\max_{0\le i\le p-s_2-1}\widetilde{Z}_{i+1}^2\big)\le \sigma_{n}^2\sum_{i=0}^{p-s_2-1}\mathbb{E}(\widetilde{Z}_{i+1}^2)= 1+\mathrm{o}(1).
\end{equation*}
It remains to show $\sigma_{n}^2\sum_{i=0}^{p-s_2-1}\widetilde{Z}_{i+1}^2\stackrel{\mathbb{P}}{\rightarrow}1$, which is an immediate consequence of \eqref{eq_claim2} and
\begin{equation*}
	\sigma_n\sum_{i=0}^{p-s_2-1}\big(\mathbb{E}(\widetilde{Z}_{i+1}^2\mid \mathcal{F}_{i})-\mathbb{E}\widetilde{Z}_{i+1}^2\big)\stackrel{\mathbb{P}}{\rightarrow} 0,~\text{as}~n\rightarrow \infty.
\end{equation*}
By \eqref{eq_meanU2}, \eqref{eq_vsquare},  \eqref{eq_usquare}, \eqref{eq_crossUV} and elementary calculations, we have
\begin{equation*}
\begin{split}
\sum_{i=0}^{p-s_2-1}\big(\mathbb{E}(\widetilde{Z}_{i+1}^2 \mid\mathcal{F}_i)-\mathbb{E}\widetilde{Z}_{i+1}^2\big)
=&\sum_{i=0}^{p-s_2-1}\big(\mathbb{E}(U_{i+1}^2+V_{i+1}^2 \mid \mathcal{F}_i)-\mathbb{E}(U_{i+1}^2+V_{i+1}^2)\big) +\mathrm{o}(1)\\
\sim &\sum_{i=0}^{p-s_1-1} \left(\frac{n\big(S_2^{(i)}-\mathbb{E}S_2^{(i)}\big)}{n-1}(1-n^2 \beta_4)+2 n^2 \beta_{2,2}(S_2^{(i)}-\mathbb{E}S_2^{(i)})\right) \\
&+\sum_{i=p-s_1}^{p-s_2-1}\left(\frac{n\big(S_2^{(i)}-\mathbb{E}S_2^{(i)}\big)}{n-1}(1-n^2 \kappa_4)+2 n^2 \kappa_{2,2}(S_2^{(i)}-\mathbb{E}S_2^{(i)})\right)\\
 \sim&\left(3-n^2 \beta_4\right) \sum_{i=0}^{p-s_1-1}(S_2^{(i)}-\mathbb{E}S_2^{(i)})+ \left(3-n^2 \kappa_4\right) \sum_{i=p-s_1}^{p-s_2-1}(S_2^{(i)}-\mathbb{E}S_2^{(i)}).
\end{split}
\end{equation*}
The first part follows from $n^2 \beta_4\sim n^{2-\alpha / 2}\mathcal{L}(n^{1/2})$ by Lemma \ref{lem_moment_rates}, and
\begin{equation*}
\begin{split}
\sum_{i=0}^{p-s_1-1}(S_2^{(i)}-\mathbb{E}S_2^{(i)})= & \sum_{i=0}^{p-s_1-1}(S_2^{(i)}-\frac{1}{n})-\sum_{i=0}^{p-s_1-1}(\mathbb{E}S_2^{(i)}-\frac{1}{n}) \\
= & \sum_{i=0}^{p-s_1-1} \sum_{\ell=1}^n(q_{i, \ell \ell}^2-\frac{1}{n^2})-\sum_{i=0}^{p-s_1-1} \sum_{\ell=1}^n(\mathbb{E}\left[q_{i, \ell \ell}^2\right]-\frac{1}{n^2})\\ =&\mathrm{O}_{\mathbb{P}}(n^{-1 / 2}+\mathcal{L}(n^{1/2})n^{(2-\alpha)/2})+\mathrm{O}(n^{-1/2}+\mathcal{L}(n^{1/2})n^{(2-\alpha)/2}),
\end{split}
\end{equation*}
where in the last line we used the fact that $0 \le \sum_{\ell=1}^n(q_{i, \ell \ell}^2-\frac{1}{n^2})$ by Lemma \ref{lemma_diagonals} and Markov's inequality. For the second part, noting the fact $S_2^{(i)}\ge 1/n$ and applying Markov's inequality, we have
\[
\sum_{i=p-s_1}^{p-s_2-1}S_2^{(i)}=\mathrm{O}_{\mathbb{P}} (\sum_{i=p-s_1}^{p-s_2-1}\mathbb{E}S_2^{(i)})=\mathrm{O}_{\mathbb{P}} (\log\log n)
\]
by \eqref{eq_largeS2}. This further gives
\begin{equation*}
\begin{split}
\sum_{i=p-s_1}^{p-s_2-1}(S_2^{(i)}-\mathbb{E}S_2^{(i)})=\mathrm{O}_{\mathbb{P}}(\log\log n).
\end{split}
\end{equation*}
Therefore, we obtain
\begin{equation*}
	\sigma_n\sum_{i=0}^{p-s_2-1}\big(\mathbb{E}(\widetilde{Z}_{i+1}^2\mid \mathcal{F}_{i})-\mathbb{E}\widetilde{Z}_{i+1}^2\big)\stackrel{\mathbb{P}}{\rightarrow} 0,~\text{as}~n\rightarrow \infty
\end{equation*}
in view of \eqref{eq_newcondition} and $\sigma_{n}^2\sim \log^{-1} n$.

\noindent\textbf{Case II.} $n^{19/20}\le n-p=:n\rho_n$, where $\rho_n=1-p/n\rightarrow 0$.\\
Following a similar argument above, we have
\begin{equation*}
\frac{1}{[-2\log(1-p/n)]^{2}}\sum_{i=0}^{p-s_2-1}\mathbb{E}V_{i+1}^4\le \frac{C}{(-\log\rho_n)^{2}}\sum_{i=0}^{p-s_2-1}\frac{1}{(n-i)^2}\le \frac{C}{s_2 (-\log\rho_n)^{2}}=\mathrm{o}(1).
\end{equation*}
For the $U$-part, by \eqref{eq_largesumq2} and \eqref{eq_usquaresmall}, we have
\begin{equation*}
	\sum_{i=0}^{p-s_1-1}\mathbb{E}U_{i+1}^2\lesssim \mathcal{L}(n^{1/2})n^{(3-\alpha)/2}+[\mathcal{L}(n^{1/2})]^2n^{3-\alpha}
\end{equation*}
and
\begin{equation*}
\sum_{i=p-s_1}^{p-s_2-1}\mathbb{E}U_{i+1}^2\le C\sum_{i=p-s_1}^{p-s_2-1}\sum_{k=1}^n\mathbb{E}q_{i,kk}^2\le C,
\end{equation*}
which together with \eqref{eq_usquaresmall} imply
\begin{equation*}
	\frac{1}{[-2\log(1-p/n)]}\sum_{i=0}^{p-s_2-1}\mathbb{E}U_{i+1}^2\le \frac{\mathcal{L}(n^{1/2})n^{(3-\alpha)/2}+[\mathcal{L}(n^{1/2})]^2n^{3-\alpha}+C}{(-\log \rho_n)}=\mathrm{o}(1)
\end{equation*}
by \eqref{eq_newcondition} and $\rho_n\rightarrow 0$. The rest is similar and thus omitted.

\subsubsection{Proof of \eqref{eq_claim2}}
By \eqref{eq_usquare}, \eqref{eq_crossUV}, and \eqref{eq_vsquare2}, we have
\begin{equation*}
	\begin{split}
	    \sum_{i=0}^{p-s_2-1}\mathbb{E}(\widetilde{Z}_{i+1}^2\mid \mathcal{F}_{i})=&\sum_{i=0}^{p-s_1-1}[(n^2\beta_4-1)(\sum_{u=1}^{n}q_{i,uu}^2-n^{-1})+2n^2\beta_{2,2}\sum_{u\ne v}q_{i,uv}^2]\\
        &+\sum_{i=p-s_1}^{p-s_2-1}[(n^2\kappa_4-1)(\sum_{u=1}^{n}q_{i,uu}^2-n^{-1})+2n^2\kappa_{2,2}\sum_{u\ne v}q_{i,uv}^2]+\mathrm{o}(1).
	\end{split}
\end{equation*}
Noticing the fact that $(\sum_{u=1}^{n}q_{i,uu}^2-n^{-1})\ge 0$, Lemma \ref{lemma_diagonals} tegother with Markov's inequality implies
\begin{equation*}
	\begin{split}
	    \sum_{i=0}^{p-s_1-1}(n^2\beta_4-1)(\sum_{u=1}^{n}q_{i,uu}^2-n^{-1})
        \lesssim \mathcal{L}(n^{1/2})n^{2-\alpha/2}(\mathrm{O}_{\mathbb{P}}(n^{-1/2}+\mathcal{L}(n^{1/2})n^{1-\alpha/2})).
	\end{split}
\end{equation*}
By \eqref{eq_largesumq1} and \eqref{eq_largesumq2}, we have
\begin{equation*}
	\frac{1}{\sqrt{-\log (1-(p-1)/n)}}\sum_{i=p-s_1}^{p-s_2-1}\sum_{u=1}^{n}\mathbb{E}q_{i,uu}^2=\mathrm{O}(\frac{\log\log n}{\sqrt{\log n}})+\mathrm{O}(\frac{C}{\sqrt{-\log \rho_n}})=\mathrm{o}(1)
\end{equation*}
as $\rho_n=1-p/n\rightarrow 0$.
Moreover, one has
\begin{equation*}
    |\beta_{2,2}-\kappa_{2,2}|=\mathrm{o}(n^{-5/2}\log n)
\end{equation*}
for \eqref{eq_newcondition}, which further implies
\begin{equation*}
    \left|\sum_{i=p-s_1}^{p-s_2-1}2n^2\kappa_{2,2}\sum_{u\ne v}q_{i,uv}^2-\sum_{i=p-s_1}^{p-s_2-1}2n^2\beta_{2,2}\sum_{u\ne v}q_{i,uv}^2\right|\le \mathrm{o}(n^{-1/2}\log n) \sum_{i=p-s_1}^{p-s_2-1}\frac{1}{n-i}\le \mathrm{o}(n^{-1/2}\log^2 n)
\end{equation*}
since $\sum_{u\ne v}q_{i,uv}^2\le (n-i)^{-1}$.
Therefore, to show \eqref{eq_claim2}, it suffices to show
\begin{equation}\label{eq_claim21}
	\sigma_{n}\big(\sum_{i=0}^{p-s_2-1}\widetilde{Z}_{i+1}^2-\sum_{i=0}^{p-s_2-1}2n^2\beta_{2,2}\sum_{u\ne v}q_{i,uv}^2\big)\stackrel{\mathbb{P}}{\rightarrow}0,~\text{as}~n\rightarrow \infty.
\end{equation}
On the one hand, recalling the definition of $\widetilde{Z}_{i+1}$, one has
\begin{equation}\label{eq_z2decomp}
\begin{split}
\widetilde{Z}_{i+1}^2
=&(U_{i+1}+V_{i+1})^2=\big[\sum_{u=1}^{n}q_{i,uu}(nY_{i+1,u}^2-1)+\sum_{u\ne v}q_{i,uv}n(Y_{i+1,u}Y_{i+1,v}-\beta_{1,1})\big]^2\\
=&U_{i+1}^2+2\sum_{u\ne v}q_{i,uv}^2n^2(Y_{i+1,u}Y_{i+1,v}-\beta_{1,1})^2\\
&+2\sum_{\substack{u_1\ne v_1,u_2\ne v_2\\ \{u_1,v_1\}\ne \{u_2,v_2\}}}q_{i,u_1v_1}q_{i,u_2v_2}n^2(Y_{i+1,u_1}Y_{i+1,v_1}-\beta_{1,1})(Y_{i+1,u_2}Y_{i+1,v_2}-\beta_{1,1})\\
&+2\sum_{u=1}^{n}q_{i,uu}(nY_{i+1,u}^2-1)\sum_{u\ne v}q_{i,uv}n(Y_{i+1,u}Y_{i+1,v}-\beta_{1,1})\\
=&:U_{i+1}^2+2W_{1}(i)+2W_{2}(i)+2W_{3}(i).
\end{split}
\end{equation}
In what follows, we analyze those terms one by one.
For the first term, from \eqref{eq_usquaresmall}, \eqref{eq_largeS2}, and \eqref{eq_largesumq2}, we obtain, for $n-p\le n^{19/20}$,
\begin{equation}\label{eq_sharp1}
\sigma_{n}\sum_{i=0}^{p-s_2-1}\mathbb{E}U_{i+1}^2\lesssim \frac{\mathcal{L}(n^{1/2})n^{(3-\alpha)/2}+[\mathcal{L}(n^{1/2})]^2n^{3-\alpha}+\log\log n}{\sqrt{\log n}}\lesssim \log^{-2\mathfrak{c}}n=\mathrm{o}(1),
\end{equation}
in view of \eqref{eq_newcondition}, which further implies $\sigma_{n}\sum_{i=0}^{p-s_2-1}U_{i+1}^2=\mathrm{o}_{\mathbb{P}}(1)$ by Markov's inequality.
Noticing the Gaussian repalcement for the last $s_1$ rows, we have, for $0\le i\le p-s_1-1$,
\begin{equation*}
\mathbb{E}U_{i+1}^2\lesssim n^{2}\beta_4 \sum_{u=1}^{n}(\mathbb{E}q_{i,uu}^2-n^{-2})\lesssim \frac{n^{2}\beta_4}{(n-i)},
\end{equation*}
and, for $p-s_1\le i\le p-s_3-1$,
\begin{equation*}
\mathbb{E}U_{i+1}^2\le C\mathbb{E}S_{2}^{(i)}\lesssim C\frac{1}{n-i}\mathbb{E}\max_{1\le k\le n}p_{i,kk} \le \frac{1}{n-i}.
\end{equation*}
From the proof of Lemmas \ref{lemma_upperdiagonal} and \ref{lemma_diagonals}, we have
\begin{equation*}
    \mathbb{E}U_{i+1}^2\le C\mathbb{E}S_{2}^{(i)}\lesssim C\frac{1}{n-i}\mathbb{E}\max_{1\le k\le n}p_{i,kk} \le \frac{1}{n-i}\log^{-a} n
\end{equation*}
for $p-s_3\le i\le p-s_2-1$. Moreover, it is straightforward that
\begin{equation*}
	\mathbb{E}V_{i+1}^2=2n^2\beta_{2,2}\big(\frac{1}{n-i}-\mathbb{E}S_{2}^{(i)}\big)+\mathrm{O}(\frac{n^{-1}}{n-i})\le \frac{4}{n-i}.
\end{equation*}
by \eqref{eq_vsquare}.
By the Cauchy-Schwarz inequality and \eqref{eq_Emaxpll}, we have
\begin{equation*}
	\begin{split}
	\sigma_{n}\sum_{i=0}^{p-s_2-1}\mathbb{E}|W_{3}(i)|\le & C\sigma_{n}\sum_{i=0}^{p-s_2-1}\big(\mathbb{E}[\sum_{u=1}^{n}q_{i,uu}(nY_{i+1,u}^2-1)]^2\big)^{1/2}\big(\mathbb{E}[\sum_{u\ne v}q_{i,uv}n(Y_{i+1,u}Y_{i+1,v}-\beta_{1,1})]^2\big)^{1/2}\\
    \lesssim & \sigma_n\sum_{i=0}^{p-s_1-1}\frac{1}{n-i}\big(\mathbb{E}U_{i+1}^2\big)^{1/2}+\sigma_n^2\sum_{i=p-s_1}^{p-s_2-1}\frac{1}{n-i}\big(\mathbb{E}U_{i+1}^2\big)^{1/2}\\
	\lesssim & \sigma_{n}\sum_{i=0}^{p-s_1-1}\big( \frac{n^{2}\beta_4}{(n-i)^3}\big)^{1/2}+\sigma_n\sum_{i=p-s_1}^{p-s_3-1}\frac{1}{n-i}\frac{1 }{(n-i)^{1/2}}+\sigma_n\sum_{i=p-s_3}^{p-s_2-1}\frac{1}{n-i}\frac{\log^{-a/2}n }{(n-i)^{1/2}}\\
	\lesssim & C\sigma_{n} p[\mathcal{L}(n^{1/2})n^{-1/2-\alpha/4}]+\sigma_{n}\frac{C}{s_3}++\sigma_{n}\log^{-a/2}n=\mathrm{o}(1),
	\end{split}
\end{equation*}
for $\alpha\ge 3$, which together with Markov's inequality implies
$\sigma_{n}\sum_{i=0}^{p-s_2-1}W_{3}(i)=\mathrm{o}_{\mathbb{P}}(1)$. Similarly, one has
\begin{equation*}
\begin{split}
W_{2}(i)=&\sum_{u_1\ne v_1\ne v_2}q_{i,u_1v_1}q_{i,u_1v_2}n^2(Y_{i+1,u_1}Y_{i+1,v_1}-\beta_{1,1})(Y_{i+1,u_1}Y_{i+1,v_2}-\beta_{1,1})\\
&+\sum_{u_1\ne v_1\neq u_2\ne v_2}q_{i,u_1v_1}q_{i,u_2v_2}n^2(Y_{i+1,u_1}Y_{i+1,v_1}-\beta_{1,1})(Y_{i+1,u_2}Y_{i+1,v_2}-\beta_{1,1}),
\end{split}
\end{equation*}
which further gives
\begin{equation*}
	\begin{split}
	\mathbb{E}W_{2}(i)=\sum_{u_1\ne v_1\ne v_2}q_{i,u_1v_1}q_{i,u_1v_2}n^2(\beta_{2,1,1}-\beta_{1,1}^2)+\sum_{u_1\ne v_1\neq u_2\ne v_2}q_{i,u_1v_1}q_{i,u_2v_2}n^2(\beta_{1,1,1,1}-\beta_{1,1}^2)
	\lesssim  \frac{n^{-1}}{n-i}
	\end{split}
\end{equation*}
due to Lemma \ref{lemma_Qbounds}, $\beta_{1,1}\le \mathrm{o}(n^{-3}), \beta_{2,1,1}\le \mathrm{o}(n^{-4})$ and $\beta_{1,1,1,1}\le \mathrm{o}(n^{-5})$ for $\alpha\ge 3$. Moreover, by Lemma \ref{lem_moment_rates}, one has, for $0\le i\le p-s_1-1$,
\begin{equation*}
	\begin{split}
	\mathbb{E}\big(W_{2}(i)\big)^2\lesssim& \mathbb{E}\sum_{u_1\ne v_1\ne v_2}q_{i,u_1v_1}^2q_{i,u_2v_2}^2n^4(\beta_{4,2,2}-\beta_{1,1}\beta_{3,2,1}-\beta_{1,1}^2\beta_{2,2}-\beta_{1,1}^2\beta_{2,1,1}-\beta_{1,1}^4)\\
	&+\mathbb{E}\sum_{u_1\ne v_1\ne u_2\ne v_2}q_{i,u_1v_1}^2q_{i,u_2v_2}^2n^4(\beta_{2,2,2,2}-\beta_{1,1}^2\beta_{2,2}-\beta_{1,1}^2\beta_{1,1,1,1}-\beta_{1,1}^4)\\
	\lesssim &\mathcal{L}(n^{1/2})n^{2-\alpha/2}\mathbb{E}\sum_{u_1\ne v_1\ne v_2}q_{i,u_1v_1}^2q_{i,u_2v_2}^2+\mathbb{E}\sum_{u_1\ne v_1\ne u_2\ne v_2}q_{i,u_1v_1}^2q_{i,u_2v_2}^2\\
	\le &\mathcal{L}(n^{1/2})n^{2-\alpha/2} (n-i)^{-3}+(n-i)^{-2}
	\end{split}
\end{equation*}
since $\sum_{v_1}q_{i,uv_1}^2=q_{i,uu}/(n-i)$ and $\sum_{u,v}q_{i,uv}^2=1/(n-i)$, where the odd terms are negligible by Lemma \ref{lemma_odd_moment}. Recalling the replacement argument, for $p-s_1\le i\le p-1$, we get
\begin{equation*}
	\begin{split}
	\mathbb{E}\big(W_{2}(i)\big)^2\lesssim& \mathbb{E}\sum_{u_1\ne v_1\ne v_2}q_{i,u_1v_1}^2q_{i,u_2v_2}^2n^4\kappa_{4,2,2}
	+\mathbb{E}\sum_{u_1\ne v_1\ne u_2\ne v_2}q_{i,u_1v_1}^2q_{i,u_2v_2}^2n^4\kappa_{2,2,2,2}\\
	\lesssim &C\mathbb{E}\sum_{u_1\ne v_1\ne v_2}q_{i,u_1v_1}^2q_{i,u_2v_2}^2+C\mathbb{E}\sum_{u_1\ne v_1\ne u_2\ne v_2}q_{i,u_1v_1}^2q_{i,u_2v_2}^2\le C(n-i)^{-2}
	\end{split}
\end{equation*}
due to symmetry. Thus, one has
\begin{equation*}
	\begin{split}
	&\sigma_{n}^2\mathbb{E}(\sum_{i=0}^{p-s_2-1}W_{2}(i))^2= \sigma_{n}^2\sum_{i=0}^{p-s_2-1}\mathbb{E}(W_{2}(i))^2+2\sigma_{n}^2\sum_{i\ne j}^{p-s_2-1}\mathbb{E}(W_{2}(i)W_{2}(j))\\
	\le &C\sigma_{n}^2\sum_{i=0}^{p-s_1-1}\big(\frac{\mathcal{L}(n^{1/2})n^{2-\alpha/2}}{(n-i)^3} +\frac{1}{(n-i)^2}\big)+C\sigma_{n}^2\sum_{i=p-s_1}^{p-s_2-1}\frac{1}{(n-i)^2}+2\sigma_{n}^2\sum_{i\ne j}^{p-s_2-1}\frac{n^{-2}}{(n-i)(n-j)}
	\le \mathrm{o}(1)
	\end{split}
\end{equation*}
for $\alpha\ge 3$, where we used that for $i\ne j$, by conditional expectation argument,
\begin{equation*}
	\begin{split}
	\mathbb{E}(W_{2}(i)W_{2}(j))=\mathbb{E}W_{2}(i)\mathbb{E}W_{2}(j)\le C\frac{n^{-2}}{(n-i)(n-j)}.
	\end{split}
\end{equation*}
Therefore, we have $\sigma_{n}^2\sum_{i=0}^{p-s_2-1}W_{2}(i)=\mathrm{o}_{\mathbb{P}}(1)$. Next, considering $W_{1}(i)$, one has
\begin{equation*}
	\begin{split}
	W_1(i)=
	&\sum_{u\ne v}q_{i,uv}^2n^2(Y_{i+1,u}^2Y_{i+1,v}^2-2\beta_{1,1}Y_{i+1,u}Y_{i+1,v}+\beta_{1,1}^2)\\
	=&\sum_{u\ne v}q_{i,uv}^2n^2(Y_{i+1,u}^2Y_{i+1,v}^2-2\beta_{1,1}Y_{i+1,u}Y_{i+1,v})+\mathrm{o}(n^{-4}(n-i)^{-1})
	\end{split}
\end{equation*}
since $\sum_{u\ne v}q_{i,uv}^2\le 1/(n-i)$ and $\beta_{1,1}=\mathrm{o}(n^{-3})$ for $\alpha\ge 3$, where the second term satisfies
\begin{equation*}
	\begin{split}
	&\mathbb{E}\big(\sum_{i=0}^{p-s_2-1}\sum_{u\ne v}q_{i,uv}^2Y_{i+1,u}Y_{i+1,v}\big)^2\\
	\le &\sum_{i=0}^{p-s_2-1}\big(\sum_{u\ne v}\mathbb{E}q_{i,uv}^4\beta_{2,2}+\sum_{u_1\ne v_1\ne u_2}\mathbb{E}q_{i,u_1v_1}^2q_{i,u_2v_1}^2\beta_{2,1,1}+\sum_{u_1\ne v_1\ne u_2\ne v_2}\mathbb{E}q_{i,u_1v_1}^2q_{i,u_2v_2}^2\beta_{1,1}^2\big)\\
	&+2\sum_{i\ne j}^{p-s_2-1}\mathbb{E}\sum_{u\ne v}q_{i,uv}^2\sum_{u\ne v}q_{j,uv}^2\beta_{1,1}^2\\
	\le &\sum_{i=0}^{p-s_2-1}\big(\frac{n^{-2}}{(n-i)^3}+\frac{n^{-4}}{(n-i)^3}+\frac{n^{-6}}{(n-i)^2}\big)+\sum_{i\ne j}^{p-s_2-1}\frac{n^{-6}}{(n-i)(n-j)}\le Cn^{-2} ,
	\end{split}
\end{equation*}
since $q_{i,uv}\le 1/(n-i)$, $\sum_{v}q_{i,uv}^2=q_{i,uu}/(n-i)$. Thus, one has
\begin{equation*}
	\sum_{i=0}^{p-s_2-1}W_1(i)=2\sum_{i=0}^{p-s_2-1}\sum_{u\ne v}q_{i,uv}^2n^2Y_{i+1,u}^2Y_{i+1,v}^2+\mathrm{o}_{\mathbb{P}}(1).
\end{equation*}
Thus, by elementary calculation and conditional expectation argument, we get
\begin{equation*}
	\begin{split}
	&\mathbb{E}\big(\sum_{i=0}^{p-s_2-1}\sum_{u\ne v}q_{i,uv}^2n^2(Y_{i+1,u}^2Y_{i+1,v}^2-\beta_{2,2})\big)^2\\
	\lesssim &\sum_{i=0}^{p-s_1-1}n^4\big(\frac{1}{(n-i)^3}n^{-\alpha}[\mathcal{L}(n^{1/2})]^2+\frac{1}{(n-i)^3}n^{-2-\alpha/2}\mathcal{L}(n^{1/2})+\frac{1}{(n-i)^2}n^{-4}\big)\\
    &+\sum_{i=p-s_1}^{p-s_2-1}n^4\big(\frac{1}{(n-i)^3}n^{-4}+\frac{1}{(n-i)^3}n^{-4}+\frac{1}{(n-i)^2}n^{-4}\big)\\
	\lesssim &n(n-p+s_1)^{-2}[\mathcal{L}(n^{1/2})]^2+(n-p+s_1)^{-1}\mathcal{L}(n^{1/2})+(n-p+s_2)^{-1}+s_2^{-1}\le \mathrm{o}(1)
	\end{split}
\end{equation*}
for $\alpha\ge 3$ by Lemma \ref{lem_moment_rates} and the estimates $q_{i,uv}\le 1/(n-i)$, $\sum_{v}q_{i,uv}^2=q_{i,uu}/(n-i)$ and $\sum_{u,v}q_{i,uv}^2=1/(n-i)$. This together with Markov's inequality gives
\begin{equation*}
	\sigma_n\sum_{i=0}^{p-s_2-1}(W_{1}(i)-2n^2\beta_{2,2}\sum_{u\ne v}q_{i,uv}^2)\stackrel{\mathbb{P}}{\rightarrow} 0,
\end{equation*}
which combined with \eqref{eq_z2decomp} and the above estimates completes the proof of \eqref{eq_claim21}, and thus \eqref{eq_claim2} as desired.

\subsubsection{Proof of \eqref{eq_claim3}}
Unlike \cite{bao2015logarithmic} and \cite{wang2018}, we have to distinguish three cases, which cover the case $\lim_{n\rightarrow \infty}p/n<1$ in our setting \eqref{eq_oldcondition}.
\begin{lemma}
\label{lemma_error}
For $E_{i+1}$ defined in \eqref{eq_error} and $a>0$,
\begin{itemize}
\item[(i)] if $1\le n-p\le n^{19/20}$ and $\widetilde{Z}_{i+1}\ge -1+(\log\log n)^{-a/2}$, we have
\begin{equation*}
|E_{i+1}|\le C(U_{i+1}^2+|V_{i+1}|^{2+\delta})\log \log n,
\end{equation*}
\item[(ii)] if $ n-p\ge n^{19/20}$, $\rho_n=1-p/n\rightarrow 0$ and $\widetilde{Z}_{i+1}\ge -1+(\log \log\rho_n^{-1})^{-1}$, we have
\begin{equation*}
|E_{i+1}|\le C(U_{i+1}^2+|V_{i+1}|^{2+\delta})\log \log \rho_n^{-1},
\end{equation*}
\item[(iii)] if $\lim_{n\rightarrow \infty}p/n<1$ and $\widetilde{Z}_{i+1}\ge -1+1/1000$, we have
$|E_{i+1}|\le C(U_{i+1}^2+|V_{i+1}|^{2+\delta})$;
\end{itemize}
for any $0< \delta\le 1$, where $C:=C(a,\delta)$ is a positive constant that only depends on $a$ and $\delta$.
\end{lemma}
\begin{proof}
	The proof is similar to that of Lemma 4.1 in \cite{bao2015logarithmic}. The only difference is the different lower bounds for $\widetilde{Z}_{i+1}+1$. So we omit the details
	of the proof.
\end{proof}
\begin{lemma}\label{lemma_conditionprobability}
    Under the setting of Theorem \ref{thm_logdet} or Theorem \ref{thm_pn},
    \begin{itemize}
    	\item[(i)] if $1\le n-p\le n^{19/20}$, we have
    	 $\sum_{i=0}^{p-s_2-1}\mathbb{P}(\widetilde{Z}_{i+1}< -1+(\log\log n)^{-a/2})\rightarrow 0$;
    	\item[(ii)] if $ n-p\ge n^{19/20}$ and $\rho_n=1-p/n\rightarrow 0$, we have
    	$\sum_{i=0}^{p-s_2-1}\mathbb{P}(\widetilde{Z}_{i+1}< -1+(\log\log \rho_n^{-1})^{-1})\rightarrow 0$;
    	\item[(iii)] if $\lim_{n\rightarrow \infty}p/n<1$, we have
    	$\sum_{i=0}^{p-s_2-1}\mathbb{P}(\widetilde{Z}_{i+1}< -1+1/1000)\rightarrow 0$.
    \end{itemize}
\end{lemma}
\begin{proof}
	Recall $\widetilde{Z}_{i+1}+1=(U_{i+1}+1)+V_{i+1}$ by \eqref{eq_decomZ}.
	If $1\le n-p\le n^{19/20}$, note that
	\begin{equation*}
	\begin{split}
	&\mathbb{P}(\widetilde{Z}_{i+1}< -1+(\log\log n)^{-a/2})=\mathbb{P}(U_{i+1}+1+V_{i+1}<(\log\log n)^{-a/2})\\
	\le &\mathbb{P}(|U_{i+1}+1|<2(\log\log n)^{-a/2})+\mathbb{P}(|V_{i+1}|\ge\frac{1}{2}(\log\log n)^{-a/2})\\
	\le & \mathbb{P}(|n\mathbf{y}_{i+1}^{\top}(\operatorname{diag}\mathbf{Q}_i)\mathbf{y}_{i+1}|<2(\log\log n)^{-a/2})+C(\log\log n)^{2a}\mathbb{E}V_{i+1}^4
	\end{split}
	\end{equation*}
	by Markov's inequality. By Lemma \ref{lemma_offdiag} and \eqref{eq_newcondition}, we have
	\begin{equation*}
		(\log\log n)^{2a}\sum_{i=0}^{p-s_2-1}\mathbb{E}V_{i+1}^4\le (\log\log n)^{2a}\big(\sum_{i=0}^{p-s_1-1}\frac{[\mathcal{L}(n^{1/2})]^2n^{4-\alpha}}{(n-i)^3}+\sum_{i=0}^{p-s_2-1}\frac{1}{(n-i)^2}\big)=\mathrm{o}(1)
	\end{equation*}
	since $s_1\asymp n$ and $s_2\asymp (\log n)^{1/2}$ for $1\le n-p\le n^{19/20}$.
	It suffices to consider the $U$-part.
	First, by Lemma \ref{lemma_highpronorm}, we have, with high probability, for $0\le i\le p-s_1-1$, $\|\mathbf{b}_{i+1}\|^2\ge cn$
	for some constant $c>0$. This suggests that
	\begin{equation*}
		\sum_{i=0}^{p-s_1-1}\mathbb{P}(|n\mathbf{y}_{i+1}^{\top}(\operatorname{diag}\mathbf{Q}_i)\mathbf{y}_{i+1}|<2(\log\log n)^{-a/2})\lesssim \sum_{i=0}^{p-s_1-1}\mathbb{P}(|\mathbf{b}_{i+1}^{\top}(\operatorname{diag}\mathbf{Q}_i)\mathbf{b}_{i+1}|<C(\log\log n)^{-a/2}).
	\end{equation*}
	Denote
	\begin{equation*}
		\hat{b}_{jk}=b_{jk}\mathbbm{1}(|b_{jk}|<(\log\log n)^2),~~\tilde{b}_{jk}=\frac{\hat{b}_{jk}-\mathbb{E}\hat{b}_{jk}}{\sqrt{\operatorname{Var}(\hat{b}_{jk})}}.
	\end{equation*}
	It follows that
	\begin{equation*}
		\mathbb{E}\hat{b}_{jk}=\mathrm{O}((\log\log n)^{2(1-\alpha)}\mathcal{L}((\log\log n)^2)),~~\operatorname{Var}(\hat{b}_{jk})=1+\mathrm{O}((\log\log n)^{2(2-\alpha)}\mathcal{L}((\log\log n)^2))
	\end{equation*}
	by Lemma \ref{lemma_truncatemoment}. Consequently, $\tilde{b}_{jk}=\hat{b}_{jk}+(\log\log n)^{-2}$,
	which implies that
	\begin{equation*}
		\tilde{b}_{jk}^2\le 2\hat{b}_{jk}+(\log\log n)^{-2}
	\end{equation*}
	for sufficiently large $n$. Therefore,
\begin{equation*}
	\begin{split}
	&\mathbb{P}(|\mathbf{b}_{i+1}^{\top}(\operatorname{diag}\mathbf{Q}_i)\mathbf{b}_{i+1}|<C(\log\log n)^{-a/2})\le \mathbb{P}(\sum_{k=1}^{n}q_{i,kk}\hat{b}_{i+1,k}^2<C(\log\log n)^{-a/2})\\
	\le &\mathbb{P}(\sum_{k=1}^{n}q_{i,kk}\tilde{b}_{i+1,k}^2<C(\log\log n)^{-a/2})
	\le \mathbb{P}(|\sum_{k=1}^{n}q_{i,kk}\tilde{b}_{i+1,k}^2-1|\ge 1/2),
	\end{split}
\end{equation*}
which together with Lemma \ref{lemma_4momentG} yields that
\begin{equation*}
\begin{split}
\mathbb{P}(|\mathbf{b}_{i+1}^{\top}(\operatorname{diag}\mathbf{Q}_i)\mathbf{b}_{i+1}|<C(\log\log n)^{-a/2})\le &C\log\log n)^{14}\operatorname{tr}\mathbf{Q}_i^{4}+C((\log\log n)^{4}\operatorname{tr}\mathbf{Q}_i^2)^2\\
\le& C(\log\log n)^{14}\frac{1}{(n-i)^3}+C(\log\log n)^{8}\frac{1}{(n-i)^2},
\end{split}
\end{equation*}
where we used
\begin{equation*}
	\mathbb{E}\tilde{b}_{jk}^8\le \mathbb{E}\hat{b}_{jk}^8\lesssim (\log\log n)^{14},~\mathbb{E}\tilde{b}_{jk}^4\le \mathbb{E}\hat{b}_{jk}^4\lesssim (\log\log n)^{4}
\end{equation*}
by Lemma \ref{lemma_truncatemoment}. Thus, we get
\begin{equation*}
\begin{split}
&\sum_{i=0}^{p-s_1-1}\mathbb{P}(|n\mathbf{y}_{i+1}^{\top}(\operatorname{diag}\mathbf{Q}_i)\mathbf{y}_{i+1}|<2(\log\log n)^{-a/2})\\
\lesssim &C\sum_{i=0}^{p-s_1-1}\big((\log\log n)^{14}\frac{1}{(n-i)^3}+(\log\log n)^{8}\frac{1}{(n-i)^2}\big)\le \frac{C(\log\log n)^{14}}{s_1}\rightarrow 0.
\end{split}
\end{equation*}
Note that the last $s_1$ rows of $\mathbf{X}$ consist of standard Gaussian. It follows that, for $p-s_1\le i\le p-s_2-1$,
\begin{equation*}
	\begin{split}
	&\mathbb{P}(|n\mathbf{y}_{i+1}^{\top}\operatorname{diag}\mathbf{Q}_i\mathbf{y}_{i+1}|<2(\log\log n)^{-a/2})=\mathbb{P}(|\frac{n}{\|\mathbf{b}_{i+1}\|^2}\mathbf{b}_{i+1}^{\top}\operatorname{diag}\mathbf{Q}_i\mathbf{b}_{i+1}|<2(\log\log n)^{-a/2})\\
	\le &\mathbb{P}(|\frac{n}{\|\mathbf{b}_{i+1}\|^2}\mathbf{b}_{i+1}^{\top}\operatorname{diag}\mathbf{Q}_i\mathbf{b}_{i+1}-1|\ge 1/2)\le C (\operatorname{tr}\mathbf{Q}_i^4+(\operatorname{tr}\mathbf{Q}_i^2)^2)\le \frac{C}{(n-i)^2},
	\end{split}
\end{equation*}
where we used the independence of $\|\mathbf{b}_{i+1}\|$ and $\mathbf{b}_{i+1}/\|\mathbf{b}_{i+1}\|$. This further implies
\begin{equation*}
\sum_{i=0}^{p-s_2-1}\mathbb{P}(\widetilde{Z}_{i+1}< -1+(\log\log n)^{-a/2})\rightarrow 0
\end{equation*}
with high probability, as $n\rightarrow \infty$.

The cases of $n-p\ge n^{19/20}$ and $\lim_{n\rightarrow \infty}p/n<1$ are similar, where the only difference is that ``$\log\log n$'' is replaced by ``$\log\log \rho_n^{-1}$'' and $1000$, respectively. So we omit the details of the proof.
\end{proof}

Now we proceed to \eqref{eq_claim3}. Consider the case $0\le n-p\le n^{19/20}$ first.
Invoking Lemma \ref{lemma_momentestimateUV}, Lemma \ref{lemma_error} and Lemma \ref{lemma_conditionprobability} above, we have with probability $1-\mathrm{o}(1)$,
\begin{equation*}
	\sum_{i=0}^{p-s_2-1}E_{i+1}\le C\sum_{i=0}^{p-s_2-1}(U_{i+1}^2+|V_{i+1}|^{2+\delta})\log \log n,~0< \delta\le 1.
\end{equation*}
Thus to show \eqref{eq_claim3}, it suffices to verify that
\begin{equation*}
	\sigma_{n}\lim_{n\rightarrow \infty}\log \log n \sum_{i=0}^{p-s_2-1}(\mathbb{E}U_{i+1}^2+\mathbb{E}|V_{i+1}|^{2+\delta})=0,
\end{equation*}
which can be implied by
\begin{equation*}
\lim_{n\rightarrow \infty}\frac{\log \log n}{\sqrt{\log n}} \sum_{i=0}^{p-s_2-1}[\mathbb{E}U_{i+1}^2+(\mathbb{E}|V_{i+1}|^4)^{(2+\delta)/4}]=0,
\end{equation*}
in view of Lyapunov's inequality. For the $U$-part, by \eqref{eq_usquare2}, one has
\begin{equation}\label{eq_sharp2}
	\begin{split}
	\frac{\log \log n}{\sqrt{\log n}}\sum_{i=0}^{p-s_2-1}\mathbb{E}U_{i+1}^2\lesssim& \frac{\log \log n}{\sqrt{\log n}}\big(\mathcal{L}(n^{1/2})n^{(3-\alpha)/2}+[\mathcal{L}(n^{1/2})]^2n^{3-\alpha}+\log\log  n\big)\\
	\le &\frac{\log\log n}{\log^{2\mathfrak{c}} n}=\mathrm{o}(1)
	\end{split}
\end{equation}
in view of \eqref{eq_newcondition}. For the $V$-part, according to Lemma \ref{lemma_offdiag}, one has
\begin{equation*}
	\begin{split}
	&\log \log n\sum_{i=0}^{p-s_2-1}(\mathbb{E}|V_{i+1}|^4)^{(2+\delta)/4}\\
	\le &\log \log n\sum_{i=0}^{p-s_1-1}\big(\frac{[\mathcal{L}(n^{1/2})]^2n^{4-\alpha}}{(n-i)^3}\big)^{\frac{2+\delta}{4}}+\log \log n\sum_{i=0}^{p-s_2-1}\big(\frac{1}{(n-i)^2}\big)^{\frac{2+\delta}{4}}\\
	\lesssim &\log \log n\sum_{i=0}^{p-s_1-1}\frac{[\mathcal{L}(n^{1/2})]^{1+\delta/2}}{n^{1+\delta/2}}+\log \log n \sum_{i=0}^{p-s_2-1}\frac{1}{(n-i)^{1+\delta/2}}\\
	\le & n^{-\delta/2}[\mathcal{L}(n^{1/2})]^{1+\delta/2}\log \log n+(n-p+s_2)^{-\delta/2}\log \log n=\mathrm{o}(1)
	\end{split}
\end{equation*}
as $n\rightarrow \infty$ with $s_1\asymp n$ and $s_2\asymp \log^{1/2} n$ for some small constant $c>0$, completing the proof of \eqref{eq_claim3}. The cases of $n-p\ge n^{19/20}$ and $\lim_{n\rightarrow \infty}p/n<1$ follows similarly by applying Lemma \ref{lemma_error} and Lemma \ref{lemma_conditionprobability}, respectively. This completes the proof of \eqref{eq_claim3}.

\subsubsection{Proof of \eqref{eq_claim4}}
Recalling that we have shown that
\begin{equation*}
	\sigma_n\sum_{i=0}^{p-s_2-1}\big(\mathbb{E}(\widetilde{Z}_{i+1}^2 \mid\mathcal{F}_i)-\mathbb{E}\widetilde{Z}_{i+1}^2\big) \stackrel{\mathbb{P}}{\rightarrow} 0,~\text{as}~n\rightarrow \infty,
\end{equation*}
it remains to show
\begin{equation}\label{eq_cmu}
	\sigma_{n}\big(-\log (1-\frac{p}{n})-\frac{p}{n}-c_n+\mu_n\big) \rightarrow 0,~\text{as}~n\rightarrow \infty,
\end{equation}
where we used \eqref{eq_crossUV}, \eqref{eq_usquare2}, and \eqref{eq_vsquare2}.
Recalling the definitions $\mu_n=(p-n+\frac{1}{2}) \log (1-\frac{p-1}{n})-p+\frac{p}{n}$ and $c_n=-p \log n+\log (n(n-1) \cdots(n-p+1))$, it suffices to show
\begin{equation*}
\sigma_{n}\left(\left(p-n-\frac{1}{2}\right) \log (1-\frac{p-1}{n})-p-\sum_{i=1}^{p-1} \log (1-i / n)\right) \rightarrow 0, ~\text{as}~ n \rightarrow \infty.
\end{equation*}
By elementary calculations, we have
\begin{equation*}
\begin{split}
\sum_{i=1}^p \log (1-i / n)=& (p-1) \log n-\log (n-1)!+\log (n-p)!\\
=&(p-1) \log n-(n-1) \log (n-1)+(n-1)-\frac{\log (2 \pi(n-1))}{2}+\mathrm{O}(n^{-1})\\
&+(n-p) \log (n-p)-(n-p)+\frac{\log (2 \pi(n-p))}{2}+\mathrm{O}((n-p)^{-1}) \\
=&p-1+\big(n-\frac{1}{2}\big) \log \big(\frac{n}{n-1}\big)+\big(n-p+\frac{1}{2}\big) \log \big(1-\frac{p-1}{n}\big)+\mathrm{O}((n-p)^{-1})\\
=&p-\big(n-p+\frac{1}{2}\big) \log \big(1-\frac{p-1}{n}\big)+\mathrm{O}((n-p)^{-1}),
\end{split}
\end{equation*}
where we used Stirling's formula $\log (n!)=n \log n-n+\frac{1}{2} \log (2 \pi n)+\mathrm{O}(n^{-1})$ in the second line, and the fact that  $-1+(n-\frac{1}{2}) \log (\frac{n}{n-1})\rightarrow 0$ as $n \rightarrow \infty$. Thus, we finish the proof of \eqref{eq_claim4}.

\subsection{End game: Last few rows are negligible}
Firstly, we note that the last $s_2$ rows of $\mathbf{X}$ consist of standard Gaussian after the replacement in Step 2. The proof is inspired by \cite{bao2015logarithmic,nguyen2014random,wang2018}, who used the fact that the last $s_2$ quadratic forms $\mathbf{b}_{i+1}^{\top}\mathbf{Q}_i\mathbf{b}_i$ are independent Chi-square random variables. However, it is not the case for sample correlation matrices due to the self-normalization. Fortunately, the independence between $\mathbf{b}_{i+1}/\|\mathbf{b}_{i+1}\|$ and $\|\mathbf{b}_{i+1}\|$ for standard Gaussian random vectors $\mathbf{b}_{i+1}$ guarantees the desired result. Specifically, we aim to show
\begin{equation*}
    \mathbb{P}\left(\left|\sum_{i=p-s_2}^{p-1}\frac{\log (\frac{n\mathbf{y}_{i+1}^{\top}\mathbf{P}_i\mathbf{y}_{i+1}}{n-i})}{\sqrt{-2\log (1-(p-1)/n)-2p/n}}\right|\ge \log ^{-1/2+c} n\right)=\mathrm{o}(\exp(-\log^{c/2} n))
\end{equation*}
for some constant $0<c<1/100$. Let us first consider the lower tail; it suffices to show
\begin{equation*}
    \mathbb{P}\left(\sum_{i=p-s_2}^{p-1}\frac{\log (\frac{n\mathbf{y}_{i+1}^{\top}\mathbf{P}_i\mathbf{y}_{i+1}}{n-i})}{\sqrt{-2\log (1-(p-1)/n)-2p/n}}\le -\log ^{-1/2+c} n\right)=\mathrm{o}(\exp(-\log^{c/2} n)).
\end{equation*}
For the case of $p-s_2\le i\le p-1$, by the Chernoff bounds, it is easy to show that, for any $0<\epsilon<1/2$
\begin{equation*}
    \mathbb{P}\big(|\frac{\|\mathbf{b}_{i+1}\|^2}{n}-1|>\epsilon\big)\le 2\exp(-n\epsilon^2/8).
\end{equation*}
Thus, similary to the result of \cite{nguyen2014random}, as $b_{ij}$ are i.i.d. standard Gaussian, one can verify that $\mathbf{b}_{i+1}^{\top}\mathbf{P}_i\mathbf{b}_{i+1}$ are independent Chi-square random variables of degree $n-i$. Thus, by Lemma \ref{lemma_highpronorm}, we can follow Section 7 of \cite{nguyen2014random} to show that
\begin{equation*}
    n\mathbf{y}_{i+1}^{\top}\mathbf{P}_i\mathbf{y}_{i+1}=\frac{n}{\|\mathbf{b}_{i+1}\|^2}\mathbf{b}_{i+1}^{\top}\mathbf{P}_i\mathbf{b}_{i+1}\asymp \mathbf{b}_{i+1}^{\top}\mathbf{P}_i\mathbf{b}_{i+1}
\end{equation*}
is at least $\exp(-\frac{\sqrt{2}}{4}\log^c n)$ with probability $1-\exp(-\Omega(\log^{c} n))$ for $i=p-1$ and $i=p-2$. So we can omit these terms from the sum. It now suffices to show that
\begin{equation*}
    \mathbb{P}\left(\sum_{i=p-s_2}^{p-3}\frac{\log (\frac{n\mathbf{y}_{i+1}^{\top}\mathbf{P}_i\mathbf{y}_{i+1}}{n-i})}{\sqrt{2\log n}}\le -\frac{1}{2}\log ^{-1/2+c} n\right)=\mathrm{o}(\exp(-\log^{c/2} n)).
\end{equation*}
where we notice the fact that $-\log (1-(p-1)/n)\asymp \log n$ for $ n-n^{\theta}\le p\le n$ with any constant $0<\theta<1$.
Flipping the inequality inside the probability (by changing the sign of the RHS and swapping the denominators and numerators in the logarithms of the LHS) and using the Laplace transform trick (based on the conditional argument), we see that the probability in question is at most
\begin{equation*}
    \frac{\mathbb{E}\prod_{i=p-s_2}^{p-3}\frac{n-i}{n\mathbf{y}_{i+1}^{\top}\mathbf{P}_i\mathbf{y}_{i+1}}}{\exp((1/\sqrt{2})\log^c n)}=\frac{\prod_{i=p-s_2}^{p-3}\mathbb{E}\left(\frac{n-i}{n\mathbf{y}_{i+1}^{\top}\mathbf{P}_i\mathbf{y}_{i+1}}\mid \mathcal{F}_i\right)}{\exp((1/\sqrt{2})\log^c n)}.
\end{equation*}
By the independence between $\mathbf{b}_{i+1}/\|\mathbf{b}_{i+1}\|$ and $\|\mathbf{b}_{i+1}\|$ for standard Gaussian random vectors $\mathbf{b}_{i+1}$, we have
\begin{equation}\label{eq_lasts2}
    \mathbb{E}\left(\frac{n-i}{n\mathbf{y}_{i+1}^{\top}\mathbf{P}_i\mathbf{y}_{i+1}}\mid \mathcal{F}_i\right)=\mathbb{E}\left(\frac{(n-i)\|\mathbf{b}_{i+1}^2\|}{n\mathbf{b}_{i+1}^{\top}\mathbf{P}_i\mathbf{b}_{i+1}}\mid \mathcal{F}_i\right)=\mathbb{E}\left(\frac{(n-i)n}{n\mathbf{b}_{i+1}^{\top}\mathbf{P}_i\mathbf{b}_{i+1}}\mid \mathcal{F}_i\right)=\frac{n-i}{n-i-2},
\end{equation}
where the last equality follows from the fact that $\mathbf{b}_{i+1}^{\top}\mathbf{P}_i\mathbf{b}_{i+1}$ is a Chi-square random variable with degree of freedom $n-i$, which satisfies $\mathbb{E}(\mathbf{b}_{i+1}^{\top}\mathbf{P}_i\mathbf{b}_{i+1})^{-1}=(n-i-2)^{-1}$. Thus, the numerator in \eqref{eq_lasts2} is
\begin{equation*}
    \frac{(n-p+s_2)(n-p+s_2-1)}{(n-p+1)(n-p+2)}\le Cs_2^2 \le \log ^{2} n,
\end{equation*}
which further implies the desired result since
\begin{equation*}
    \frac{\log ^{2} n}{\exp((1/\sqrt{2})\log^c n)}=\mathrm{o}\big(\exp(-\log^{c/2} n)\big)
\end{equation*}
as $n\rightarrow \infty$. The proof for the upper tail is similar, so we omit the details. Thereafter, we have finished the proof.

\appendix
\section{Preliminary results}\label{app_pre}
This section collects some auxiliary lemmas.
Recall the definition of a high probability event in Definition \ref{def_highpro}.
First, we introduce a label matrix associated with some local truncated levels, which is inspired by the resampling strategy (see \cite{aggarwal2018goe,bao2023phase,Li2024clt} for more details). Intuitively, the resampling of $X_{ij}$ implies that most of the entries are not very large, while there are $\mathrm{o}(n)$ entries larger than $n^{c_{\alpha}}$ for some prespecified truncation levels $c_{\alpha}>0$, which will be fixed within suitable cases. Consider the following label matrix $\Psi=:(\psi_{ij})$ with
\begin{equation*}
\psi_{ij}=\mathbbm{1}(|X_{ij}|>n^{c_{\alpha}}).
\end{equation*}
One has, for $\alpha>2$,
\begin{equation*}
\mathbb{P}(\psi_{ij}=1)=\mathbb{P}(|X_{ij}|>n^{c_{\alpha}})\sim \mathcal{L}(n^{c_{\alpha}})n^{-\alpha c_{\alpha}}.
\end{equation*}
Note tha fact that $(\log\log n)^{\log n}\gg n^D$ for any large fixed $D>0$ as $n\rightarrow\infty$.
For the case of $0<c_{\alpha}\le 1/\alpha$, we have for some small constant $\epsilon_{\alpha}>0$,
\begin{equation*}
\begin{split}
\mathbb{P}(\sum_{j}\psi_{ij}>n^{1-\alpha c_{\alpha}+\epsilon_{\alpha}})\lesssim &\sum_{k=n^{1-\alpha c_{\alpha}+\epsilon_{\alpha}}}^{n}\binom{n}{k} [\mathcal{L}(n^{c_{\alpha}})n^{-\alpha c_{\alpha}}]^{k}[1-\mathcal{L}(n^{c_{\alpha}})n^{-\alpha c_{\alpha}}]^{n-k}\\
\lesssim &\sum_{k=n^{1-\alpha c_{\alpha}+\epsilon_{\alpha}}}^{n}\big(C\frac{n^{1-\alpha c_{\alpha}}\mathcal{L}(n^{c_{\alpha}})}{k}\big)^k\lesssim n^{-D}
\end{split}
\end{equation*}
for any large fixed $D>0$, where we used the elementary inequality $\binom{n}{k}\lesssim (Cn/k)^k$ in the second line. For the case of $1/\alpha<c_{\alpha}$, we have
\begin{equation*}
\begin{split}
\mathbb{P}(\sum_{j}\psi_{ij}>\log n)\lesssim &\sum_{k=\log n}^{n}\binom{n}{k} [\mathcal{L}(n^{c_{\alpha}})n^{-\alpha c_{\alpha}}]^{k}[1-\mathcal{L}(n^{c_{\alpha}})n^{-\alpha c_{\alpha}}]^{n-k}\\
\lesssim &\sum_{k=\log n}^{n}\big(C\frac{n^{1-\alpha c_{\alpha}}\mathcal{L}(n^{c_{\alpha}})}{k}\big)^k\le n(\log n)^{-\log n}\lesssim n^{-D}
\end{split}
\end{equation*}
due to the Potter bound in Lemma \ref{lemma_potterbound}.
Similarly, one gets
\begin{equation*}
\begin{split}
\mathbb{P}(\sum_{i,j}\psi_{ij}>n^{2-\alpha c_{\alpha}+\epsilon_{\alpha}})=\sum_{k=n^{2-\alpha c_{\alpha}+\epsilon_{\alpha}}}^{pn}\binom{pn}{k} [\mathcal{L}(n^{c_{\alpha}})n^{-\alpha c_{\alpha}}]^{k}[1-\mathcal{L}(n^{c_{\alpha}})n^{-\alpha c_{\alpha}}]^{n^2-k}\lesssim n^{-D}
\end{split}
\end{equation*}
for $0<c_{\alpha}\le 2/\alpha$ with some small constant $\epsilon_{\alpha}>0$ since $\binom{pn}{k}\lesssim (Cn^2/k)^{k}$.
This estimate suggests that there are at most $n^{2-\alpha c_{\alpha}+\epsilon_{\alpha}}$ entries of $|X_{ij}|$ larger than $n^{c_{\alpha}}$ with high probability. Moreover, with high probability, for every row or column, there are at most $n^{1-\alpha c_{\alpha}+\epsilon_{\alpha}}\lor\log n$ entries of $|X_{ij}|$ larger than $n^{c_{\alpha}}$. Observing this fact, we define the following event for $0<c_{\alpha}\le 2/\alpha$,
\begin{equation}\label{eq_highproeventBn}
\mathcal{B}_n=\{\max_{i}\sum_{j}\psi_{ij}+\max_{j}\sum_{i}\psi_{ij}\le 2(n^{1-\alpha c_{\alpha}+\epsilon_{\alpha}}\lor\log n),\sum_{i,j}\psi_{ij}\le n^{2-\alpha c_{\alpha}+\epsilon_{\alpha}}\}
\end{equation}
and focus on the analysis conditional on $\mathcal{B}_n$, since $\mathbb{P}(\mathcal{B}_n)\ge 1-n^{-D}$ for any large fixed $D>0$. Specifically,
if $1/3<c_{\alpha}\le 2/\alpha$ for $\alpha\in [3,4)$, we can set $\epsilon_{\alpha}=3(c_{\alpha}-1/3)>0$ and further define
\begin{equation}\label{eq_highproeventBn1}
\tilde{\mathcal{B}}_n=\{\max_{i}\sum_{j}\psi_{ij}+\max_{j}\sum_{i}\psi_{ij}\le 2\log n,\sum_{i,j}\psi_{ij}\le n^{2-\alpha c_{\alpha}+\epsilon_{\alpha}}\},
\end{equation}
which satisfies $\mathbb{P}( \tilde{\mathcal{B}}_n)\ge 1-n^{-D}$ and thus we can always safely work on $\mathbf{X}$ conditional on the high probability events $\mathcal{B}_n$ and $\tilde{\mathcal{B}}_n$ in the sequel. Moreover, we will set different values of $c_{\alpha}$ to achieve the desired estimations.
Define the index set $\mathrm{I}_c=\{j\in [n]: \sum_{i}\psi_{ij}\ge 1\}$ and its companion $\mathrm{T}_c=\{j\in [n]: \sum_{i}\psi_{ij}=0\}$. It follows that $|\mathrm{I}_c|\le n^{2-\alpha c_{\alpha}+\epsilon_{\alpha}}$ with high probability. With the label index matrix $\Psi$, we adapt the decomposition of $X_{ij}$ as $X_{ij}=X_{ij}(\mathbbm{1}(\psi_{ij}=0)+\mathbbm{1}(\psi_{ij}=1))$.
To get desired moment estimations of $X_{ij}$, we give the following moment estimations for heavy-tailed random variables, also known as Karamata’s theorem (Theorem 1.6.1 in \cite{Bingham1989}), describing the behavior of truncated moments of the regularly varying random variable.
\begin{lemma}[Truncated moments]\label{lemma_truncatemoment}
	Given Assumption \ref{assump1} for $\alpha\in (2,4)$, we have, as $x\rightarrow \infty$
	\begin{equation*}
	\begin{split}
	\mathbb{E}|\xi|^{\beta}\mathbbm{1}(|\xi|\le x)\sim \frac{\alpha}{\beta-\alpha}x^{\beta}\mathbb{P}(|\xi|>x),~\text{if}~\beta>\alpha,\\
	\mathbb{E}|\xi|^{\beta}\mathbbm{1}(|\xi|>x)\sim \frac{\alpha}{\alpha-\beta}x^{\beta}\mathbb{P}(|\xi|>x),~\text{if}~\beta<\alpha.
	\end{split}
	\end{equation*}
\end{lemma}

We begin with an extension of Lemma 3.4 in \cite{bao2015logarithmic}, which can handle the heavy-tailed case.
\begin{lemma}\label{lemma_4momentG}
	Suppose $X_i, i=1, \ldots, n$ are independent real random variables with common mean zero and variance $1$. Moreover, we assume $\max _i \mathbb{E}\left|X_i\right|^s \le v_s$. Let $M_n=\left(m_{i j}\right)_{n, n}$ be a nonnegative definite matrix which is deterministic. Then we have
	\begin{align*}
	\mathbb{E}\left|\sum_{i=1}^n m_{i i} X_i^2-\operatorname{tr} M_n\right|^4 \le C\left(v_8 \operatorname{tr} M_n^4+\left(v_4 \operatorname{tr} M_n^2\right)^2\right)
	\end{align*}
	and
	\begin{align*}
	\mathbb{E}\left|\sum_{u \neq v} m_{u v} X_u X_v\right|^4 \le C v_4^2\|M_n\|^2\operatorname{tr} (M_n^2)
	\end{align*}
	for some positive constant $C$.
\end{lemma}
\begin{proof}
	The diagonal part is the same as that of \cite{bao2015logarithmic}. For the off-diagonal part, we have
	\begin{align*}
	\begin{aligned}
	&\mathbb{E}\left|\sum_{u \neq v} m_{u v} X_u X_v\right|^4  =\sum_{u_i \neq v_i, i=1, \ldots, 4} \prod_{i=1}^4 m_{u_i v_i} \mathbb{E} \prod_{i=1}^4 X_{u_i} X_{v_i} \\
	\le&  C\left(v_4^2 \sum_{u \neq v} m_{u v}^4+v_3^2 v_2 \sum_{u, v, r} m_{u v}^2 m_{u r} m_{r v}+v_2^4 \sum_{u, v, r, w} m_{u v} m_{v r} m_{r w} m_{w u}\right) \\
	\le & C\left(v_4^2 \sum_{u \neq v} m_{u v}^2\|M_n\|^2+v_3^2 v_2 \sum_{u, v} m_{u v}^2 \|M_n\|^2+v_2^4 \operatorname{tr} M_n^4\right)
	 \le C v_4^2\|M_n\|^2\operatorname{tr} (M_n^2),
	\end{aligned}
	\end{align*}
	where in the third line we used $m_{uv}^2\le \|M_n\|^2$ and $|\sum_{r}m_{ur}m_{rv}|=|(M_nM_n^{\top})_{uv}|\le \|M_n\|^2$, and in the last line we used $\sum_{u,v}m_{u,v}^2=\operatorname{tr}M_n^2$ and $\operatorname{tr}M_n^4\le C\|M_n\|^2 \operatorname{tr}M_n^2$.
\end{proof}

Next, we extend Lemma B.25 of \cite{bai2010spectral}, which states the concentration of the quadratic forms after truncation.
\begin{lemma}\label{lemma_quadratic3}
	Let $\mathbf{A}=(a_{ij})$ be one $n\times n$ non-random symmetric positive-definite matrix and $\mathbf{x}=(X_{1},X_{2},\ldots,X_{n})^{\top}$ a random vector with i.i.d. components satisfying \eqref{eq_newcondition} after truncation at $n^{2/3}\log n$. Then, with high probability, we have
	\begin{equation*}
	\mathbb{E}(\mathbf{x}^{\top}\mathbf{A}\mathbf{x}-\operatorname{tr}(\mathbf{A}))^2\lesssim \mathcal{L}(n^{c_{\alpha}})n^{(4-\alpha)c_{\alpha}}\operatorname{tr}(\mathbf{A}\mathbf{A}^{\top})+\mathrm{o}(n^{-4/3}) (\operatorname{tr}\mathbf{A})^2+n^{2/3}\log^{2} n\cdot(\log n+n^{1-\alpha c_{\alpha}+\epsilon_{\alpha}})\|\mathbf{A}\|^2,
	\end{equation*}
	where $c_{\alpha}\in (0,2/3)$ and $\epsilon_{\alpha}>0$ is some small constant. Moreover, if $|X_i|\le n^{c_{\alpha}}$ for $1\le i\le n$ with $2/9< c_{\alpha}<2/3$, then for any $s\ge 1$, we have
	\begin{equation*}
	\mathbb{E}|\mathbf{x}^{\top}\mathbf{A}\mathbf{x}-\operatorname{tr}(\mathbf{A})|^{s}\le C_{s}[(v_4\operatorname{tr}(\mathbf{A}\mathbf{A}^{\top}))^{s/2}+v_{2s}\operatorname{tr}(\mathbf{A}\mathbf{A}^{\top})^{s/2}],
	\end{equation*}
	where $v_{2s}= \mathcal{L}(n^{c_{\alpha}})n^{(2s-\alpha)c_{\alpha}}$.
\end{lemma}
\begin{proof}
	Note the decomposition
	\begin{equation*}
	\mathbf{x}^{\top}\mathbf{A}\mathbf{x}-\operatorname{tr}(\mathbf{A})=\sum_{u}(X_u^2-1)A_{uu}+\sum_{u\ne v}A_{uv}X_{u}X_{v}.
	\end{equation*}
	Given the global truncation level $n^{2/3}\log n$, we have
	\begin{equation*}
	\mu=:\mathbb{E}X_{u}=\mathbb{E}\xi\mathbbm{1}(|\xi|<n^{2/3}\log n)=\mathrm{O}(\mathcal{L}(n^{2/3}\log n)n^{(1-\alpha)2/3}\log^{1-\alpha} n)
	\end{equation*}
	and
	\begin{equation*}
	\sigma^2=:\mathbb{E}X_u^2=\mathbb{E}\xi^2\mathbbm{1}(|\xi|<n^{2/3}\log n)=1-\mathrm{O}(\mathcal{L}(n^{2/3}\log n)n^{(2-\alpha)2/3}\log^{2-\alpha} n)
	\end{equation*}
	by Lemma \ref{lemma_truncatemoment}. Under  \eqref{eq_newcondition}, by Potter's bound in Lemma \ref{lemma_potterbound}, we have
	\begin{equation}\label{eq_estimatemusigma}
	\mu=\mathrm{o}(n^{-4/3})~\text{and}~1-\sigma^2=\mathrm{o}(n^{-2/3}).
	\end{equation}
	To derive the high probability bounds for the quadratics, consider the resampling for $\mathbf{x}$ at truncation level $n^{c_{\alpha}}$ for $c_{\alpha}\in (0,2/3)$, which will be fixed later. Define the label vector $\mathbf{l}=:(\psi_1,\ldots,\psi_n)^{\top}=(|X_1|>n^{c_{\alpha}},\ldots,|X_n|>n^{c_{\alpha}})^{\top}$. It follows that, we have $\sum_{u}\psi_u\le \log n\lor n^{1-\alpha c_{\alpha}+\epsilon_{\alpha}}$ with high probability. In the sequel, we condition on the following high probability event
	\begin{equation*}
	\mathcal{C}_n=\{\sum_{u}\psi_u\le \log n+n^{1-\alpha c_{\alpha}+\epsilon_{\alpha}}\}.
	\end{equation*}
	By Lemma \ref{lemma_truncatemoment}, one has
	\begin{equation*}
	\mathbb{E}X_u\mathbbm{1}(\psi_u=0)=\mathrm{O}(\mathcal{L}(n^{c_{\alpha}})n^{(1-\alpha)c_{\alpha}}), \mathbb{E}X_{u}^2\mathbbm{1}(\psi_u=0)= 1-\mathrm{O}(\mathcal{L}(n^{c_\alpha})n^{(2-\alpha)c_{\alpha}})
	\end{equation*}
	and for integer $q\ge 2$,
	\begin{equation*}
	\mathbb{E}X_u^{2q}\mathbbm{1}(\psi_u=0)\sim\mathcal{L}(n^{c_{\alpha}})n^{(2q-\alpha)c_{\alpha}}.
	\end{equation*}
	Moreover, it is obvious that
	\begin{equation*}
		\begin{split}
		\mathbb{E}X_{u}^4\mathbbm{1}(\psi_u=1)\le \mathbb{E}X_{u}^4= \mathbb{E}\xi^4\mathbbm{1}(|\xi|<n^{2/3}\log n)\lesssim\mathcal{L}(n^{2/3}\log n)n^{2(4-\alpha)/3}\log^{4-\alpha} n\lesssim n^{2/3}\log^{2} n
		\end{split}
	\end{equation*}
	by Lemma \ref{lemma_truncatemoment} and \eqref{eq_newcondition}.
	For $s=2$, it follows that
	\begin{equation*}
	\begin{split}
	\mathbb{E}(\mathbf{x}^{\top}\mathbf{A}\mathbf{x}-\operatorname{tr}(\mathbf{A}))^2\le  2\mathbb{E}\big(\sum_{u}(X_u^2-1)A_{uu}\big)^2+2\mathbb{E}\big(\sum_{u\ne v}A_{uv}X_{u}X_{v}\big)^2.
	\end{split}
	\end{equation*}
	For the diagonal part, with high probability,
	\begin{equation*}
	\begin{split}
	&\mathbb{E}\big(\sum_{u}(X_u^2-1)A_{uu}\big)^2
	=
	\sum_{u}\mathbb{E}(X_u^2-1)^2A_{uu}^2+\sum_{u\ne v}\mathbb{E}(X_u^2-1)(X_v^2-1)A_{uu}A_{vv}\\
	= &\sum_{u}\mathbb{E}(X_u^2-1)^2A_{uu}^2(\mathbbm{1}(\psi_u=0)+\mathbbm{1}(\psi_u=1))+(1-\sigma^2)^2\sum_{u\ne v}A_{uu}A_{vv}\\
	\lesssim & \mathcal{L}(n^{c_{\alpha}})n^{(4-\alpha)c_{\alpha}}\sum_{u}A_{uu}^2+n^{2/3}\log^{2} n\cdot(\log n+n^{1-\alpha c_{\alpha}+\epsilon_{\alpha}})\|\mathbf{A}\|^2+\mathrm{o}(n^{-4/3}) (\operatorname{tr}\mathbf{A})^2\\
	\le & \mathcal{L}(n^{c_{\alpha}})n^{(4-\alpha)c_{\alpha}}\operatorname{tr}(\mathbf{A}\mathbf{A}^{\top})+n^{2/3}\log^{2} n\cdot(\log n+n^{1-\alpha c_{\alpha}+\epsilon_{\alpha}})\|\mathbf{A}\|^2+\mathrm{o}(n^{-4/3}) (\operatorname{tr}\mathbf{A})^2,
	\end{split}
	\end{equation*}
	due to $\max_u A_{uu}\le \|\mathbf{A}\|$, \eqref{eq_estimatemusigma} and the definition of $\mathcal{C}_n$. For the off-diagonal part, we have
	\begin{equation*}
	\begin{split}
	\mathbb{E}\big(\sum_{u\ne v}A_{uv}X_{u}X_{v}\big)^2
	\lesssim & \sum_{u\ne v}A_{uv}^2 +\mu^2\sum_{u\ne v\ne m}A_{uv}A_{um}+\mu^4\sum_{u\ne v\ne m\ne k}A_{uv}A_{mk}\\
	\lesssim & \operatorname{tr}(\mathbf{A}\mathbf{A}^{\top})+\mu^2n\operatorname{tr}(\mathbf{A}\mathbf{A}^{\top})+\mu^4n^2\operatorname{tr}(\mathbf{A}\mathbf{A}^{\top})\lesssim (2+\mu^4 n^2)\operatorname{tr}(\mathbf{A}\mathbf{A}^{\top}),
	\end{split}
	\end{equation*}
	where we used the facts
	\[
	\sum_{u\ne v\ne m}A_{uv}A_{um}\le \sum_{u}(\sum_{v}A_{uv})^2\le Cn\sum_{u}\sum_{v}A_{uv}^2\le Cn\operatorname{tr}(\mathbf{A}\mathbf{A}^{\top})\]
	and
	\[
	\sum_{u\ne v\ne m\ne k}A_{uv}A_{mk}\le (\sum_{u\ne v}A_{uv})^2\le Cn^2\sum_{u\ne v}A_{uv}^2\le  Cn^2\operatorname{tr}(\mathbf{A}\mathbf{A}^{\top}).
	\]
	The last inequality follows from $\mu=\mathrm{o}(n^{-4/3})$ and $\sigma^2=\mathrm{o}(n^{-2/3})$ in \eqref{eq_estimatemusigma}. Thus, for $c_{\alpha}\in (0,2/3)$, it follows that
	\begin{equation*}
	\mathbb{E}(\mathbf{x}^{\top}\mathbf{A}\mathbf{x}-\operatorname{tr}(\mathbf{A}))^2\lesssim \mathcal{L}(n^{c_{\alpha}})n^{(4-\alpha)c_{\alpha}}\operatorname{tr}(\mathbf{A}\mathbf{A}^{\top})+n^{2/3}\log^{2} n\cdot(
\log n+n^{1-\alpha c_{\alpha}+\epsilon_{\alpha}})\|\mathbf{A}\|^2+\mathrm{o}(n^{-4/3}) (\operatorname{tr}\mathbf{A})^2
	\end{equation*}
	with high probability.
	
	For the second statement, with a similar argument for Lemma B.26 of \cite{bai2010spectral}, conditional on the event $\{\sum_{u}\psi_u=0\}$, we can get,
	\begin{equation*}
	\begin{split}
	\mathbb{E}|\mathbf{x}^{\top}\mathbf{A}\mathbf{x}-\operatorname{tr}(\mathbf{A})|
	\le  \sum_{u}|A_{uu}|\mathbb{E}|X_u^2-1|+[\mathbb{E}(\sum_{u\ne v}A_{uv}X_{u}X_{v})^2]^{1/2}
	\lesssim (\operatorname{tr}(\mathbf{A}\mathbf{A}^{\top})^{1/2}+(v_4\operatorname{tr}\mathbf{A}\mathbf{A}^{\top})^{1/2}),
	\end{split}
	\end{equation*}
	where the last inequality follows from
	\begin{equation*}
	v_4=\mathbb{E}\xi^4\mathbbm{1}(|\xi|<n^{c_{\alpha}})\sim \mathcal{L}(n^{c_{\alpha}})n^{(4-\alpha)c_{\alpha}}\ge \hat{\mu}^4 n^2
	\end{equation*}
	since $\hat{\mu}=:\mathbb{E}\xi\mathbbm{1}(|\xi|<n^{c_{\alpha}})\lesssim \mathcal{L}(n^{c_{\alpha}})n^{(1-\alpha)c_{\alpha}}$ and $c_{\alpha}\in (2/9,2/3)$.
	For $1< s\le 2$, by Lemma 2.12 and Theorem A.13 of \cite{bai2010spectral}, we have
	\begin{equation*}
	\begin{split}
	\mathbb{E}|\sum_{u}A_{uu}(X_u^2-1)|^s\le C\mathbb{E}\left(\sum_{u}|A_{uu}^2||X_u^2-1|^2\right)^{s/2}\le C\sum_{u}|A_{uu}|^s\mathbb{E}|X_u^2-1|^s\le Cv_{2s}\operatorname{tr}(\mathbf{A}\mathbf{A}^{\top})^{s/2}.
	\end{split}
	\end{equation*}
	The off-diagonal part can be bounded by the H\"{o}lder inequality as
	\begin{equation*}
	\mathbb{E}|\sum_{u\ne v}A_{uv}X_{u}X_{v}|^s\le C[\mathbb{E}(\sum_{u\ne v}A_{uv}X_{u}X_{v})^2]^{s/2}\le Cv_{2s}\operatorname{tr}(\mathbf{A}\mathbf{A}^{\top})^{s/2}.
	\end{equation*}
	Following the proof of Lemma B.25 of \cite{bai2010spectral} by induction on $s$, we can get, for any $s\ge 1$,
	\begin{equation*}
	\mathbb{E}|\mathbf{x}^{\top}\mathbf{A}\mathbf{x}-\operatorname{tr}(\mathbf{A})|^{s}\le C_{s}[(v_4\operatorname{tr}(\mathbf{A}\mathbf{A}^{\top}))^{s/2}+v_{2s}\operatorname{tr}(\mathbf{A}\mathbf{A}^{\top})^{s/2}],
	\end{equation*}
	where $v_{2s}=:\mathbb{E}X_u^{2s}\mathbbm{1}(\psi_u=0)\sim\mathcal{L}(n^{c_{\alpha}})n^{(2s-\alpha)c_{\alpha}}$ by Lemma \ref{lemma_truncatemoment}.
\end{proof}
\begin{lemma}[Berry-Esseen bound, Theorem 5.4 of \citep{Petrov1995}] \label{lemma_berryesseen}
	Let $Z_1, \ldots, Z_m$ be independent real random variables such that $\mathbb{E} Z_j=0$ and $\mathbb{E}\left|Z_j\right|^3<\infty, j=1, \ldots, n$. Assume that
	\begin{align*}
	\sigma_j^2=\mathbb{E} Z_j^2, \quad D_m=\sum_{i=1}^m \sigma_j^2, \quad L_m=D_m^{-3 / 2} \sum_{j=1}^m \mathbb{E}\left|Z_j\right|^3 .
	\end{align*}
	Then there exists a constant $C>0$ such that
	\begin{align*}
	\sup _x\left|\mathbb{P}\big(D_m^{-1 / 2} \sum_{j=1}^m Z_j \le x\big)-\Phi(x)\right| \le C L_m,
	\end{align*}
	where $\Phi(x)$ is the cumulative distribution function of the standard normal distribution.
\end{lemma}

\section{Diagonals of the projection matrices}\label{sec_diag}
The diagonals of the normalized projection matrices $\mathbf{Q}_i$ (which is equivalent to $\mathbf{P}_i$ except the normalized factor $1/(n-i)$) in \eqref{eq_defQ}) play an essential role in the proof of Theorems \ref{thm_logdet} and \ref{thm_pn}, which perform well-concentrated with variance estimation. However, under the heavy-tailed setting, the fluctuation is difficult to control, especially in the near-singularity case. By the relation $\mathbf{B}_i\mathbf{B}_i^{\top}=\sum_{\ell=1}^{n}\mathbf{v}_{\ell,i}\mathbf{v}_{\ell,i}^{\top}$ and the resolvent identity, we have
\begin{equation*}
\begin{split}
\mathbf{v}_{\ell,i}^{\top}(\mathbf{B}_i\mathbf{B}_i^{\top})^{-1}\mathbf{v}_{\ell,i}=\mathbf{v}_{\ell,i}^{\top}(\mathbf{B}_{i(\ell)}\mathbf{B}_{i(\ell)}^{\top})^{-1}\mathbf{v}_{\ell,i}-\mathbf{v}_{\ell,i}^{\top}(\mathbf{B}_i\mathbf{B}_i^{\top})^{-1}\mathbf{v}_{\ell,i}\mathbf{v}_{\ell,i}^{\top}(\mathbf{B}_{i(\ell)}\mathbf{B}_{i(\ell)}^{\top})^{-1}\mathbf{v}_{\ell,i},
\end{split}
\end{equation*}
where $\mathbf{B}_{i(\ell)}\mathbf{B}_{i(\ell)}^{\top}=\sum_{s\ne \ell}^{n}\mathbf{v}_{s,i}\mathbf{v}_{s,i}^{\top}$. This further gives
\begin{equation}\label{eq_rewrite}
p_{i,\ell\ell}=1-\mathbf{v}_{\ell,i}^{\top}(\mathbf{B}_{i}\mathbf{B}_{i}^{\top})^{-1}\mathbf{v}_{\ell,i}=1-\frac{\mathbf{v}_{\ell,i}^{\top}(\mathbf{B}_{i(\ell)}\mathbf{B}_{i(\ell)}^{\top})^{-1}\mathbf{v}_{\ell,i}}{1+\mathbf{v}_{\ell,i}^{\top}(\mathbf{B}_{i(\ell)}\mathbf{B}_{i(\ell)}^{\top})^{-1}\mathbf{v}_{\ell,i}}=\frac{1}{1+\mathbf{v}_{\ell,i}^{\top}(\mathbf{B}_{i(\ell)}\mathbf{B}_{i(\ell)}^{\top})^{-1}\mathbf{v}_{\ell,i}}.
\end{equation}
The first result is an extension of Lemma B.2 of \cite{heiny2023logdet}.
\subsection{Concentration for the diagonals of the projection matrices}
\begin{lemma}\label{lemma_diagonals}
	Let $\alpha\in [3,4)$ and $\lim_{n\rightarrow \infty}p/n<1$. It holds that
	\begin{equation}\label{eq_concentraqikk}
	0\le \sum_{i=0}^{p-1}\sum_{\ell=1}^{n}\big(\mathbb{E}[q_{i,\ell\ell}^2]-\frac{1}{n^2}\big)\le \mathrm{O}(n^{-1/2}+n^{(2-\alpha)/2}\mathcal{L}(n^{1/2})).
	\end{equation}
\end{lemma}
\begin{proof}
Notice the definition of $\mathbf{Q}_i$ in \eqref{eq_defQ}, which implies that $q_{i,\ell\ell}$ for $1\le \ell\le n$ share the same distribution with $\mathbb{E}q_{i,\ell\ell}=1/n$. Applying Jensen’s inequality and the fact that $\sum_{\ell}q_{i,\ell\ell}=1$ with $0\le q_{i,\ell\ell}\le 1/(n-i)$, we obtain that
\begin{equation*}
\frac{1}{n^k}=\big(\frac{1}{n}\sum_{\ell=1}^{n}q_{i,\ell\ell}\big)^{k}\le \frac{1}{n}\sum_{\ell=1}^{n}q_{i,\ell\ell}^k
\end{equation*}
for any integer $k\ge 1$. Taking expectation on both hands gives the lower bound as desired. It suffices to consider the upper bound.
Following a simialr argument of Lemma B.2 with $k=2$ in \cite{heiny2023logdet}, we have
\begin{equation*}
\delta_n=:\sum_{i=0}^{p-1}\sum_{\ell=1}^{n}\big(\mathbb{E}[q_{i,\ell\ell}^2]-\frac{1}{n^2}\big)=\sum_{i=0}^{p-1}\sum_{\ell=1}^{n}\mathbb{E}\big(q_{i,\ell\ell}-\mathbb{E}(q_{i,\ell\ell})\big)^2.
\end{equation*}
Invoke the fact that $q_{i,\ell\ell}$ are identically distributed over $\ell$.
Applying Minkowski's inequality gives
\begin{equation*}
	\delta_n=\sum_{i=0}^{p-1}\frac{n}{(n-i)^2}\mathbb{E}\big(\mathbf{v}_{1,i}^{\top}(\mathbf{B}_{i}\mathbf{B}_{i}^{\top})^{-1}\mathbf{v}_{1,i}-\mathbb{E}(\mathbf{v}_{1,i}^{\top}(\mathbf{B}_{i}\mathbf{B}_{i}^{\top})^{-1}\mathbf{v}_{1,i})\big)^2\le C(\delta_n^{(1)}+\delta_n^{(2)}+\delta_n^{(3)}),
\end{equation*}
where
\begin{equation*}
	\begin{split}
	\delta_n^{(1)}=&\sum_{i=0}^{p-1}\frac{n}{(n-i)^2}\mathbb{E}\big(\mathbf{v}_{1,i}^{\top}(\mathbf{B}_{i}\mathbf{B}_{i}^{\top})^{-1}\mathbf{v}_{1,i}-\mathbf{v}_{1,i}^{\top}(\mathbf{B}_{i}\mathbf{B}_{i}^{\top}+\epsilon_n n\mathbf{I}_i)^{-1}\mathbf{v}_{1,i}\big)^2,\\
	\delta_n^{(2)}=&\sum_{i=0}^{p-1}\frac{n}{(n-i)^2}\mathbb{E}\big(\mathbf{v}_{1,i}^{\top}(\mathbf{B}_{i}\mathbf{B}_{i}^{\top}+\epsilon_n n\mathbf{I}_i)^{-1}\mathbf{v}_{1,i}-\frac{\mathbb{E}\operatorname{tr}(\mathbf{B}_{i(1)}\mathbf{B}_{i(1)}^{\top}+\epsilon_n n\mathbf{I}_i)^{-1}}{1+\mathbb{E}\operatorname{tr}(\mathbf{B}_{i(1)}\mathbf{B}_{i(1)}^{\top}+\epsilon_n n\mathbf{I}_i)^{-1}}\big)^2,\\
	\delta_n^{(3)}=&\sum_{i=0}^{p-1}\frac{n}{(n-i)^2}\mathbb{E}\big(\frac{\mathbb{E}\operatorname{tr}(\mathbf{B}_{i(1)}\mathbf{B}_{i(1)}^{\top}+\epsilon_n n\mathbf{I}_i)^{-1}}{1+\mathbb{E}\operatorname{tr}(\mathbf{B}_{i(1)}\mathbf{B}_{i(1)}^{\top}+\epsilon_n n\mathbf{I}_i)^{-1}}-\mathbb{E}(\mathbf{v}_{1,i}^{\top}(\mathbf{B}_{i}\mathbf{B}_{i}^{\top})^{-1}\mathbf{v}_{1,i})\big)^2.
	\end{split}
\end{equation*}
Here $\epsilon_n$ is a sequence tending to zero arbitrarily slower than $1/n$ and will be fixed later. For $\delta_n^{(1)}$, note $\mathbf{v}_{1,i}^{\top}(\mathbf{B}_{i}\mathbf{B}_{i}^{\top})^{-1}\mathbf{v}_{1,i}\le 1$, which further implies
\begin{equation*}
	\begin{split}
	&|\mathbf{v}_{1,i}^{\top}(\mathbf{B}_{i}\mathbf{B}_{i}^{\top})^{-1}\mathbf{v}_{1,i}-\mathbf{v}_{1,i}^{\top}(\mathbf{B}_{i}\mathbf{B}_{i}^{\top}+\epsilon_n n\mathbf{I}_i)^{-1}\mathbf{v}_{1,i}|
	=\epsilon_{n}n\cdot \mathbf{v}_{1,i}^{\top}(\mathbf{B}_{i}\mathbf{B}_{i}^{\top})^{-1}(\mathbf{B}_{i}\mathbf{B}_{i}^{\top}+\epsilon_n n\mathbf{I}_i)^{-1}\mathbf{v}_{1,i}\\
	\le &\frac{\epsilon_{n} n}{\lambda_{\min}(\mathbf{B}_{i}\mathbf{B}_{i}^{\top}+\epsilon_n n\mathbf{I}_i)}\mathbf{v}_{1,i}^{\top}(\mathbf{B}_{i}\mathbf{B}_{i}^{\top})^{-1}\mathbf{v}_{1,i}
	\le  \frac{\epsilon_{n}}{(1-\sqrt{i/n})^2}\le C\epsilon_{n}.
	\end{split}
\end{equation*}
It can be verified that
\begin{equation}\label{eq_delta1}
	\delta_n^{(1)}\le C\sum_{i=0}^{p-1}\frac{n}{(n-i)^2}\epsilon_{n}^2=\mathrm{O}(\epsilon_{n}^2).
\end{equation}
For $\delta_n^{(2)}$, similaly to \eqref{eq_rewrite}, one has
\begin{equation}\label{eq_delta20}
	\begin{split}
	&\mathbb{E}\big(\mathbf{v}_{1,i}^{\top}(\mathbf{B}_{i}\mathbf{B}_{i}^{\top}+\epsilon_n n\mathbf{I}_i)^{-1}\mathbf{v}_{1,i}-\frac{\mathbb{E}\operatorname{tr}(\mathbf{B}_{i(1)}\mathbf{B}_{i(1)}^{\top}+\epsilon_n n\mathbf{I}_i)^{-1}}{1+\mathbb{E}\operatorname{tr}(\mathbf{B}_{i(1)}\mathbf{B}_{i(1)}^{\top}+\epsilon_n n\mathbf{I}_i)^{-1}}\big)^2\\
	=&\mathbb{E}\big(\frac{\mathbf{v}_{1,i}^{\top}(\mathbf{B}_{i(1)}\mathbf{B}_{i(1)}^{\top}+\epsilon_n n\mathbf{I}_i)^{-1}\mathbf{v}_{1,i}}{1+\mathbf{v}_{1,i}^{\top}(\mathbf{B}_{i(1)}\mathbf{B}_{i(1)}^{\top}+\epsilon_n n\mathbf{I}_i)^{-1}\mathbf{v}_{1,i}}-\frac{\mathbb{E}\operatorname{tr}(\mathbf{B}_{i(1)}\mathbf{B}_{i(1)}^{\top}+\epsilon_n n\mathbf{I}_i)^{-1}}{1+\mathbb{E}\operatorname{tr}(\mathbf{B}_{i(1)}\mathbf{B}_{i(1)}^{\top}+\epsilon_n n\mathbf{I}_i)^{-1}}\big)^2\\
	\le & C\mathbb{E}(\mathbf{v}_{1,i}^{\top}(\mathbf{B}_{i(1)}\mathbf{B}_{i(1)}^{\top}+\epsilon_n n\mathbf{I}_i)^{-1}\mathbf{v}_{1,i}-\operatorname{tr}(\mathbf{B}_{i(1)}\mathbf{B}_{i(1)}^{\top}+\epsilon_n n\mathbf{I}_i)^{-1})^2\\
	&+C\mathbb{E}(\operatorname{tr}(\mathbf{B}_{i(1)}\mathbf{B}_{i(1)}^{\top}+\epsilon_n n\mathbf{I}_i)^{-1}-\mathbb{E}\operatorname{tr}(\mathbf{B}_{i(1)}\mathbf{B}_{i(1)}^{\top}+\epsilon_n n\mathbf{I}_i)^{-1})^2.
	\end{split}
\end{equation}
Similarly to (B.8) of \cite{heiny2023logdet}, using the martingale differences and Sherman-Morrison formula gives
\begin{equation}\label{eq_delta22}
	\mathbb{E}(\operatorname{tr}(\mathbf{B}_{i(1)}\mathbf{B}_{i(1)}^{\top}+\epsilon_n n\mathbf{I}_i)^{-1}-\mathbb{E}\operatorname{tr}(\mathbf{B}_{i(1)}\mathbf{B}_{i(1)}^{\top}+\epsilon_n n\mathbf{I}_i)^{-1})^2\le \mathrm{O}(\epsilon_n^{-2}n^{-1}).
\end{equation}
Define $\mathbb{E}_{\ell}:=\mathbb{E}[\cdot\mid \mathbf{v}_{\ell,i},\ldots,\mathbf{v}_{n,i}]$ for $\ell=1,\ldots,n$. It follows that
\begin{equation*}
	\begin{split}
	&\mathbb{E}(\mathbf{v}_{1,i}^{\top}(\mathbf{B}_{i(1)}\mathbf{B}_{i(1)}^{\top}+\epsilon_n n\mathbf{I}_i)^{-1}\mathbf{v}_{1,i}-\operatorname{tr}(\mathbf{B}_{i(1)}\mathbf{B}_{i(1)}^{\top}+\epsilon_n n\mathbf{I}_i)^{-1})^2\\
	=&\mathbb{E}(\mathbf{v}_{1,i}^{\top}(\mathbf{B}_{i(1)}\mathbf{B}_{i(1)}^{\top}+\epsilon_n n\mathbf{I}_i)^{-1}\mathbf{v}_{1,i}-\mathbb{E}_2\mathbf{v}_{1,i}^{\top}(\mathbf{B}_{i(1)}\mathbf{B}_{i(1)}^{\top}+\epsilon_n n\mathbf{I}_i)^{-1}\mathbf{v}_{1,i})^2.
	\end{split}
\end{equation*}
Let $\hat{v}_{1j,i}=v_{1j,i}\mathbbm{1}(|v_{1j,i}|<n^{1/2})$. By Lemma \ref{lemma_truncatemoment} for $\alpha\in[3,4)$, we have
\begin{equation}\label{eq_v1imoments1}
\mathbb{E}\hat{v}_{1j,i}=-\mathbb{E}v_{1j,i}\mathbbm{1}(|v_{1j,i}|\ge n^{1/2})=\mathrm{O}(\mathcal{L}(n^{1/2})n^{(1-\alpha)/2})=\mathrm{o}(n^{-1+\epsilon_0})
\end{equation}
where $\epsilon_0>0$ can be arbitrarily small due to Potter's bound, and
\begin{equation}\label{eq_v1imoments2}
1-\mathbb{E}\hat{v}_{1j,i}^2=\mathbb{E}v_{1j,i}^2\mathbbm{1}(|v_{1j,i}|\ge n^{1/2})=\mathrm{O}(\mathcal{L}(n^{1/2})n^{(2-\alpha)/2}).
\end{equation}
It follows that
\begin{equation}\label{eq_delta21}
\begin{split}
&\mathbb{E}(\mathbf{v}_{1,i}^{\top}(\mathbf{B}_{i(1)}\mathbf{B}_{i(1)}^{\top}+\epsilon_n n\mathbf{I}_i)^{-1}\mathbf{v}_{1,i}-\mathbb{E}_2\mathbf{v}_{1,i}^{\top}(\mathbf{B}_{i(1)}\mathbf{B}_{i(1)}^{\top}+\epsilon_n n\mathbf{I}_i)^{-1}\mathbf{v}_{1,i})^2\\
\le &C\mathbb{E}(\mathbf{v}_{1,i}^{\top}(\mathbf{B}_{i(1)}\mathbf{B}_{i(1)}^{\top}+\epsilon_n n\mathbf{I}_i)^{-1}\mathbf{v}_{1,i}-\hat{\mathbf{v}}_{1,i}^{\top}(\mathbf{B}_{i(1)}\mathbf{B}_{i(1)}^{\top}+\epsilon_n n\mathbf{I}_i)^{-1}\hat{\mathbf{v}}_{1,i})^2\\
&+C\mathbb{E}(\mathbb{E}_2\mathbf{v}_{1,i}^{\top}(\mathbf{B}_{i(1)}\mathbf{B}_{i(1)}^{\top}+\epsilon_n n\mathbf{I}_i)^{-1}\mathbf{v}_{1,i}-\mathbb{E}_2\hat{\mathbf{v}}_{1,i}^{\top}(\mathbf{B}_{i(1)}\mathbf{B}_{i(1)}^{\top}+\epsilon_n n\mathbf{I}_i)^{-1}\hat{\mathbf{v}}_{1,i})^2\\
&+C\mathbb{E}(\hat{\mathbf{v}}_{1,i}^{\top}(\mathbf{B}_{i(1)}\mathbf{B}_{i(1)}^{\top}+\epsilon_n n\mathbf{I}_i)^{-1}\hat{\mathbf{v}}_{1,i}-\mathbb{E}_2\hat{\mathbf{v}}_{1,i}^{\top}(\mathbf{B}_{i(1)}\mathbf{B}_{i(1)}^{\top}+\epsilon_n n\mathbf{I}_i)^{-1}\hat{\mathbf{v}}_{1,i})^2.
\end{split}
\end{equation}
Consider the first term in \eqref{eq_delta21}.
Similarly to (B.11) of \cite{heiny2023logdet}, we get
\begin{equation}\label{eq_trundiff}
	\begin{split}
	&\mathbb{E}(\mathbf{v}_{1,i}^{\top}(\mathbf{B}_{i(1)}\mathbf{B}_{i(1)}^{\top}+\epsilon_n n\mathbf{I}_i)^{-1}\mathbf{v}_{1,i}-\hat{\mathbf{v}}_{1,i}^{\top}(\mathbf{B}_{i(1)}\mathbf{B}_{i(1)}^{\top}+\epsilon_n n\mathbf{I}_i)^{-1}\hat{\mathbf{v}}_{1,i})^2\\
	=&\mathbb{E}((\mathbf{v}_{1,i}-\hat{\mathbf{v}}_{1,i})^{\top}(\mathbf{B}_{i(1)}\mathbf{B}_{i(1)}^{\top}+\epsilon_n n\mathbf{I}_i)^{-1}(\mathbf{v}_{1,i}+\hat{\mathbf{v}}_{1,i}))^2\\
	\le &C \mathbb{E}((\mathbf{v}_{1,i}-\hat{\mathbf{v}}_{1,i})^{\top}(\mathbf{B}_{i(1)}\mathbf{B}_{i(1)}^{\top}+\epsilon_n n\mathbf{I}_i)^{-1}(\mathbf{v}_{1,i}-\hat{\mathbf{v}}_{1,i}))\\
	\le & C\mathbb{E}\frac{(\mathbf{v}_{1,i}-\hat{\mathbf{v}}_{1,i})^{\top}(\mathbf{v}_{1,i}-\hat{\mathbf{v}}_{1,i})}{\lambda_{\min}(\mathbf{B}_{i(1)}\mathbf{B}_{i(1)}^{\top}+\epsilon_n n\mathbf{I}_i)}\le \frac{C}{n}\mathbb{E}\big(\sum_{j=1}^{n}v_{1j,i}^2\mathbbm{1}(|v_{1j,i}|\ge n^{1/2})\big)\\
	\le & C\mathbb{E}v_{11,i}^2\mathbbm{1}(|v_{11,i}|\ge n^{1/2})=\mathrm{O}(\mathcal{L}(n^{1/2})n^{(2-\alpha)/2}).
	\end{split}
\end{equation}
The second term of \eqref{eq_delta21} satisfies
\begin{equation*}
	\begin{split}
	&\mathbb{E}(\mathbb{E}_2\mathbf{v}_{1,i}^{\top}(\mathbf{B}_{i(1)}\mathbf{B}_{i(1)}^{\top}+\epsilon_n n\mathbf{I}_i)^{-1}\mathbf{v}_{1,i}-\mathbb{E}_2\hat{\mathbf{v}}_{1,i}^{\top}(\mathbf{B}_{i(1)}\mathbf{B}_{i(1)}^{\top}+\epsilon_n n\mathbf{I}_i)^{-1}\hat{\mathbf{v}}_{1,i})^2\\
	\le &C\mathbb{E}(\operatorname{tr}(\mathbf{B}_{i(1)}\mathbf{B}_{i(1)}^{\top}+\epsilon_n n\mathbf{I}_i)^{-1}(1-\mathbb{E}\hat{v}_{1j,i}^2))^2+C\mathbb{E}(\sum_{u\ne v}(\mathbb{E}\hat{v}_{1j,i})^2[(\mathbf{B}_{i(1)}\mathbf{B}_{i(1)}^{\top}+\epsilon_n n\mathbf{I}_i)^{-1}]_{uv})^2\\
	\le & C[\mathcal{L}(n^{1/2})]^2n^{2-\alpha}\mathbb{E}(\operatorname{tr}(\mathbf{B}_{i(1)}\mathbf{B}_{i(1)}^{\top}+\epsilon_n n\mathbf{I}_i)^{-1})^2+C[\mathcal{L}(n^{1/2})]^4n^{2(1-\alpha)}\cdot n^2\mathbb{E}(\operatorname{tr}(\mathbf{B}_{i(1)}\mathbf{B}_{i(1)}^{\top}+\epsilon_n n\mathbf{I}_i)^{-2})\\
	\le & C[\mathcal{L}(n^{1/2})]^2n^{2-\alpha}+C[\mathcal{L}(n^{1/2})]^4n^{3-2\alpha},
	\end{split}
\end{equation*}
since $\|(\mathbf{B}_{i(1)}\mathbf{B}_{i(1)}^{\top}+\epsilon_n n\mathbf{I}_i)^{-1}\|\le Cn^{-1}$ and
\begin{equation*}
	\begin{split}
	\mathbb{E}(\operatorname{tr}(\mathbf{B}_{i(1)}\mathbf{B}_{i(1)}^{\top}+\epsilon_n n\mathbf{I}_i)^{-1})^2
	\le& [\mathbb{E}(\operatorname{tr}(\mathbf{B}_{i(1)}\mathbf{B}_{i(1)}^{\top}+\epsilon_n n\mathbf{I}_i)^{-1})]^2+\operatorname{Var}(\operatorname{tr}(\mathbf{B}_{i(1)}\mathbf{B}_{i(1)}^{\top}+\epsilon_n n\mathbf{I}_i)^{-1})\\
	\le & C(1-i/n)^{-1}+\mathrm{O}(n^{-1}\epsilon_{n}^{-2})\le C
	\end{split}
\end{equation*}
by Lemma \ref{lemma_esdconvergence} for $\epsilon_{n}\ge n^{-1/2}$.
Moreover, one has
\begin{equation*}
\mathbb{E}\hat{v}_{1j,i}^4=\mathbb{E}v_{1j,i}^4\mathbbm{1}(|v_{1j,i}|<n^{1/2})\sim \frac{\alpha}{4-\alpha}n^{(4-\alpha)/2}\mathcal{L}(n^{1/2})=\mathrm{O}(\mathcal{L}(n^{1/2})n^{(4-\alpha)/2})
\end{equation*}
as $n\rightarrow \infty$ by Lemma \ref{lemma_truncatemoment}.
For the last term in \eqref{eq_delta21}, applying Lemma \ref{lemma_quadratic3} with $c_{\alpha}=1/2$ gives
\begin{equation*}
	\begin{split}
	&\mathbb{E}(\hat{\mathbf{v}}_{1,i}^{\top}(\mathbf{B}_{i(1)}\mathbf{B}_{i(1)}^{\top}+\epsilon_n n\mathbf{I}_i)^{-1}\hat{\mathbf{v}}_{1,i}-\mathbb{E}_2\hat{\mathbf{v}}_{1,i}^{\top}(\mathbf{B}_{i(1)}\mathbf{B}_{i(1)}^{\top}+\epsilon_n n\mathbf{I}_i)^{-1}\hat{\mathbf{v}}_{1,i})^2\\
	\le & C\mathbb{E}\hat{v}_{11,i}^4\mathbb{E}(\operatorname{tr}(\mathbf{B}_{i(1)}\mathbf{B}_{i(1)}^{\top}+\epsilon_n n\mathbf{I}_i)^{-2})
	\le C\mathcal{L}(n^{1/2})n^{(2-\alpha)/2}.
	\end{split}
\end{equation*}
Thus, we get, for $\epsilon_{n}\ge n^{-1/2}$,
\begin{equation}\label{eq_delta23}
	\begin{split}
	\mathbb{E}(\mathbf{v}_{1,i}^{\top}(\mathbf{B}_{i(1)}\mathbf{B}_{i(1)}^{\top}+\epsilon_n n\mathbf{I}_i)^{-1}\mathbf{v}_{1,i}-\mathbb{E}_2\mathbf{v}_{1,i}^{\top}(\mathbf{B}_{i(1)}\mathbf{B}_{i(1)}^{\top}+\epsilon_n n\mathbf{I}_i)^{-1}\mathbf{v}_{1,i})^2
	\lesssim \mathcal{L}(n^{1/2})n^{(2-\alpha)/2},
	\end{split}
\end{equation}
which together with \eqref{eq_delta22} further implies
\begin{equation}\label{eq_delta2}
	\delta_n^{(2)}\le C\sum_{i=0}^{p-1}\frac{n}{(n-i)^2}(\epsilon_{n}^{-2}n^{-1}+\mathcal{L}(n^{1/2})n^{(2-\alpha)/2})=\mathrm{O}(\epsilon_{n}^{-2}n^{-1}+\mathcal{L}(n^{1/2})n^{(2-\alpha)/2})
\end{equation}
for $\epsilon_{n}\ge n^{-1/2}$. Next, we proceed to $\delta_n^{(3)}$. Observe that
\begin{equation*}
	\begin{split}
	&\big(\frac{\mathbb{E}\operatorname{tr}(\mathbf{B}_{i(1)}\mathbf{B}_{i(1)}^{\top}+\epsilon_n n\mathbf{I}_i)^{-1}}{1+\mathbb{E}\operatorname{tr}(\mathbf{B}_{i(1)}\mathbf{B}_{i(1)}^{\top}+\epsilon_n n\mathbf{I}_i)^{-1}}-\mathbb{E}(\mathbf{v}_{1,i}^{\top}(\mathbf{B}_{i}\mathbf{B}_{i}^{\top})^{-1}\mathbf{v}_{1,i})\big)^2\\
	\le &C\left(\mathbb{E}\big(\mathbf{v}_{1,i}^{\top}(\mathbf{B}_{i}\mathbf{B}_{i}^{\top}+\epsilon_n n\mathbf{I}_i)^{-1}\mathbf{v}_{1,i}-\frac{\mathbb{E}\operatorname{tr}(\mathbf{B}_{i(1)}\mathbf{B}_{i(1)}^{\top}+\epsilon_n n\mathbf{I}_i)^{-1}}{1+\mathbb{E}\operatorname{tr}(\mathbf{B}_{i(1)}\mathbf{B}_{i(1)}^{\top}+\epsilon_n n\mathbf{I}_i)^{-1}}\big)\right)^2\\
	&+C\big(\mathbb{E}(\mathbf{v}_{1,i}^{\top}(\mathbf{B}_{i}\mathbf{B}_{i}^{\top})^{-1}\mathbf{v}_{1,i})-\mathbb{E}\mathbf{v}_{1,i}^{\top}(\mathbf{B}_{i}\mathbf{B}_{i}^{\top}+\epsilon_n n\mathbf{I}_i)^{-1}\mathbf{v}_{1,i}\big)^2\\
	\le &C(\epsilon_{n}^{-2}n^{-1}+\mathcal{L}(n^{1/2})n^{(2-\alpha)/2})+C\epsilon_{n}^{2},
	\end{split}
\end{equation*}
where in the last line we used Jensen's inequality, and the results in \eqref{eq_delta1} and \eqref{eq_delta2}. As a result, we have
\begin{equation*}
	\delta_n\le C(\epsilon_{n}^2+\epsilon_{n}^{-2}n^{-1})+C\mathcal{L}(n^{1/2})n^{(2-\alpha)/2}
\end{equation*}
for $\alpha\ge 3$ and $n^{-1}\le \epsilon_{n}\le n^{-1/2}$. Choosing $\epsilon_{n}=n^{-1/4}$ finishes the proof of the lemma.
\end{proof}

\subsection{Maximum of diagonals of the projection matrices near singularity}

Now, we present the main results for the maximum of diagonals of the projection matrices $\mathbf{P}_i$, given the global truncation level at $n^{2/3}\log n$, similarly to Lemma 3.3 of \cite{bao2015logarithmic} and Lemma 3.5 of \cite{wang2018}.

\begin{lemma}\label{lemma_upperdiagonal}
    Under \eqref{eq_newcondition}, if  $n-p=\mathrm{O}(s_3)$, for $p-s_3\le i\le p-1$, we have
    \begin{equation}\label{eq_Emaxpll}
        \max_{p-s_3\le i\le p-1}\mathbb{E}\max_{1\le \ell\le n}p_{i,\ell\ell}\le C\log^{-a} n
    \end{equation}
    for some constant $C>0$.
    For $p-p^{2/3}\le i\le p-1$, we have
    \begin{equation}\label{eq_Emaxpilllarge}
        \max_{p-p^{2/3}\le i\le p-1}\mathbb{E}\max_{1\le \ell\le n}p_{i,\ell\ell}\le Cn^{-1/3}\log^{12a+4} n.
    \end{equation}
    Moreover, we have, with high probability,
    \begin{equation}\label{eq_maxpikklarge}
        \begin{split}
        \mathbb{P}(\max_{1\le \ell\le n}p_{i,\ell\ell}\ge n^{-1/8}\log s_2)\le Cn^{-3/4}\log^{10} n,
        \end{split}
    \end{equation}
    and
    \begin{equation}\label{eq_maxpikksmall}
        \begin{split}
        \mathbb{P}(\max_{1\le \ell\le n}p_{i,\ell\ell}\ge n^{-1/4}\log s_2)\le Cn^{-1/2}\log^{10} n,\\\mathbb{P}(\max_{1\le \ell\le n}p_{i,\ell\ell}\ge n^{-1/3}\log s_2)\le Cn^{-1/3}\log^{12a+4} n
        \end{split}
    \end{equation}
    for $p-p^{2/3}\le i\le p-1$.
\end{lemma}
\begin{proof}
Let $\epsilon_{n}$ be a sequence tending to zero arbitrarily slower than $n^{-1}$, which will be fixed later. Define
    \begin{equation*}
    G_i(\epsilon_n):=\big(\frac{1}{n}\mathbf{B}_{i}\mathbf{B}_{i}^{\top}+\epsilon_n\mathbf{I}_i\big)^{-1},~G_{i,\ell}(\epsilon_n):=\big(\frac{1}{n}\mathbf{B}_{i(\ell)}\mathbf{B}_{i(\ell)}^{\top}+\epsilon_n\mathbf{I}_i\big)^{-1},
    \end{equation*}
    where $\mathbf{I}_i$ denotes the $i\times i$ identity matrix. By \eqref{eq_rewrite}, we have
    \begin{equation}\label{eq_pillupper}
        p_{i,\ell\ell}\le \big(1+\frac{1}{n}\mathbf{v}_{\ell,i}^{\top}(\mathbf{B}_{i(\ell)}\mathbf{B}_{i(\ell)}^{\top}+\epsilon_n\mathbf{I}_i)^{-1}\mathbf{v}_{\ell,i}\big)^{-1}=\big(1+\frac{1}{n}\mathbf{v}_{\ell,i}^{\top}G_{i,\ell}(\epsilon_n)\mathbf{v}_{\ell,i}\big)^{-1}.
    \end{equation}
    For the first statement, it suffices to show
    \begin{equation*}
        \mathbb{E}\max_{1\le \ell\le n}\big(1+\frac{1}{n}\mathbf{v}_{\ell,i}^{\top}G_{i,\ell}(\epsilon_n)\mathbf{v}_{\ell,i}\big)^{-1}\le C\log ^{-a}p,~~p-s_3\le i\le p-1.
    \end{equation*}
    It is apparent that $G_i(\epsilon_n)$ and $G_{i,\ell}(\epsilon_n)$ are positive-definite and
    \begin{equation*}
        \|G_i(\epsilon_n)\|\le \epsilon_n^{-1},~\|G_{i,\ell}(\epsilon_n)\|\le \epsilon_n^{-1}.
    \end{equation*}
    Moreover, one has
    \begin{equation}\label{eq_diffiil}
        |\operatorname{tr}G_i(\epsilon_n)-\operatorname{tr}G_{i,\ell}(\epsilon_n)|=\frac{\frac{1}{n}\mathbf{v}_{\ell,i}^{\top}G_{i,\ell}(\epsilon_n)^2\mathbf{v}_{\ell,i}}{1+\frac{1}{n}\mathbf{v}_{\ell,i}^{\top}G_{i,\ell}(\epsilon_n)\mathbf{v}_{\ell,i}}\le \epsilon_n^{-1}
    \end{equation}
    by the Sherman-Morrison formula. Set
    \begin{equation}\label{eq_defsiepsn}
        s_i(\epsilon_n)=\mathbb{E}\frac{1}{n}\operatorname{tr}G_i(\epsilon_n),
    \end{equation}
    which satisfies
    \begin{equation*}
        s_i(\epsilon_n)=2\big(\epsilon_n+1-\frac{i}{n}+\sqrt{(\epsilon_n+1-\frac{i}{n})^2+\frac{4\epsilon_n i}{n}}\big)^{-1}(1+\mathrm{o}(1))
    \end{equation*}
    for $\epsilon_n\ge n^{-5/12}$ by Lemma \ref{lemma_esdconvergence}. Thus, for $p-s_3\le i\le p-1$, we get
    \begin{equation}\label{eq_sieps}
        s_i(\epsilon_n)\sim (\epsilon_n^{1/2}+(1-i/n))^{-1}.
    \end{equation}
    Moreover, define the indicator
    \begin{align*}
\chi(i)=\mathbbm{1}_{\left\{(1 / n) \operatorname{tr} G_{i}(\epsilon_n) \ge s_i(\epsilon_n)/10\right\}},
\end{align*}
and denote the ($u, v$)-th entry of $G_{i, \ell}(\epsilon_n)$ by $G_{i,\ell}(u, v)$ below. Recalling the Sherman-Morrison formula, we have
\begin{equation*}
    \operatorname{tr}G_i(\epsilon_n)-\operatorname{tr}G_{i,\ell}(\epsilon_n)=-\frac{\frac{1}{n}\mathbf{v}_{\ell,i}^{\top}G_{i,\ell}(\epsilon_n)^{2}\mathbf{v}_{\ell,i}}{1+\frac{1}{n}\mathbf{v}_{\ell,i}^{\top}G_{i,\ell}(\epsilon_n)\mathbf{v}_{\ell,i}}\le 0.
\end{equation*}
Moreover, for ease of presentation, when there is no confusion, we will omit the parameter $\epsilon_n$ from the notation $G_{i, \ell}(\epsilon_n)$ and $G_{i}(\epsilon_n)$. Define the events
\begin{equation*}
    \begin{split}
        D_1(i)=&\left\{\max _{1 \le \ell \le n}\left|\sum_{1 \le u \ne v \le i} G_{i,\ell}(u, v) X_{u \ell} X_{v \ell}\right| \le \log ^{-2a} n \cdot \operatorname{tr} G_{i, \ell}\right\},\\
        D_2(i)=&\left\{\sum_{j=1}^i G_{i, \ell}(j, j) X_{j \ell}^2>\frac{1}{20} \operatorname{tr} G_{i, \ell}, \right\}.
    \end{split}
\end{equation*}
Then we have
\begin{equation}\label{eq_decomp_pill}
\begin{aligned}
& \mathbb{E} \max_{1\le \ell\le n}\big(1+\frac{1}{n}\mathbf{v}_{\ell,i}^{\top}G_{i,\ell}(\epsilon_n)\mathbf{v}_{\ell,i}\big)^{-1} \\
\le & \mathbb{P}\left\{\frac{1}{n} \operatorname{tr} G_{i} \le s_i(\epsilon_n)/10\right\}+\mathbb{E} \chi(i) \max_{1\le \ell\le n}\left(1+\frac{1}{n}\mathbf{v}_{\ell,i}^{\top}G_{i,\ell}(\epsilon_n)\mathbf{v}_{\ell,i}\right)^{-1}[\mathbbm{1}(D_1(i))+\mathbbm{1}(D_1^c(i))] \\
\le & \mathbb{P}\left\{\frac{1}{n} \operatorname{tr} G_{i} \le s_i(\epsilon_n)/10\right\}+\mathbb{E} \chi(i) \max_{1\le \ell\le n}\left(1+\frac{1}{n} \sum_{j=1}^i G_{i, \ell}(j, j) X_{j \ell}^2-\frac{1}{n}\log ^{-2a} n \cdot \operatorname{tr} G_{i, \ell}\right)^{-1}[\mathbbm{1}(D_2(i))+\mathbbm{1}(D_2^c(i))]\\
& +\mathbb{P}\left\{\frac{1}{n} \max _{1 \le \ell \le n}\left|\sum_{1 \le u \ne v \le i} G_{i,\ell}(u, v) X_{u \ell} X_{v \ell}\right| \ge \frac{1}{n}\log ^{-2a} n \cdot \operatorname{tr} G_{i, \ell}, \frac{1}{n} \operatorname{tr} G_{i} \ge s_{i}(\epsilon_n)/10\right\} \\
\le & \mathbb{P}\left\{\frac{1}{n} \operatorname{tr} G_{i} \le s_i(\epsilon_n)/10\right\}+\mathbb{E} \chi(i) \max_{1\le \ell\le n}\left(1+\frac{1}{20}\frac{1}{n} \operatorname{tr} G_{i, \ell}-\frac{1}{n}\log ^{-2a} n \cdot \operatorname{tr} G_{i, \ell}\right)^{-1} \\
& +C \sum_{k=1}^n \mathbb{P}\left\{\sum_{j=1}^i G_{i, \ell}(j, j) X_{j \ell}^2<\frac{1}{20} \operatorname{tr} G_{i, \ell}, \frac{1}{n} \operatorname{tr} G_{i} \ge s_{i}(\epsilon_n)/10\right\} \\
& +\mathbb{P}\left\{\frac{1}{n} \max _{1 \le \ell \le n}\left|\sum_{1 \le u \ne v \le i} G_{i, \ell}(u, v) X_{u \ell} X_{v \ell}\right| \ge \frac{1}{n}\log ^{-2a} n \cdot \operatorname{tr} G_{i, \ell}, \frac{1}{n} \operatorname{tr} G_{i} \ge s_{i}(\epsilon_n)/10\right\} \\
=&: \mathcal{S}_1+\mathcal{S}_2+\mathcal{S}_3+\mathcal{S}_4,
\end{aligned}
\end{equation}
where the above last inequality follows from \eqref{eq_diffiil} and $n^{-1}\epsilon_n^{-1}\rightarrow 0$ by our choices for $\epsilon_n$. By Lemma \ref{lemma_esdconvergence}, we can estimate the first term $\mathcal{S}_1$ as follows. Note that by definition $s_i(\epsilon_n) \sim (\epsilon_n^{1/2}+1-i/n)^{-1}$ in \eqref{eq_sieps} for $p-s_3 \le i \le p-1$, we have
\begin{equation}\label{eq_maxpills1}
\begin{aligned}
\mathcal{S}_1=\mathbb{P}\left\{\frac{1}{n} \operatorname{tr} G_{i} \le s_i(\epsilon_n)/10\right\} & \le \mathbb{P}\left\{\left|\frac{1}{n} \operatorname{tr} G_{i}-\mathbb{E} \frac{1}{n} \operatorname{tr} G_{i}\right| \ge \frac{9}{10}s_i(\epsilon_n)\right\} \\
& \le C[s_i(\epsilon_n)]^{-2} \operatorname{Var}\left\{\frac{1}{n} \operatorname{tr} G_{i}\right\}=\mathrm{O}(n^{-1}\epsilon_n^{-2}[s_i(\epsilon_n)]^{-2}),
\end{aligned}
\end{equation}
by the variance estimation \eqref{eq_varepsn} in Lemma \ref{lemma_esdconvergence}.
For the second term $\mathcal{S}_2$ of \eqref{eq_decomp_pill}, with the definition of $\chi(i)$, it is obvious that
\begin{equation}\label{eq_maxpills2}
\mathcal{S}_2=\mathbb{E}{\chi(i)}\left(1+\frac{1}{20} \cdot \frac{1}{n} \operatorname{tr} G_{i,\ell}-\frac{1}{n}\log ^{-2a} n \cdot \operatorname{tr} G_{i, \ell}\right)^{-1} \le C[s_i(\epsilon_n)]^{-1}
\end{equation}
by $\operatorname{tr}G_i\le \operatorname{tr}G_{i,\ell}$.
Now we deal with the third term $\mathcal{S}_3$. Set
\begin{align*}
\hat{X}_{j \ell}=X_{j \ell} \mathbbm{1}_{\left\{\left|X_{j \ell}\right| \le \log n\right\}}, \quad \tilde{X}_{j \ell}=\frac{\hat{X}_{j \ell}-\mathbb{E} \hat{X}_{j \ell}}{\sqrt{\operatorname{Var}(\hat{X}_{j \ell})}}.
\end{align*}
Since $G_{i,\ell}(j, j)$'s are positive and $\hat{X}_{j \ell}^2 \le X_{j \ell}^2$, one has
\begin{align*}
\begin{aligned}
\mathcal{S}_3
 \le \sum_{\ell=1}^n \mathbb{P}\left\{\sum_{j=1}^i G_{i,\ell}(j, j) \hat{X}_{j \ell}^2<\frac{1}{20} \operatorname{tr} G_{i,\ell}, \frac{1}{n} \operatorname{tr} G_{i} \ge s_{i}(\epsilon_n)/10\right\} .
\end{aligned}
\end{align*}
Moreover, by \eqref{eq_newcondition} and Lemma \ref{lemma_truncatemoment}, it is easy to derive that
\begin{align*}
\mathbb{E} \hat{X}_{j \ell}\sim \mathrm{O}\big(\mathcal{L}(\log n)\log ^{(1-\alpha)}n \big)\le \log^{-1} n, \quad \operatorname{Var}(\hat{X}_{j \ell})=1+\mathrm{O}\big(\mathcal{L}(\log n)\log ^{(2-\alpha) } n\big) >1/2
\end{align*}
as $n\rightarrow \infty$. Consequently, $\tilde{X}_{j \ell}^2 \le 2 \hat{X}_{j \ell}^2+\mathrm{O}(\log^{-2} n)$
for sufficiently large $n$. Therefore, we have
\begin{equation}\label{eq_maxpills3}
\begin{aligned}
\mathcal{S}_3& \le \sum_{\ell=1}^n \mathbb{P}\left\{\sum_j G_{i,\ell}(j, j) \hat{X}_{j \ell}^2<\frac{1}{20} \operatorname{tr} G_{i,\ell}, \frac{1}{n} \operatorname{tr} G_{i} \ge s_{i}(\epsilon_n)/10\right\} \\
& \le \sum_{\ell=1}^n \mathbb{P}\left\{\frac{1}{n}\left|\sum_j G_{i,\ell}(j, j) \tilde{X}_{j \ell}^2-\operatorname{tr} G_{i,\ell}\right| \ge s_{i}(\epsilon_n)/20\right\} \\
& \le (\log^{10}n) n^{-4} [s_{i}(\epsilon_n)]^{-4}\sum_{k=1}^n \mathbb{E}\left[ \operatorname{tr} G_{i,\ell}^4+\left(\operatorname{tr} G_{i, \ell}^2\right)^2\right] \\
& \lesssim  (\log^{10}n)\big([s_{i}(\epsilon_n)]^{-3}n^{-2}\epsilon_n^{-3}+\epsilon_n^{-2}[s_{i}(\epsilon_n)]^{-4}n^{-1}([s_{i}(\epsilon_n)]^2+n^{-1}\epsilon_n^{-2})\big).
\end{aligned}
\end{equation}
In the third inequality we used the fact \eqref{eq_diffiil} and in the fourth inequality we used Lemma \ref{lemma_4momentG} and the fact $\mathbb{E}|\tilde{X}_{i j}|^{3+t}\lesssim \mathrm{O}(\mathcal{L}(\log n) \log^{t }n )\lesssim \log^{2t} n$ for any $t \ge 1$ by the definition of $\tilde{X}_{i j}$ and Lemma \ref{lemma_truncatemoment}. In the last inequality, we used the trivial bound $\|G_{i,\ell}\|\le \epsilon_n^{-1}$ and
\begin{equation*}
    \mathbb{E}(n^{-1}\operatorname{tr}G_{i,\ell})^2=(\mathbb{E}n^{-1}\operatorname{tr}G_{i,l})^2+\operatorname{Var}(n^{-1}\operatorname{tr}G_{i,\ell})\le C([s_{i}(\epsilon_n)]^2+n^{-1}\epsilon_n^{-2})
\end{equation*}
by Lemma \ref{lemma_esdconvergence}.
Now we begin to deal with the last term $\mathcal{S}_4$ of \eqref{eq_decomp_pill}. Note that by Lemma \ref{lemma_4momentG} for off-diagonals,
\begin{equation}\label{eq_maxpills4}
\begin{aligned}
\mathcal{S}_4
 &  \le \sum_{k=1}^n \mathbb{P}\left\{\frac{1}{n}\left|\sum_{u \neq v} G_{i,\ell}(u, v) X_{u \ell} X_{v \ell}\right| \ge s_{i}(\epsilon_n)\log ^{-3a} n\right\} \\
&  \le Cn^{-4}[s_{i}(\epsilon_n)]^{-4}\log ^{12a} n  \sum_{\ell=1}^n \mathbb{E}\left(\sum_{u \neq v} G_{i,\ell}(u, v) X_{u \ell} X_{v \ell}\right)^4 \\
&  \lesssim C (\mathbb{E}X_{j\ell}^4)^2 n^{-4}[s_{i}(\epsilon_n)]^{-4}\log ^{12a} n \sum_{\ell=1}^n \mathbb{E}(\|G_{i,\ell}\|^2\operatorname{tr} G_{i,\ell}^2) \\
& \lesssim (\log^{12a+2} n)[\mathcal{L}(n^{2/3}\log n)]^2 n^{2(1-\alpha)/3}\epsilon_n^{-3}[s_i(\epsilon_n)]^{-3}
\end{aligned}
\end{equation}
by $\mathbb{E}X_{j\ell}^4\lesssim \mathcal{L}(n^{2/3}\log n)n^{2(4-\alpha)/3}\log^{(4-\alpha)}n$ due to the global truncation $n^{2/3}\log n$.
Therefore, for all $p-s_3\le i\le p-1$, we have $\log^{a} n\le s_i(\epsilon_n)\le \epsilon_n^{-1/2}$, which further implies
\begin{equation*}
    \begin{split}
        \mathbb{E}\max_{1\le \ell\le n}(1+\frac{1}{n}\mathbf{v}_{\ell,i}G_{i,\ell}(\epsilon_n)\mathbf{v}_{\ell,i})^{-1}\lesssim [s_i(\epsilon_n)]^{-1}\lesssim \log^{-a} n,
    \end{split}
\end{equation*}
by choosing $\epsilon_n=n^{-1/4}$ which satisfies the condition of Lemma \ref{lemma_esdconvergence}, completing the proof of the first statement.

For the second claim, for $p-p^{2/3}\le i\le p-1$, setting $\epsilon_n=Cn^{-2/3}$ in Lemma \ref{lemma_esdconvergence} for some large constant $C>0$, we have $s_i(\epsilon_n)\asymp \epsilon_n^{-1/2}\sim n^{1/3}$ with high probability. Following a similar argument above, we get
\begin{equation*}
    \mathcal{S}_1=\mathrm{O}(n^{-1/3}), \mathcal{S}_2=\mathrm{O}(n^{-1/3}), \mathcal{S}_3= \mathrm{O}(n^{-1/3}\log^{10} n), \mathcal{S}_4=\mathrm{O}(n^{-1/3}\log^{12a+4}n),
\end{equation*}
which together with \eqref{eq_decomp_pill} give
\begin{equation*}
    \begin{split}
      \mathbb{E}\max_{1\le \ell\le n}(1+\frac{1}{n}\mathbf{v}_{\ell,i}G_{i,\ell}(\epsilon_n)\mathbf{v}_{\ell,i})^{-1}\lesssim \mathrm{O}(n^{-1/3}\log^{12a+4} n)
    \end{split}
\end{equation*}
for $p-p^{2/3}\le i\le p-1$ as desired.

It remains to consider the last three statements. On the one hand, for $p-p^{2/3}\le i\le p-1$,
setting $\epsilon_n=n^{-1/4}$ in Lemma \ref{lemma_esdconvergence} gives
\begin{equation*}
    s_{i}(\epsilon_n)=n^{1/8}+\mathrm{O}_{\prec}(1)\asymp n^{1/8},
\end{equation*}
with high probability. Recall $s_2=[-\log^{1/2}(1-(p-1)/n)]\rightarrow\infty$.
Following a similar argument in \eqref{eq_decomp_pill}, we have, for $\epsilon_n=n^{-1/4}$,
\begin{equation*}
\begin{aligned}
&\mathbb{P}(\max_{1\le \ell\le n}p_{i,\ell\ell}\ge n^{-1/8}\log s_2)\le  \mathbb{P} \big(\max_{1\le \ell\le n}\big(1+\frac{1}{n}\mathbf{v}_{\ell,i}^{\top}G_{i,\ell}(\epsilon_n)\mathbf{v}_{\ell,i}\big)^{-1}\ge n^{-1/8}\log s_2\big) \\
\le & \mathbb{P}\left\{\frac{1}{n} \operatorname{tr} G_{i} \le s_i(\epsilon_n)/10\right\}\\
&+\mathbb{P} \left( \max_{1\le \ell\le n}\left(1+\frac{1}{20}\frac{1}{n} \operatorname{tr} G_{i, \ell}-\frac{1}{n}\log ^{-2a} n \cdot \operatorname{tr} G_{i, \ell}\right)^{-1}\ge n^{-1/8}\log s_2, \frac{1}{n} \operatorname{tr} G_{i} \ge s_{i}(\epsilon_n)/10\right) \\
& +C \sum_{k=1}^n \mathbb{P}\left\{\sum_{j=1}^i G_{i, \ell}(j, j) X_{j \ell}^2<\frac{1}{20} \operatorname{tr} G_{i, \ell}, \frac{1}{n} \operatorname{tr} G_{i} \ge s_{i}(\epsilon_n)/10\right\} \\
& +\mathbb{P}\left\{\frac{1}{n} \max _{1 \le \ell \le n}\left|\sum_{1 \le u \ne v \le i} G_{i, \ell}(u, v) X_{u \ell} X_{v \ell}\right| \ge \frac{1}{n}\log ^{-2a} n \cdot \operatorname{tr} G_{i, \ell}, \frac{1}{n} \operatorname{tr} G_{i} \ge s_{i}(\epsilon_n)/10\right\} \\
=&: \mathcal{S}_1+\mathcal{S}_5+\mathcal{S}_3+\mathcal{S}_4,
\end{aligned}
\end{equation*}
where
\begin{equation*}
    \mathcal{S}_5\le \mathbb{P}(Cn^{-1/8}\ge n^{-1/8}\log s_2, \frac{1}{n}\operatorname{tr}G_i\ge n^{1/8}/10)=0
\end{equation*}
since $s_2\rightarrow \infty$ for sufficiently large $n$. Note that the case for $\epsilon_n=n^{-1/4}$ implies
\begin{equation*}
    \mathcal{S}_1=\mathrm{O}(n^{-3/4}), \mathcal{S}_3=\mathrm{O}(n^{-3/4}\log^{10}n), \mathcal{S}_4=\mathrm{o}(n^{-3/4}),
\end{equation*}
by \eqref{eq_maxpills1}, \eqref{eq_maxpills3}, and \eqref{eq_maxpills4}. Thus, one has
\begin{equation*}
\mathbb{P}(\max_{1\le \ell\le n}p_{i,\ell\ell}\ge n^{-1/8}\log s_2)\le \mathrm{O}(n^{-3/4}\log^{10} n)
\end{equation*}
 as desired.
Similarly, setting $\epsilon_n=n^{-1/4}$ and $\epsilon_n=Cn^{-2/3}$, we get
\begin{equation*}
\mathbb{P}(\max_{1\le \ell\le n}p_{i,\ell\ell}\ge n^{-1/4}\log s_2)\le \mathrm{O}(n^{-1/2}\log^{10} n)
\end{equation*}
and
\begin{equation*}
    \mathbb{P}(\max_{1\le \ell\le n}p_{i,\ell\ell}\ge n^{-1/3} \log s_2)\le \mathrm{O}(n^{-1/3}\log^{12a+4} n)
\end{equation*}
for $p-p^{2/3}\le i\le p-1$, respectively, completing the proof of this proposition.
\end{proof}

\section{Convergence rate of expected spectral distributions}\label{sec_esd}
This section investigates the estimation and convergence rate of expected spectral distributions for the heavy-tailed case near singularity, which is a key input for the proof of Lemma \ref{lemma_upperdiagonal}. Let $\epsilon_{n}$ be a sequence tending to zero arbitrarily slower than $n^{-1}$, which will be fixed later. Define
\begin{equation*}
G_i(\epsilon_n):=\big(\frac{1}{n}\mathbf{B}_{i}\mathbf{B}_{i}^{\top}+\epsilon_n\mathbf{I}_i\big)^{-1},~G_{i,\ell}(\epsilon_n):=\big(\frac{1}{n}\mathbf{B}_{i(\ell)}\mathbf{B}_{i(\ell)}^{\top}+\epsilon_n\mathbf{I}_i\big)^{-1},
\end{equation*}
where $\mathbf{I}_i$ denotes the $i\times i$ identity matrix. For simplicity, we denote
\begin{equation}\label{eq_defGil0}
G_i(0):=\big(\frac{1}{n}\mathbf{B}_{i}\mathbf{B}_{i}^{\top}\big)^{-1},~G_{i,\ell}(0):=\big(\frac{1}{n}\mathbf{B}_{i(\ell)}\mathbf{B}_{i(\ell)}^{\top}\big)^{-1}.
\end{equation}
Our proof is based on the resampling techniques and the method by \cite{bai2010spectral}, which relies on the concentration of the quadratic forms.

We begin with a precise estimate of the least eigenvalue of $\mathbf{B}_{i}\mathbf{B}_{i}^{\top}$ for $i/n\rightarrow 1$ when $\alpha>2$, which is illustrated by the following lemma due to \cite[Theorem 2]{dabagia2024smallest},
\begin{lemma}\label{lemma_smallesteigenvalue}
	Let $\mathbf{A}\in \mathbb{R}^{N\times n}$ be a random matrix with mean-zero isotropic columns and independent entries $\xi$ satisfying $\mathbb{E}|\xi|^{2+\beta}<D$. The smallest singular values $\sigma_n(\mathbf{A})$ satisfies
	\begin{equation*}
	\mathbb{P}(\sigma_{n}(\mathbf{A})\le \epsilon_0(\sqrt{N+1}-\sqrt{n}))\le (C\epsilon_0)^{N-n+1}+e^{-cN}
	\end{equation*}
	for all $\epsilon_0>0$, where $c,C$ depend only on $\beta$ and $D$.
\end{lemma}
Choosing $\epsilon_0 \in (1/(4C),1/(2C))$ (since $C$ depends only on $\beta$ and $D$, which is constant in our setting $\alpha\ge 3$) in Lemma \ref{lemma_smallesteigenvalue}, one has
\begin{equation*}
\mathbb{P}(\lambda_{\min}(\mathbf{B}_{i}\mathbf{B}_{i}^{\top})\le \epsilon_0^2 (\sqrt{n+1}-\sqrt{i})^2)\le 2^{-n+i}+e^{-cn},
\end{equation*}
which implies that with high probability, $\lambda_{\min}(n^{-1}\mathbf{B}_{i(\ell)}\mathbf{B}_{i(\ell)}^{\top})\ge c(1-\sqrt{i/n})^2\sim (1-i/n)^2$ as $i/n\rightarrow 1$. Let $\lambda_{1}\ge \lambda_{2}\ge \cdots\ge \lambda_{p}$ be the ordered eigenvalues of $n^{-1}XX^{\top}$. One has $\lambda_p\ge c(1-(p-1)/n)^2$ with high probability for some small $c>0$ by Lemma \ref{lemma_smallesteigenvalue} under the assumption that $n-p\rightarrow \infty$. One can obtain that
\begin{equation*}
\operatorname{tr}(G(0))=\sum_{k=1}^{p}\frac{1}{\lambda_{k}}\ge \sum_{k=1}^{p}\frac{1}{\lambda_{k}+\epsilon_{n}}=\operatorname{tr}(G(\epsilon_{n})),
\end{equation*}
which further gives
\begin{equation*}
\operatorname{tr}(G(0))-\operatorname{tr}(G(\epsilon_{n}))=\sum_{k=1}^{p}\frac{\epsilon_{n}}{(\lambda_{k}+\epsilon_{n})\lambda_{k}}\le \sum_{k=1}^{p}\frac{1}{\lambda_{k}+\epsilon_{n}}=\operatorname{tr}(G(\epsilon_{n})),
\end{equation*}
whenever we choose $0<\epsilon_{n}\le \lambda_p\asymp (1-p/n)^2$. This implies that $\operatorname{tr}G(0)$ and $\operatorname{tr}G(\epsilon_n)$ have the same order. Since one can not get an estimate of $\operatorname{tr}(G(0))$ directly for the near-singularity case, we can get a good estimate from $\operatorname{tr}(G(\epsilon_{n}))$ by choosing an appropriate and small enough $\epsilon_{n}$, which further provides a good estimate for $\operatorname{tr}G(0)$. Moreover, notice the monotonicity that, for $\epsilon_n^{\prime}\ge \epsilon_n^{\prime\prime}\ge n^{-1}$,
\begin{equation}\label{eq_monotonicityG}
    \operatorname{tr}G(\epsilon_n^{\prime})\le \operatorname{tr}G(\epsilon_n^{\prime\prime})\le \operatorname{tr}G(0).
\end{equation}
\begin{lemma}\label{lemma_esdconvergence}
	Let $X=(X_{ij})_{p\times n}$ be a random matrix with i.i.d. entries following Assumption \ref{assump1} with \eqref{eq_newcondition}, where $n/2\le p\le n$. Set $y_p=p/n$. Let $G(\epsilon_n)=(n^{-1}XX^{\top}+\epsilon_n\mathbf{I}_p)^{-1}$ with $\epsilon_{n}\gg n^{-5/12}$. We have, with high probability,
	\begin{equation}\label{eq_meanesd}
		\mathbb{E}(p^{-1}\operatorname{tr}G(\epsilon_n))=s_{p}(\epsilon_n)+\mathrm{O}_{\prec}\big(n^{-5/6}(\log^4 n)(\epsilon_n+(1-y_p)^2)^{-2}\big),
	\end{equation}
	where
	\begin{equation*}
		s_{p}(\epsilon_n)=2\left(\epsilon_n+1-\frac{p}{n}+\sqrt{(\epsilon_n+1-\frac{p}{n})^{2}+\frac{4\epsilon_{n} p}{n}}\right)^{-1}.
	\end{equation*}
    For $n/2\le p\le n-n^{19/20}$ and $n^{-5/6+1/20}\le \epsilon_n\le n^{-1/10}$, we have
    \begin{equation}\label{eq_smallepsn}
        \mathbb{E}(p^{-1}\operatorname{tr}G(\epsilon_n))=s_{p}(\epsilon_n)+\mathrm{O}_{\prec}(n^{-5/6}\log^4 n\cdot(1-y_p)^{-4}).
    \end{equation}
    For $n-n^{2/3}\le p\le n$ with $n^{-2/3}\le \epsilon_n\ll n^{-4/9}$, we have
\begin{equation}\label{eq_largeepsn}
   \mathbb{E}(p^{-1}\operatorname{tr}G(\epsilon_n)) =s_p(\epsilon_n)+\mathrm{O}_{\prec}(n^{-1}\epsilon_n^{-2}) \asymp \epsilon_n^{-1/2}.
\end{equation}
    Moreover, for any $\epsilon_n\gg n^{-1}$, we have
    \begin{equation}\label{eq_varepsn}
		\operatorname{Var}(p^{-1}\operatorname{tr}G(\epsilon_n))=\mathrm{O}(n^{-1}\epsilon_n^{-2}).
	\end{equation}
\end{lemma}
\begin{proof}
	Define
	\begin{equation*}
	\begin{split}
	r_{p}(\epsilon_n)=\mathbb{E}\frac{\operatorname{tr}G(\epsilon_n)}{p}
	=\frac{1}{p}\mathbb{E}\sum_{k=1}^{p}\frac{1}{n^{-1}\mathbf{x}_k\mathbf{x}_k^{\top}+\epsilon_n-n^{-2}\mathbf{x}_kX_{(k)}^{\top}[n^{-1}X_{(k)}X_{(k)}^{\top}+\epsilon_n\mathbf{I}_{p-1}]^{-1}X_{(k)}\mathbf{x}_k^{\top}},
	\end{split}
	\end{equation*}
	where $\mathbf{x}_k$ is the $k$-th row of $X$ and $X_{(k)}$ is the submatrix of $X$ by deleting its $k$-th row. Let $y_p=p/n$. Then we can rewrite $r_p(\epsilon_n)$ as
	\begin{equation}\label{eq_rpeps}
		r_p(\epsilon_n)=\frac{1}{p}\sum_{k=1}^{p}\mathbb{E}\frac{1}{a_k+1+\epsilon_n-y_p+y_p\epsilon_n r_p(\epsilon_n)}=\frac{1}{1+\epsilon_n-y_p+y_p\epsilon_n r_p(\epsilon_n)}+b,
	\end{equation}
	where
	\begin{equation*}
		a_k=\frac{1}{n}\sum_{j=1}^{n}(X_{kj}^2-1)+y_p-y_p\epsilon_n r_p(\epsilon_n)-n^{-2}\mathbf{x}_kX_{(k)}^{\top}[n^{-1}X_{(k)}X_{(k)}^{\top}+\epsilon_n\mathbf{I}_{p-1}]^{-1}X_{(k)}\mathbf{x}_k^{\top}
	\end{equation*}
	and
	\begin{equation}\label{eq_bpeps}
		\begin{split}
		b=&b_p=-\frac{1}{p}\sum_{k=1}^{p}\mathbb{E}\frac{a_k}{(1+\epsilon_n-y_p+y_p\epsilon_n r_p(\epsilon_n))(1+\epsilon_n-y_p+y_p\epsilon_n r_p(\epsilon_n)+a_k))}\\
		=&-\frac{1}{p}\sum_{k=1}^{p}\mathbb{E}\frac{a_k}{(1+\epsilon_n-y_p+y_p\epsilon_n r_p(\epsilon_n))^2}\\
		&-\frac{1}{p}\sum_{k=1}^{p}\mathbb{E}\frac{a_k^2}{(1+\epsilon_n-y_p+y_p\epsilon_n r_p(\epsilon_n))^2(1+\epsilon_n-y_p+y_p\epsilon_n r_p(\epsilon_n)+a_k))}.
		\end{split}
	\end{equation}
 It follows that
 \begin{equation}\label{eq_sqrpeps}
     \begin{split}
         r_{p}(\epsilon_n)=\frac{1}{2y_p\epsilon_n}\big(\sqrt{(1+\epsilon_n-y_p+y_p\epsilon_n b )^2+4y_p\epsilon_n}-(1+\epsilon_n-y_p+y_p\epsilon_n b )\big).
     \end{split}
 \end{equation}
 It can be verified that
 \begin{equation*}
     a_k+1-y_p+y_p\epsilon_n r_p(\epsilon_n)=\frac{1}{n}\mathbf{x}_k\big(\mathbf{I}_n-n^{-1}X_{(k)}^{\top}[n^{-1}X_{(k)}X_{(k)}^{\top}+\epsilon_n\mathbf{I}_{p-1}]^{-1}X_{(k)}\big)\mathbf{x}_k^{\top}\ge 0,
 \end{equation*}
 which further gives
 \begin{equation*}
     \frac{1}{|1+\epsilon_n-y_p+y_p\epsilon_n r_p(\epsilon_n)+a_k|}\le \epsilon_n^{-1}.
 \end{equation*}
 Thus, by \eqref{eq_bpeps}, we have
 \begin{equation}\label{eq_absbeps}
     \begin{split}
         |b|\le \frac{1}{p}\sum_{k=1}^{p}(|\mathbb{E}a_k|+\epsilon_{n}^{-1}\mathbb{E}a_k^2)(1+\epsilon_n-y_p+y_p\epsilon_n r_p(\epsilon_n))^{-2}.
     \end{split}
 \end{equation}
 By the definition of $a_k$, one has
 \begin{equation*}
     \begin{split}
         \mathbb{E}a_k=&\mathbb{E}\big(y_p-\epsilon_{n}\frac{\operatorname{tr}G(\epsilon_n)}{n}-\frac{1}{n}\operatorname{tr}[n^{-1}X_{(k)}X_{(k)}^{\top}+\epsilon_n \mathbf{I}_{p-1}]^{-1}\frac{1}{n}X_{(k)}X_{(k)}^{\top}\big)\\
         = &\frac{1}{n}-\frac{\epsilon_n}{n}\mathbb{E}\big[\operatorname{tr}\big(\frac{1}{n}XX^{\top}+\epsilon_n \mathbf{I}_p\big)^{-1}-\operatorname{tr}\big(\frac{1}{n}X_{(k)}X_{(k)}^{\top}+\epsilon_n \mathbf{I}_{p-1}\big)^{-1}\big]\ge 0.
     \end{split}
 \end{equation*}
Note the Sherman-Morrison formula
 \begin{equation*}
     \operatorname{tr}\big(\frac{1}{n}XX^{\top}+\epsilon_n \mathbf{I}_p\big)^{-1}-\operatorname{tr}\big(\frac{1}{n}X_{(k)}X_{(k)}^{\top}+\epsilon_n \mathbf{I}_{p-1}\big)^{-1}=-\frac{\frac{1}{n}\mathbf{x}_k\big(\frac{1}{n}X_{(k)}X_{(k)}^{\top}+\epsilon_n \mathbf{I}_{p-1}\big)^{-2}\mathbf{x}_k^{\top}}{1+\frac{1}{n}\mathbf{x}_k\big(\frac{1}{n}X_{(k)}X_{(k)}^{\top}+\epsilon_n \mathbf{I}_{p-1}\big)^{-1}\mathbf{x}_k^{\top}}\le 0,
 \end{equation*}
So, $|\operatorname{tr}\big(\frac{1}{n}XX^{\top}+\epsilon_n \mathbf{I}_p\big)^{-1}-\operatorname{tr}\big(\frac{1}{n}X_{(k)}X_{(k)}^{\top}+\epsilon_n \mathbf{I}_{p-1}\big)^{-1}|\le \epsilon_n^{-1}$. This further implies the following bound
 \begin{equation*}
     n^{-1}\le \mathbb{E}a_k\le 2n^{-1}
 \end{equation*}
and $b\le 0$ since both terms of \eqref{eq_bpeps} are negative due to $\mathbb{E}a_k\ge 0$ and $1+\epsilon_n-y_p+y_p\epsilon_n r_p(\epsilon_n)+a_k\ge \epsilon_n$. Now we come to $\mathbb{E}a_k^2$, which can be estimated by
 \begin{equation}\label{eq_ak2}
     \begin{split}
         \mathbb{E}a_k^2=\operatorname{Var}(a_k)+(\mathbb{E}a_k)^2= \operatorname{Var}(a_k)+\mathrm{O}(n^{-2}).
     \end{split}
 \end{equation}
 Let
\begin{equation*}
    \Gamma_k=(\gamma_{ij}(k))=\frac{1}{n}X_{(k)}^{\top}[n^{-1}X_{(k)}X_{(k)}^{\top}+\epsilon_n\mathbf{I}_{p-1}]^{-1}X_{(k)},
\end{equation*}
which satisfies $\Gamma_k\prec \mathbf{I}_n$.
 Consider the following decomposition
 \begin{equation}\label{eq_decomak}
     \begin{split}
         a_k-\mathbb{E}a_k
         =&\frac{1}{n}\sum_{j=1}^{n}(X_{kj}^2-1)-\big(n^{-1}\mathbf{x}_k\Gamma_k\mathbf{x}_k^{\top}-n^{-1}\operatorname{tr}\Gamma_k\big)-(n^{-1}\operatorname{tr}\Gamma_k-\mathbb{E}n^{-1}\operatorname{tr}\Gamma_k)\\
         =&:T_1+T_2+T_3
     \end{split}
 \end{equation}
 since $\mathbb{E}\big(n^{-2}\mathbf{x}_kX_{(k)}^{\top}[n^{-1}X_{(k)}X_{(k)}^{\top}+\epsilon_n\mathbf{I}_{p-1}]^{-1}X_{(k)}\mathbf{x}_k^{\top}\big)=\mathbb{E}n^{-1}\operatorname{tr}\Gamma_k$. Specifically, consider the following label matrix $\Psi=:(\psi_{ij})$ with the resampling level $n^{c_{\alpha}}$, where $c_{\alpha}\in (0,2/3)$ which will be fixed later. Similarly, given $1\le k\le p$, there are at most $n^{1-\alpha c_{\alpha}+\epsilon_{\alpha}}$ entries of $|X_{kj}|$ larger than $n^{c_{\alpha}}$ with high probability. Applying Lemma \ref{lemma_quadratic3} for $\mathbf{A}=\mathbf{I}_n$ and $\mathbf{A}=\Gamma_k$ respectively, we have, with high probability,
\[
\mathbb{E}T_1^2+\mathbb{E}T_2^2\lesssim \mathcal{L}(n^{c_{\alpha}})n^{(4-\alpha)c_{\alpha}-1}+n^{-4/3}\log^2 n\cdot(\log n+n^{1-\alpha c_{\alpha}+\epsilon_{\alpha}}).
\]
Note the identity
\begin{equation}\label{eq_trGammak}
    \begin{split}
        \operatorname{tr}\Gamma_k
        =\operatorname{tr}(\mathbf{I}_{p-1}-\epsilon_n[n^{-1}X_{(k)}X_{(k)}^{\top}+\epsilon_n\mathbf{I}_{p-1}]^{-1})
        =p-1-\epsilon_{n}\operatorname{tr}\big(n^{-1}X_{(k)}X_{(k)}^{\top}+\epsilon_n\mathbf{I}_{p-1}\big)^{-1},
    \end{split}
\end{equation}
which further implies
\begin{equation*}
T_3=-n^{-1}\epsilon_{n}\big(\operatorname{tr}\big(n^{-1}X_{(k)}X_{(k)}^{\top}+\epsilon_n\mathbf{I}_{p-1}\big)^{-1}-\mathbb{E}\operatorname{tr}\big(n^{-1}X_{(k)}X_{(k)}^{\top}+\epsilon_n\mathbf{I}_{p-1}\big)^{-1}\big)
\end{equation*}
and thus
 \begin{equation*}
     \mathbb{E}T_3^2=\frac{\epsilon_n^2}{n^2}\mathbb{E}\big|\operatorname{tr}\big(n^{-1}X_{(k)}X_{(k)}^{\top}+\epsilon_n\mathbf{I}_{p-1}\big)^{-1}-\mathbb{E}\operatorname{tr}\big(n^{-1}X_{(k)}X_{(k)}^{\top}+\epsilon_n\mathbf{I}_{p-1}\big)^{-1}\big|^2.
 \end{equation*}
Let $\mathbb{E}_d$ be the conditional expectation given $\{X_{ij},d+1\le i\le p,1\le j\le n\}$. Define
\begin{equation*}
    \begin{split}
        \gamma_d(k)=&\mathbb{E}_{d-1}\operatorname{tr}\big(n^{-1}X_{(k)}X_{(k)}^{\top}+\epsilon_n\mathbf{I}_{p-1}\big)^{-1}-\mathbb{E}_{d}\operatorname{tr}\big(n^{-1}X_{(k)}X_{(k)}^{\top}+\epsilon_n\mathbf{I}_{p-1}\big)^{-1}\\
        =&\mathbb{E}_{d-1}g_d(k)-\mathbb{E}_{d}g_d(k),~~d=1,2,\ldots,p,
    \end{split}
\end{equation*}
where
\begin{equation*}
    g_d(k)=\operatorname{tr}\big(n^{-1}X_{(k)}X_{(k)}^{\top}+\epsilon_n\mathbf{I}_{p-1}\big)^{-1}-\operatorname{tr}\big(n^{-1}X_{(k,d)}X_{(k,d)}^{\top}+\epsilon_n\mathbf{I}_{p-2}\big)^{-1}
\end{equation*}
which satisfies $|g_d(k)|\le \epsilon_n^{-1}$. It follows that
\begin{equation*}
    \operatorname{tr}\big(n^{-1}X_{(k)}X_{(k)}^{\top}+\epsilon_n\mathbf{I}_{p-1}\big)^{-1}-\mathbb{E}\operatorname{tr}\big(n^{-1}X_{(k)}X_{(k)}^{\top}+\epsilon_n\mathbf{I}_{p-1}\big)^{-1}=\sum_{d=1}^{p}\gamma_d(k),
\end{equation*}
which further gives
\begin{equation*}
    \mathbb{E}T_3^2\le \frac{\epsilon_n^{2}}{n^2}\sum_{d=1}^{p}\mathbb{E}|\gamma_d(k)|^2\le \mathrm{O}(n^{-1}).
\end{equation*}
Therefore, under \eqref{eq_newcondition}, we can set $c_{\alpha}=1/6$ to get a nearly optimal bound as
\begin{equation*}
    \operatorname{Var}(a_k)\lesssim n^{-5/6}\log^4 n
\end{equation*}
with high probability, which further gives
\begin{equation}\label{eq_ak21}
    \mathbb{E}a_k^2\lesssim n^{-5/6}\log^4 n
\end{equation}
with high probability. Substituting $|\mathbb{E}a_k|=\mathrm{O}(n^{-1})$ and \eqref{eq_ak21} into \eqref{eq_absbeps} gives
\begin{equation}\label{eq_estimationb}
    |b|\lesssim n^{-5/6}(\log^4 n)\cdot \epsilon_n^{-1}(1-y_p+\epsilon_n+y_p\epsilon_n r_p(\epsilon_n))^{-2}.
\end{equation}
Notice the elementary equality
\begin{equation*}
\begin{split}
\sqrt{(x+y)^2+z}-(x+y)-(\sqrt{x^2+z}-x)
=\frac{-2yz}{(\sqrt{(x+y)^2+z}+\sqrt{x^2+z}+y)(x+\sqrt{x^2+z})}.
\end{split}
\end{equation*}
Then by \eqref{eq_sqrpeps} with $x=(1-y_p+\epsilon_n)\ge \epsilon_n> 0, y=y_p\epsilon_n b\le 0$ and $z=4y_p\epsilon_n\ge 0$, we have
\begin{equation*}
\begin{split}
(\sqrt{(x+y)^2+z})^2-(x-y)^2=4xy+z=4y_p\epsilon_n[(1-y_p+\epsilon_n)b+1]\ge 0
\end{split}
\end{equation*}
since $1+(1-y_p+\epsilon_n)b\ge 1-\mathrm{O}_{\prec}(n^{-5/6}(\log^4 n)\cdot\epsilon_n^{-1}(1-y_p+\epsilon_n+y_p\epsilon_n r_p(\epsilon_n))^{-1})\ge 0$ due to the trivial bound that $|b|\lesssim \mathrm{O}\big(n^{-5/6}(\log^4 n)\cdot\epsilon_n^{-1}(1-y_p+\epsilon_n)^{-2}\big)$  with $\epsilon_n\gg n^{-5/12}$, which further gives
\begin{equation*}
\sqrt{(x+y)^2+z}+y\ge x
\end{equation*}
for our choices of $x,y,z$. So, we have
\begin{equation*}
    0\le \sqrt{(x+y)^2+z}-(x+y)-(\sqrt{x^2+z}-x)\le \mathrm{O}(\frac{yz}{x^2+z})
\end{equation*}
since $y<0$. This observation implies, with high probability,
\begin{equation*}
\begin{split}
r_p(\epsilon_n)=&\frac{1}{2y_p\epsilon_n}\big(\sqrt{(1+\epsilon_n-y_p )^2+4y_p\epsilon_n}-(1+\epsilon_n-y_p)\big)+\mathrm{O}\big(\frac{y_p\epsilon_n |b|}{(1-y_p+\epsilon_n)^2+4y_p\epsilon_n}\big)\\
=&2\big(\sqrt{(1+\epsilon_n-y_p )^2+4y_p\epsilon_n}+(1+\epsilon_n-y_p)\big)^{-1}+\mathrm{O}_{\prec}\big(\frac{n^{-5/6}(\log^4 n)}{((1-y_p)^2+\epsilon_n)(1-y_p+\epsilon_n+y_p\epsilon_n r_p(\epsilon_n))^2}\big)
\end{split}
\end{equation*}
by \eqref{eq_estimationb}. Now we consider the different cases for $y_p$. For the case of $1-y_p\ge \epsilon_n^{1/2}$, we have $r_p(\epsilon_n)\asymp (1-y_p)^{-1}$, which further gives
\begin{equation*}
r_{p}(\epsilon_n)=2\big(\sqrt{(1+\epsilon_n-y_p )^2+4y_p\epsilon_n}+(1+\epsilon_n-y_p)\big)^{-1}+\mathrm{O}_{\prec}\big(n^{-5/6}(\log^4 n)(1-y_p)^{-4}\big).
\end{equation*}
For the case of $1-y_p\le \epsilon_n^{1/2}$, we have $r_p(\epsilon_n)\asymp \epsilon_n^{-1/2}$, which further implies
\begin{equation*}
r_{p}(\epsilon_n)=2\big(\sqrt{(1+\epsilon_n-y_p )^2+4y_p\epsilon_n}+(1+\epsilon_n-y_p)\big)^{-1}+\mathrm{O}_{\prec}\big(n^{-5/6}(\log^4 n)\epsilon_n^{-2}\big).
\end{equation*}
Thus, we have
\begin{equation*}
r_{p}(\epsilon_n)=2\big(\sqrt{(1+\epsilon_n-y_p )^2+4y_p\epsilon_n}+(1+\epsilon_n-y_p)\big)^{-1}+\mathrm{O}_{\prec}\big(n^{-5/6}(\log^4 n)(\epsilon_n+(1-y_p)^2)^{-2}\big),
\end{equation*}
with high probability, which completes the proof of the first statement.

Now, we consider the near-singularity case, say $p/n\rightarrow 1$. For the first part $n/2\le p\le n-n^{19/20}$, following the similar argument for the estimation of $|b|$ from \eqref{eq_estimationb} above, if $\epsilon_n\le n^{-1/10}\le (1-y_p)^2$, we have
\begin{equation*}
    |b|\lesssim \mathrm{O}(n^{-5/6}\log^4 n\epsilon_n^{-1}(1-y_p)^{-2}),
\end{equation*}
which satisfies
\begin{equation*}
    1+(1-y_p+\epsilon_n)b\ge 1-\mathrm{O}_{\prec}(n^{-5/6}\log^4 n\epsilon_n^{-1}(1-y_p)^{-1})\ge 1/2
\end{equation*}
if we choose $\epsilon_n$ such that $n^{-1/10}\ge \epsilon_n\ge n^{-5/6}(1-y_p)^{-1}$. Therefore, we can get
\begin{equation*}
\begin{split}
    r_p(\epsilon_n)=&\frac{1}{2y_p\epsilon_n}\big(\sqrt{(1+\epsilon_n-y_p )^2+4y_p\epsilon_n}-(1+\epsilon_n-y_p)\big)+\mathrm{O}\big(\frac{y_p\epsilon_n |b|}{(1-y_p+\epsilon_n)^2+4y_p\epsilon_n}\big)\\
=&2\big(\sqrt{(1+\epsilon_n-y_p )^2+4y_p\epsilon_n}+(1+\epsilon_n-y_p)\big)^{-1}+\mathrm{O}_{\prec}\big(\frac{n^{-5/6}\log^4 n}{(1-y_p)^4}\big)
\end{split}
\end{equation*}
for $n/2\le p\le n-n^{19/20}$ and $n^{-5/6+1/20}\le \epsilon_n\le n^{-1/10}$, finishing the proof of the second statement.

Next, we use the bootstrap strategy to get an estimation for $ n-n^{2/3}\le p\le n$, which implies $(1-p/n)^2\le n^{-2/3}$. In the sequel, we consider $\epsilon_n\gg n^{-2/3}$.
Notice the decomposition
\begin{equation}\label{eq_decomak2}
a_k-\mathbb{E}a_k=(n^{-1}\mathbf{x}_k(\mathbf{I}_n-\Gamma_k)\mathbf{x}_k^{\top}-n^{-1}\operatorname{tr}(\mathbf{I}_n-\Gamma_k))-(n^{-1}\operatorname{tr}\Gamma_k-\mathbb{E}n^{-1}\operatorname{tr}\Gamma_k)=:(T_1+T_2)+T_3.
\end{equation}
Moreover, it can be checked that $\|\mathbf{I}_n-\Gamma_k\|\le 1+\|\Gamma_k\|\le 2$ and
\begin{equation*}
    \operatorname{tr}(\mathbf{I}_n-\Gamma_k)=n-p+1+\epsilon_n\operatorname{tr}(n^{-1}X_{(k)}X_{(k)}^{\top}+\epsilon_n \mathbf{I}_{p-1})\le n-p+2+\epsilon_n \operatorname{tr}\big(\frac{1}{n}XX^{\top}+\epsilon_n \mathbf{I}_p\big)^{-1}
\end{equation*}
by \eqref{eq_trGammak} and the Sherman-Morrison formula
\[
-\epsilon_n^{-1}\le \operatorname{tr}\big(\frac{1}{n}XX^{\top}+\epsilon_n \mathbf{I}_p\big)^{-1}-\operatorname{tr}\big(\frac{1}{n}X_{(k)}X_{(k)}^{\top}+\epsilon_n \mathbf{I}_{p-1}\big)^{-1}\le 0.
\]
Applying Lemma \ref{lemma_quadratic3} to the positive definite matrix $\mathbf{A}=\mathbf{I}_n-\Gamma_k$ gives, with high probability,
\begin{equation*}
    \begin{split}
        \mathbb{E}(T_1+T_2)^2\lesssim &\mathcal{L}(n^{c_{\alpha}})n^{(4-\alpha)c_{\alpha}-2}\mathbb{E}\operatorname{tr}(\mathbf{I}_n-\Gamma_k)+n^{-4/3}\log^2 n\cdot(\log n+n^{1-\alpha c_{\alpha}+\epsilon_{\alpha}})+\mathrm{o}(n^{-4/3})\\
        \le & n^{c_{\alpha}-1}(1-y_p+\epsilon_n r_p(\epsilon_n))+n^{-1/3-3c_{\alpha}}\log^4 n+\mathrm{o}(n^{-1})
    \end{split}
\end{equation*}
for $c_{\alpha}\in (0,2/3)$. Thus, following a similar argument above for $\mathbb{E}T_3^2=\mathrm{O}(n^{-1})$, we get
\begin{equation*}
    \mathbb{E}a_k^2\lesssim n^{c_{\alpha}-1}(1-y_p+\epsilon_n r_p(\epsilon_n))+n^{-1/3-3c_{\alpha}}\log^4 n+\mathrm{O}(n^{-1})\lesssim n^{c_{\alpha}-1}\epsilon_n r_p(\epsilon_n)+\mathrm{O}(n^{-1}),
\end{equation*}
if we choose $c_{\alpha}\in (2/9,1/3)$ since $1-y_p=(n-p)/n\le n^{-1/3}$ for $n-n^{2/3}\le p\le n$, which together with \eqref{eq_absbeps} further implies
\begin{equation*}
    \begin{split}
        |b|\lesssim &\epsilon_n^{-1}[n^{c_{\alpha}-1}\epsilon_n r_p(\epsilon_n)+\mathrm{O}(n^{-1})](1-y_p+\epsilon_n+y_p\epsilon_n r_p(\epsilon_n))^{-2}\\
        \le &n^{c_{\alpha}-1}r_p(\epsilon_n)(1-y_p+\epsilon_n+y_p\epsilon_n r_p(\epsilon_n))^{-2}+\mathrm{O}(n^{-1}\epsilon_n^{-1}(1-y_p+\epsilon_n+y_p\epsilon_n r_p(\epsilon_n))^{-2})
    \end{split}
\end{equation*}
since $1-y_p+\epsilon_n+y_p\epsilon_n r_p(\epsilon_n)\ge \epsilon_n$. Therefore, setting $c_{\alpha}=2/9+\epsilon_0$ with a small enough constant $0<\epsilon_0<1/9$, for $n^{-2/3}\le \epsilon_n\ll n^{-4/9}$, by \eqref{eq_rpeps}, we have $|\epsilon_n b|\lesssim \mathrm{o}(1)$, and thus we get
\begin{equation*}
\begin{split}
r_p(\epsilon_n)=&2\big(\sqrt{(1+\epsilon_n-y_p )^2+4y_p\epsilon_n}+(1+\epsilon_n-y_p)\big)^{-1}+\mathrm{O}(|b|)\\
=&2\big(\sqrt{(1+\epsilon_n-y_p )^2+4y_p\epsilon_n}+(1+\epsilon_n-y_p)\big)^{-1}+\mathrm{O}_{\prec}\big(\frac{n^{-1}}{((1-y_p)^2+\epsilon_n)(1-y_p+\epsilon_n+y_p\epsilon_n r_p(\epsilon_n))^2}\big)\\
=&2\big(\sqrt{(1+\epsilon_n-y_p )^2+4y_p\epsilon_n}+(1+\epsilon_n-y_p)\big)^{-1}+\mathrm{O}_{\prec}\big(n^{-1}\epsilon_n^{-2}\big).
\end{split}
\end{equation*}
Thus, to make the error term negligible, choosing $n^{-2/3}\le \epsilon_n\ll n^{-4/9}$ gives
\begin{equation*}
r_p(\epsilon_n)= 2\big(\sqrt{(1+\epsilon_n-y_p )^2+4y_p\epsilon_n}+(1+\epsilon_n-y_p)\big)^{-1} +\mathrm{O}_{\prec}\big(n^{-1}\epsilon_n^{-2}\big)\asymp \epsilon_n^{-1/2}
\end{equation*}
for all $n-n^{2/3}\le p\le n$ as desired.

Finally, similarly to the estimate towards $\mathbb{E}T_3^2$, we have
\begin{equation*}
\operatorname{Var}(n^{-1}\operatorname{tr}G(\epsilon_n))=\mathrm{O}(n^{-1}\epsilon_n^{-2})
\end{equation*}
for all $\epsilon_n\gg n^{-1}$ as desired.
\end{proof}

\section{Off-diagonal part of a quadratic form}\label{sec-offdiag}
\subsection{Proof of Lemma \ref{lemma_Qbounds}}
\begin{proof}
    Recalling the identity $\mathbf{Q}_i\mathbf{Q}_i^{\top}=\mathbf{Q}_i/(n-i)$ and $\operatorname{tr}(\mathbf{Q}_i)=1$, we have
    \begin{equation*}
        \sum_{l}q_{i,kl}^2=\frac{q_{i,kk}}{n-i}~\text{and}~\sum_{k\ne l}q_{i,kl}^2=\operatorname{tr}(\mathbf{Q}_i\mathbf{Q}_i^{\top})-S_{2}^{(i)}=\frac{1}{(n-i)}-S_{2}^{(i)}\le \frac{i}{n(n-i)},
    \end{equation*}
    by $S_{2}^{(i)}\ge n^{-1}$ in Lemma \ref{lemma_diagonals}. It is obvious that $S_{2}^{(i)}\le 1/(n-i)$ since $\sum_{k\ne l}q_{i,kl}^2\ge 0$. For the first part of \eqref{eq_boundq1}, by the Cauchy-Schwarz inequality, one has
    \begin{equation*}
        \sum_{l}|q_{i,kl}|\le \big(n\sum_{l}q_{i,kl}^2\big)^{1/2}\le \big(\frac{nq_{i,kk}}{n-i}\big)^{1/2}\le \frac{n^{1/2}}{(n-i)}.
    \end{equation*}
    For the rest parts of \eqref{eq_boundq1}, we have
    \begin{equation*}
        \sum_{k\ne l}|q_{i,kl}|\le \big(n^2\sum_{k\ne l}q_{i,kl}^2\big)^{1/2}\le \frac{n}{(n-i)^{1/2}}
    \end{equation*}
    and
    \begin{equation*}
    	\sum_{l\ne s}|q_{i,kl}q_{i,ks}|\lesssim (\sum_{l}q_{i,kl})^2 \le n\sum_{l}q_{i,kl}^2\le \frac{nq_{i,kk}}{n-i}\le \frac{n}{(n-i)^2}.
    \end{equation*}
    For the first part of \eqref{eq_boundq2}, by the Cauchy-Schwarz inequality and \eqref{eq_boundq1}, we have
    \begin{equation*}
        \begin{split}
            \sum_{k\ne l\ne s}|q_{i,kl}q_{i,ks}|\le \sum_{k}(\sum_{l}q_{i,kl})^2 \le \sum_{k}\frac{nq_{i,kk}}{n-i}=\frac{n}{n-i}.
        \end{split}
    \end{equation*}
    Now we turn to the second part of \eqref{eq_boundq2}. Applying the Cauchy-Schwarz inequality, it is straightforward that
    \begin{equation*}
      \begin{split}
      \sum_{k\ne l\ne s}|q_{i,kl}q_{i,ks}q_{i,ls}|\le& \sum_{k\ne l}|q_{i,kl}|\sum_{s}|q_{i,ks}q_{i,ls}|
      \le \sum_{k\ne l}|q_{i,kl}|\big(\sum_{s}q_{i,ks}^2\sum_{s}q_{i,ls}^2\big)^{1/2}\\
      \lesssim& \sum_{k}\big(\frac{nq_{i,kk}}{n-i}\big)^{1/2}\big(\frac{q_{i,kk}q_{i,ll}}{(n-i)^2}\big)^{1/2}\le \frac{\sum_{k}n^{1/2}q_{i,kk}}{(n-i)^2}\le \frac{n^{1/2}}{(n-i)^2},
      \end{split}
    \end{equation*}
    where in the last line we used $q_{i,ll}\le 1/(n-i)$ and $\sum_{k}q_{i,kk}=1$.
    Similarly for \eqref{eq_boundq3}, applying the Cauchy-Schwarz inequality gives
    \begin{equation*}
        \begin{split}
            \sum_{k\ne l\ne s\ne m}|q_{i,kl}q_{i,ks}q_{i,lm}|\le&\sum_{k\ne s\ne m}|q_{i,ks}|\sum_{l}|q_{i,kl}q_{i,lm}|\\
            \le &\sum_{k\ne m}\big(n\sum_{s}q_{i,ks}^2\big)^{1/2}\big(\sum_{l}q_{i,kl}^2\sum_{l}q_{i,lm}^2\big)^{1/2}
            \le \sum_{k\ne m}\big(\frac{nq_{i,kk}}{n-i}\frac{q_{i,kk}q_{i,mm}}{(n-i)^2}\big)^{1/2}\\
            \le &\frac{n^{1/2}}{(n-i)^{3/2}}\big(\sum_{k}q_{i,kk}\big)\big(n\sum_{m}q_{i,mm}\big)^{1/2}
            \le  \frac{n}{(n-i)^{3/2}}.
        \end{split}
    \end{equation*}
    Similarly, we have
    \begin{equation*}
        \begin{split}
            \sum_{k\ne l\ne s\ne m}|q_{i,kl}q_{i,ks}q_{i,km}|\le \sum_{k}(\sum_{l}|q_{i,kl}|)^3\le \big(\frac{n}{n-i}\big)^{3/2}\sum_{k}q_{i,kk}^{3/2}\le  \frac{n^{3/2}}{(n-i)^2}
        \end{split}
    \end{equation*}
    by the Cauchy-Schwarz inequality and $\sum_{k}q_{i,kk}^{3/2}\le \sum_{k}q_{i,kk}/(n-i)^{1/2}=1/(n-i)^{1/2}$.
\end{proof}

\subsection{Fourth moments of off-diagonal parts}
\begin{lemma}\label{lemma_offdiag}
    Let $\alpha\in [3,4)$. We have for $n$ sufficiently large,
    \begin{equation}\label{eq_offdiag4}
        n^4\mathbb{E}\big[\big(\sum_{k\ne l}q_{i,kl}(Y_{i+1,k}Y_{i+1,l}-\beta_{1,1})\big)^4\big]\lesssim \frac{[\mathcal{L}(n^{1/2})]^2n^{4-\alpha}}{(n-i)^3}+\frac{1}{(n-i)^2}, ~i=0,1,\ldots,p-s_1-1,
    \end{equation}
    and
    \begin{equation}\label{eq_offdiag5}
        n^4\mathbb{E}\big[\big(\sum_{k\ne l}q_{i,kl}Y_{i+1,k}Y_{i+1,l}\big)^4\big]\lesssim \frac{1}{(n-i)^2}, ~i=p-s_1,\ldots,p-1.
    \end{equation}
\end{lemma}
\begin{proof}
Consider $0\le i\le p-s_2-1$ first. Denote $T_1=\sum_{k\ne l}q_{i,kl}Y_{i+1,k}Y_{i+1,l}$ and $T_2=\sum_{k\ne l}q_{i,kl}\beta_{1,1}$. Noting the elementary identity $(a-b)^4=a^4-4a^3b+6a^2b^2-4ab^3+b^4$ and $\mathbb{E}(T_1|\mathcal{F}_i)=T_2$, we have
    \begin{equation*}
        \begin{split}
            \mathbb{E}[(T_1-T_2)^4|\mathcal{F}_i]=\mathbb{E}(T_1^4|\mathcal{F}_i)-4T_2\mathbb{E}(T_1^3|\mathcal{F}_i)+6T_2^2\mathbb{E}(T_1^2|\mathcal{F}_i)-3T_2^4=:\mathrm{V}_1+\mathrm{V}_2+\mathrm{V}_3+\mathrm{V}_4,
        \end{split}
    \end{equation*}
    where $\mathbb{E}(\cdot|\mathcal{F}_i)$ denotes the conditional expectation given $X_{1},\ldots,X_{i}$. In what follows, we analyze the above terms one by one.
    Invoking \eqref{eq_boundq1}, we have
    \begin{equation}\label{eq_off_q1}
        T_2=\beta_{1,1}\sum_{k\ne l}q_{i,kl}\le \beta_{1,1}n(\sum_{k\ne l}q_{i,kl}^2)^{1/2}\le \beta_{1,1}n/(n-i)^{1/2}.
    \end{equation}
    By Lemma \ref{lemma_Qbounds}, we have
    \begin{equation}\label{eq_off_q2}
        \begin{split}
            \mathbb{E}(T_1^2|\mathcal{F}_i)\lesssim &\sum_{k\ne l}q_{i,kl}^2n^2\beta_{2,2}+\sum_{k\ne l\ne s}q_{i,kl}q_{i,ks}n^2\beta_{2,1,1}+\sum_{k\ne l\ne s\ne m}q_{i,kl}q_{i,sm}n^2\beta_{1,1,1,1}\\
            \lesssim & \frac{n^2\beta_{2,2}}{n-i}+\frac{n^3\beta_{2,1,1}}{n-i}+\frac{n^4\beta_{1,1,1,1}}{n-i},
        \end{split}
    \end{equation}
    where the last term follows from $\sum_{k\ne l\ne s\ne m}q_{i,kl}q_{i,sm}\lesssim (\sum_{k\ne l}q_{i,kl})^2\le n^2/(n-i)$ by \eqref{eq_boundq1}. As an immediate consequence, we get
    \begin{equation*}
    	\begin{split}
    	n^4(\mathrm{V}_3+\mathrm{V}_4)\lesssim & n^4\big(\frac{n\beta_{1,1}}{(n-i)^{1/2}}\big)^4+n^4\big(\frac{n\beta_{1,1}}{(n-i)^{1/2}}\big)^2\big(\frac{n^2\beta_{2,2}}{n-i}+\frac{n^3\beta_{2,1,1}}{n-i}+\frac{n^4\beta_{1,1,1,1}}{n-i}\big)\\
    	\le&\mathrm{o}\big(\frac{n^{-4}}{(n-i)^2}+\frac{1}{(n-i)^2}\big)=\mathrm{o}\big(\frac{1}{(n-i)^2}\big),
    	\end{split}
    \end{equation*}
    where we used $\beta_{1,1}=\mathrm{o}(n^{-3}), \beta_{1,1,1,1}=\mathrm{o}(n^{-4}),\beta_{2,1,1}=\mathrm{o}(n^{-4})$ and $\beta_{2,2}=\mathrm{O}(n^{-2})$ by Lemma \ref{lemma_odd_moment}.

    Now we turn to $\mathbb{E}(T_1^3|\mathcal{F}_i)$. By direct calculation,
    \begin{equation*}
        \begin{split}
            \mathbb{E}(T_1^3|\mathcal{F}_i)\lesssim &\sum_{k\ne l}q_{i,kl}^3\beta_{3,3}+\sum_{k\ne l\ne s}\big(q_{i,kl}q_{i,ks}q_{i,ls}\beta_{2,2,2}+q_{i,kl}^2q_{i,ks}\beta_{3,2,1}\big)\\
            &+\sum_{k\ne l\ne s\ne m}\big(q_{i,kl}q_{i,ks}q_{i,km}\beta_{3,1,1,1}+(q_{i,kl}^2q_{i,sm}+q_{i,kl}q_{i,ks}q_{i,lm})\beta_{2,2,1,1}\big)\\
            &+\sum_{k\ne l\ne s\ne m\ne r}q_{i,kl}q_{i,ks}q_{i,mr}\beta_{2,1,1,1,1}+\sum_{k\ne l\ne s\ne m\ne r\ne t}q_{i,kl}q_{i,sm}q_{i,rt}\beta_{1,1,1,1,1,1}\\
            \lesssim & \frac{\beta_{3,3}}{(n-i)^2}+\frac{n\beta_{2,2,2}}{(n-i)^2}+\frac{n\beta_{3,2,1}}{(n-i)^2}+\frac{n^{3/2}\beta_{3,1,1,1}}{(n-i)^2}+\frac{n\beta_{2,2,1,1}}{(n-i)^{3/2}}+\frac{n^2\beta_{2,1,1,1,1}}{(n-i)^{3/2}}+\frac{n^3\beta_{1,1,1,1,1,1}}{(n-i)^{3/2}},
        \end{split}
    \end{equation*}
    where we used Lemma \ref{lemma_Qbounds} and the following estimates
\begin{equation*}
    \begin{split}
        &\sum_{k\ne l}|q_{i,kl}^3|\le \frac{\sum_{k\ne l}q_{i,kl}^2}{n-i}\le \frac{1}{(n-i)^2},~
        \sum_{k\ne l\ne s}|q_{i,kl}^2q_{i,ks}|\le\frac{\sum_{k\ne l\ne s}|q_{i,kl}q_{i,ks}|}{n-i} \frac{n}{(n-i)^2},\\
        &\sum_{k\ne l\ne s\ne m}|q_{i,kl}^2q_{i,sm}+q_{i,kl}q_{i,ks}q_{i,lm}|\le \sum_{k\ne l}q_{i,kl}^2\sum_{s\ne m}|q_{i,sm}|+\frac{n}{(n-i)^{3/2}}\le \frac{n}{(n-i)^{3/2}},\\
        &\sum_{k\ne l\ne s\ne m\ne r}|q_{i,kl}q_{i,ks}q_{i,mr}|\le \sum_{k\ne l\ne s}|q_{i,kl}q_{i,ks}|\sum_{m\ne r}|q_{i,mr}|\le \frac{n^2}{(n-i)^{3/2}},\\
        &\sum_{k\ne l\ne s\ne m\ne r\ne t}|q_{i,kl}q_{i,sm}q_{i,rt}|\le \big(\sum_{k\ne l}|q_{i,kl}|\big)^3\le \frac{n^3}{(n-i)^{3/2}}.
    \end{split}
\end{equation*}
Therefore, combining the bounds of $\beta$'s in Lemma \ref{lem_moment_rates} and Lemma \ref{lemma_odd_moment}, we have
\begin{equation*}
\begin{split}
&\beta_{3,3}\le \beta_{1,1}=\mathrm{o}(n^{-3}), ~\beta_{3,2,1}\lesssim n^{-1}\beta_{3,1}=\mathrm{o}(n^{-4}),\beta_{3,1,1,1}\le \beta_{1,1,1,1}=\mathrm{o}(n^{-5}),\\
&\beta_{2,2,1,1}\lesssim n^{-2}\beta_{1,1}=\mathrm{o}(n^{-5}),~\beta_{2,1,1,1,1}\lesssim n^{-1}\beta_{1,1,1,1}=\mathrm{o}(n^{-6}),~\beta_{1,1,1,1,1,1,1}=\mathrm{o}(n^{-7}),
\end{split}
\end{equation*}
and yields that
\begin{equation}\label{eq_off_q3}
    \mathbb{E}(T_1^3|\mathcal{F}_i)\lesssim \frac{n^{-2}}{(n-i)^2}.
\end{equation}
which further gives
\begin{equation*}
	n^4\mathrm{V}_2\le n^4\frac{n^{-2}}{(n-i)^2}\frac{n\beta_{1,1}}{(n-i)^{1/2}}=\mathrm{o}\big(\frac{1}{(n-i)^{5/2}}\big).
\end{equation*}
We now turn to the first term, which is more involved. Similarly, we have
\begin{equation*}
\begin{split}
&\mathbb{E}(T_1^4|\mathcal{F}_i)\\
\lesssim &\sum_{k\ne l}\big(q_{i,kl}^4\beta_{4,4}\big)+\sum_{k\ne l\ne s}\big(q_{i,kl}^2q_{i,ks}^2\beta_{4,2,2}+q_{i,kl}^3q_{i,ks}\beta_{4,3,1}+q_{i,kl}^2q_{i,ks}q_{i,ls}\beta_{3,3,2}\big)\\
&+\sum_{k\ne l\ne s\ne m}\big(q_{i,kl}^2q_{i,ks}q_{i,km}\beta_{4,2,1,1}+(q_{i,kl}^3q_{i,sm}+q_{i,kl}^2q_{i,ks}q_{i,lm})\beta_{3,3,1,1}\\
&\quad+(q_{i,kl}^2q_{i,ks}q_{i,sm}+q_{i,kl}q_{i,ks}q_{i,ls}q_{i,km})\beta_{3,2,2,1}+q_{i,kl}^2q_{i,sm}^2\beta_{2,2,2,2}\big)\\
&+\sum_{k\ne l\ne s\ne m\ne r}\big(q_{i,kl}q_{i,ks}q_{i,km}q_{i,kr}\beta_{4,1,1,1,1}+(q_{i,kl}^2q_{i,ks}q_{i,mr}+q_{i,kl}q_{i,ks}q_{i,km}q_{i,lr})\beta_{3,2,1,1,1}\\
&\quad+(q_{i,kl}^2q_{i,sm}q_{i,sr}+q_{i,kl}q_{i,ks}q_{i,ls}q_{i,mr})\beta_{2,2,2,1,1}\big)\\
&+\sum_{k\ne l\ne s\ne m\ne r\ne t}\big(q_{i,kl}q_{i,ks}q_{i,km}q_{i,rt}\beta_{3,1,1,1,1,1}+(q_{i,kl}^2q_{i,sm}q_{i,rt}+q_{i,kl}q_{i,ks}q_{i,mr}q_{i,mt})\beta_{2,2,1,1,1,1}\big)\\
&+\sum_{k\ne l\ne s\ne m\ne r\ne t\ne u}q_{i,kl}q_{i,ks}q_{i,mr}q_{i,tu}\beta_{2,1,1,1,1,1,1}+\sum_{k\ne l\ne s\ne m\ne r\ne t\ne u\ne v}q_{i,kl}q_{i,sm}q_{i,rt}q_{i,uv}\beta_{1,1,1,1,1,1,1,1}\\
=&:\theta_{i,n}^{(2)}+\theta_{i,n}^{(3)}+\theta_{i,n}^{(4)}+\theta_{i,n}^{(5)}+\theta_{i,n}^{(6)}+\theta_{i,n}^{(7)}+\theta_{i,n}^{(8)}.
\end{split}
\end{equation*}
Invoke the bounds of the entries of $\mathbf{Q}_i$ in Lemma \ref{lemma_Qbounds}. Now we estimate the terms one by one. Firstly, we have
\begin{equation*}
    \begin{split}
        \theta_{i,n}^{(2)}=\sum_{k\ne l}\big(q_{i,kl}^4\beta_{4,4}\big)\lesssim \beta_{4,4}\sum_{k}(\sum_{ l}q_{i,kl}^2)^2
        \lesssim \beta_{4,4}\sum_{k}q_{i,kk}^2\frac{1}{(n-i)^2}= \beta_{4,4}\frac{S_{2}^{(i)}}{(n-i)^2}
    \end{split}
\end{equation*}
and
\begin{equation*}
    \begin{split}
        \theta_{i,n}^{(3)}\lesssim &\sum_{k}(\sum_{l}q_{i,kl}^2)^2\beta_{4,2,2}+\sum_{k\ne s}q_{i,ks}\sum_{l}q_{i,kl}^2\frac{1}{n-i}\beta_{4,3,1}+\sum_{k\ne s}q_{i,ks}\sum_{l}q_{i,kl}^2\frac{1}{n-i}\beta_{3,3,2}\\
        \lesssim &\beta_{4,2,2}\sum_{k}q_{i,kk}^2\frac{1}{(n-i)^2} +\sum_{k\ne s}q_{i,ks}q_{i,kk}\frac{1}{(n-i)^2}(\beta_{4,3,1}+\beta_{3,3,2})\\
        \lesssim & \beta_{4,2,2}\frac{S_{2}^{(i)}}{(n-i)^2}+(\beta_{4,3,1}+\beta_{3,3,2})\frac{n^{1/2}}{(n-i)^3}
    \end{split}
\end{equation*}
by $\sum_{l}q_{i,kl}^2=q_{i,kk}/(n-i)$ and \eqref{eq_boundq1} in Lemma \ref{lemma_Qbounds}.
 Similarly, we have
\begin{equation*}
	\begin{split}
 \theta_{i,n}^{(4)}\lesssim &\beta_{4,2,1,1}\sum_{k\ne l}q_{i,kl}^2\sum_{s\ne m}q_{i,ks}q_{i,km}+\beta_{3,3,1,1}\sum_{k\ne l}q_{i,kl}^2\big(\frac{1}{n-i}\sum_{s\ne m}q_{i,sm}+(\sum_{s}q_{i,ks})(\sum_{m}q_{i,lm})\big)\\
 &+\beta_{3,2,2,1}\big(\sum_{k\ne l}q_{i,kl}^2\sum_{s}q_{i,ks}\sum_{m}q_{i,sm}+\sum_{k\ne l\ne m}q_{i,kl}q_{i,km}\sum_{s}q_{i,ks}q_{i,ls}\big)+\beta_{2,2,2,2}(\sum_{k\ne l}q_{i,kl}^2)^2\\
 \lesssim & \frac{n\beta_{4,2,1,1}}{(n-i)^3}+\frac{n\beta_{3,3,1,1}}{(n-i)^{2}}+\frac{n\beta_{3,2,2,1}}{(n-i)^3}+\frac{\beta_{2,2,2,2}}{(n-i)^2},
	\end{split}
\end{equation*}
where we used Lemma \ref{lemma_Qbounds} and
\begin{equation*}
    \begin{split}
        &\sum_{k\ne l\ne m}|q_{i,kl}q_{i,km}|\sum_{s}|q_{i,ks}q_{i,ls}|\le \sum_{k\ne l\ne m}|q_{i,kl}q_{i,km}|(\sum_{s}q_{i,ks}^2\sum_{s}q_{i,ls}^2)^{1/2}\\
        \le & \sum_{k\ne l\ne m}|q_{i,kl}q_{i,km}|\frac{(q_{i,kk}q_{i,ll})^{1/2}}{(n-i)}\le \sum_{k\ne l\ne m}|q_{i,kl}q_{i,km}|\frac{1}{(n-i)^2}\le n/(n-i)^3,
    \end{split}
\end{equation*}
where in the last line we used \eqref{eq_boundq2}. Analogously, one has
\begin{equation*}
    \begin{split}
      \theta_{i,n}^{(5)}\lesssim &\beta_{4,1,1,1,1}\sum_{k}\sum_{l\ne s}q_{i,kl}q_{i,ks}\sum_{m\ne r}q_{i,km}q_{i,kr}+\beta_{3,2,1,1,1}\big(\sum_{k\ne l\ne s}q_{i,kl}q_{i,ks}\frac{1}{n-i}\sum_{m\ne r}q_{i,mr}\\
      &+\sum_{k\ne r}\sum_{l}q_{i,kl}q_{i,lr}\sum_{s\ne m}q_{i,ks}q_{i,km}\big)\\
 &+\beta_{2,2,2,1,1}\big(\sum_{k\ne l}q_{i,kl}^2\sum_{s\ne m\ne r}q_{i,sm}q_{i,sr}+\sum_{k\ne l\ne s}q_{i,kl}q_{i,ks}q_{i,ls}\sum_{m\ne r}q_{i,mr}\big)\\
 \lesssim &\beta_{4,1,1,1,1}\sum_{k}\frac{n^2q_{i,kk}^2}{(n-i)^2}+\beta_{3,2,1,1,1}\big(\frac{n^2}{(n-i)^{5/2}}+\frac{n^2}{(n-i)^3}\big)+\beta_{2,2,2,1,1}\big(\frac{n}{(n-i)^2}+\frac{n^2}{(n-i)^{5/2}}\big)\\
 \lesssim & \frac{n\beta_{4,1,1,1,1}n^2S_{2}^{(i)}}{(n-i)^2}+\frac{n^2\beta_{3,2,1,1,1}}{(n-i)^{5/2}}+\frac{n^2\beta_{2,2,2,1,1}}{(n-i)^{5/2}}
    \end{split}
\end{equation*}
by Lemma \ref{lemma_Qbounds}.
For $\theta_{i,n}^{(6)}$, we have
\begin{equation*}
    \begin{split}
        \theta_{i,n}^{(6)}=&\beta_{3,1,1,1,1,1}\sum_{k\ne l\ne s\ne m}q_{i,kl}q_{i,ks}q_{i,km}\sum_{r\ne t}q_{i,rt}\\
        &+\beta_{2,2,1,1,1,1}\big(\sum_{k\ne l}q_{i,kl}^2(\sum_{s\ne m}q_{i,sm})(\sum_{r\ne t}q_{i,rt})+(\sum_{k\ne l\ne s}q_{i,kl}q_{i,ks})^2\big)\\
        \lesssim & \beta_{3,1,1,1,1,1}\frac{n^3}{(n-i)^{5/2}}+\beta_{2,2,1,1,1,1}\frac{n^2}{(n-i)^2}.
    \end{split}
\end{equation*}
For the last two terms, one has
\begin{equation*}
    \begin{split}
        \theta_{i,n}^{(7)}+\theta_{i,n}^{(8)}\lesssim &\beta_{2,1,1,1,1,1,1}\sum_{k\ne l\ne s}q_{i,kl}q_{i,ks}(\sum_{m\ne r}q_{i,mr})^2+\beta_{1,1,1,1,1,1,1,1}(\sum_{k\ne l}q_{i,kl})^4\\
        \lesssim &\beta_{2,1,1,1,1,1,1} \frac{n^3}{(n-i)^2}+\beta_{1,1,1,1,1,1,1,1}\frac{n^4}{(n-i)^2}.
    \end{split}
\end{equation*}
Therefore, similarly to $\mathrm{V}_2$, by Lemma \ref{lem_moment_rates} and Lemma \ref{lemma_odd_moment}, we get
\begin{equation*}
	n^4\mathrm{V}_1\lesssim n^4\beta_{4,4}\frac{S_2^{(i)}}{(n-i)^2}+\frac{1}{(n-i)^2}\lesssim \frac{[\mathcal{L}(n^{1/2})]^2n^{4-\alpha}}{(n-i)^3}+\frac{1}{(n-i)^2},
\end{equation*}
where we used the fact that $S_{2}^{(i)}\le (n-i)^{-1}$. Thereafter, we conclude that, for $\alpha\ge 3$,
\begin{equation*}
	n^4\mathbb{E}\big[\big(\sum_{k\ne l}q_{i,kl}(Y_{i+1,k}Y_{i+1,l}-\beta_{1,1})\big)^4\big]\lesssim \frac{[\mathcal{L}(n^{1/2})]^2n^{4-\alpha}}{(n-i)^3}+\frac{1}{(n-i)^2}, ~i=0,1,\ldots,p-s_1-1,
\end{equation*}
as desired.

To this end, invoking the Gaussian replacement for $p-s_1\le i\le p-1$, which suggests that $\kappa_{1,1}=0$ and all the odd terms vanish, we get
\begin{equation*}
    \begin{split}
        \mathbb{E}\big[\big(\sum_{k\ne l}q_{i,kl}Y_{i+1,k}Y_{i+1,l}\big)^4\mid \mathcal{F}_i\big]\lesssim &\sum_{k\ne l}q_{i,kl}^4\kappa_{4,4}+\sum_{k\ne l\ne s}q_{i,kl}^2q_{i,ks}^2\kappa_{4,2,2}+\sum_{k\ne l\ne s\ne m}q_{i,kl}^2q_{i,sm}^2\kappa_{2,2,2,2}\\
\lesssim & \frac{\kappa_{4,4}\sum_{k\ne l}q_{i,kl}^2}{(n-i)^2}+\frac{\kappa_{4,2,2}\sum_{k}q_{i,kk}^2}{(n-i)^2}+\frac{\kappa_{2,2,2,2}}{(n-i)^2}
\lesssim  \frac{n^{-4}}{(n-i)^2},
    \end{split}
\end{equation*}
which further implies
\begin{equation*}
    n^4\mathbb{E}\big[\big(\sum_{k\ne l}q_{i,kl}Y_{i+1,k}Y_{i+1,l}\big)^4\big]\lesssim \frac{1}{(n-i)^2}, i=p-s_1,\ldots,p-1,
\end{equation*}
finishing the proof of this proposition.
\end{proof}

\bibliography{reference.bib}

\end{document}